\documentclass[reqno]{amsart}

\usepackage{import}

\usepackage[utf8]{inputenc}
\usepackage{amsthm}
\usepackage{amsfonts}
\usepackage{graphicx}
\usepackage{amsmath, amssymb}
\usepackage{hyperref,cleveref}
\usepackage{mathrsfs}
\usepackage{multirow}
\usepackage{centernot}
\usepackage{tikz}
\usepackage{tikz-cd}
\usepackage{color}
\usetikzlibrary{arrows}
\usepackage{pgf}
\usepackage{mathtools}
\usepackage{enumerate}
\usepackage[myheadings]{fullpage}
\usepackage{comment}

\usepackage{lineno}

\newtheorem{theorem}{Theorem}[section]

\newtheorem{lemma}[theorem]{Lemma}
\newtheorem{proposition}[theorem]{Proposition}
\newtheorem{corollary}[theorem]{Corollary}
\newtheorem{conjecture}[theorem]{Conjecture}

\theoremstyle{definition}
\newtheorem{definition}[theorem]{Definition}

\newtheorem{remark}[theorem]{Remark}

\newcommand{\old}[1]{}

\newcommand{\R}{\mathbb{R}}
\newcommand{\C}{\mathbb{C}}

\newcommand{\E}{\mathbb{E}}

\newcommand{\Y}{\mathbb{Y}}

\newcommand{\PP}{\mathbb{P}}
\newcommand{\Z}{\mathbb{Z}}

\newcommand{\cov}{\mathrm{Cov}}
\newcommand{\var}{\mathrm{Var}}
\newcommand{\e}{\varepsilon}

\newcommand{\bi}{\mathbf{i}}

\renewcommand{\vec}[1]{\boldsymbol{#1}}

\newcommand{\cC}{\mathcal{C}}
\newcommand{\cD}{\mathcal{D}}
\newcommand{\cE}{\mathcal{E}}
\newcommand{\cF}{\mathcal{F}}

\newcommand{\cH}{\mathcal{H}}

\newcommand{\cK}{\mathcal{K}}
\newcommand{\cL}{\mathcal{L}}
\newcommand{\cM}{\mathcal{M}}
\newcommand{\cN}{\mathcal{N}}

\newcommand{\cP}{\mathcal{P}}

\newcommand{\cS}{\mathcal{S}}
\newcommand{\cT}{\mathcal{T}}

\newcommand{\sL}{\mathscr{L}}

\newcommand{\sgn}{\mathrm{sgn}}

\makeatletter
\newcommand{\vast}{\bBigg@{4}}
\newcommand{\Vast}{\bBigg@{5}}

\makeatother

\newcommand{\wh}[1]{\widehat{#1}}
\newcommand{\wt}[1]{\widetilde{#1}}

\numberwithin{equation}{section}

\newcommand{\hloz}{\mathbin{\rotatebox[origin=c]{90}{$\Diamond$}}}
\newcommand{\rloz}{\mathbin{\rotatebox[origin=c]{33}{$\Diamond$}}}
\newcommand{\lloz}{\mathbin{\rotatebox[origin=c]{-33}{$\Diamond$}}}

\newcommand{\hs}{h_{\raisebox{-1pt}{\tiny$N$}}^{\raisebox{+1.5pt}{\tiny$\operatorname{Shift}$}}}
\newcommand{\hh}{h_{\raisebox{-1pt}{\tiny$N$}}}

\newcommand{\cpar}{%
  \raisebox{0.1em}{\rlap{\rotatebox{-45}{\rule[.10ex]{.4pt}{.5657em}}}%
  \kern.04em%
  \rlap{\kern.38em\raisebox{0.39em}{\rule{.4em}{.4pt}}}%
  \rule{.4em}{.4pt}\kern-.05em%
  \rotatebox{-45}{\rule[.1ex]{.4pt}{.5657em}}}}

\newcommand{\lpar}{%
  \raisebox{-0.1em}{\rlap{\rule[.05ex]{.4pt}{.45em}}%
  \kern-.0em%
  \rlap{\kern.0em\raisebox{0.45em}{\rotatebox{45}{\rule{.5657em}{.4pt}}}}%
  \rotatebox{45}{\rule{.5657em}{.4pt}} \kern-.37em%
  \raisebox{.4em}{\rule[.05ex]{.4pt}{.45em}}}}
  
\newcommand{\rpar}{%
  \raisebox{0.1em}{\rlap{\rule[.05ex]{.4pt}{.45em}}%
  \kern.0em%
  \rlap{\kern.0em\raisebox{0.45em}{\rule{.4em}{.4pt}}}%
  \rule{.4em}{.4pt}\kern-.03em%
  \rule[.05ex]{.4pt}{.45em}}}

\DeclarePairedDelimiter{\abs}{\lvert}{\rvert}

\usepackage{scalerel,stackengine}
\stackMath
\newcommand\reallywidehat[1]{%
\savestack{\tmpbox}{\stretchto{%
  \scaleto{%
    \scalerel*[\widthof{\ensuremath{#1}}]{\kern-.6pt\bigwedge\kern-.6pt}%
    {\rule[-\textheight/2]{1ex}{\textheight}}
  }{\textheight}%
}{0.5ex}}%
\stackon[1pt]{#1}{\tmpbox}%
}
\DeclareSymbolFont{bbold}{U}{bbold}{m}{n}
\DeclareSymbolFontAlphabet{\mathbbold}{bbold}
\newcommand{\bbone}{\mathbbold{1}}

\renewcommand{\Re}{\operatorname{Re}}
\renewcommand{\Im}{\text{Im}}

\newcommand{\floor}[1]{\lfloor #1 \rfloor}

\newcommand{\tth}{^{th}}

\renewcommand{\Pr}{\mathbb{P}}

\newcommand{\ttau}{\tilde{\tau}}
\newcommand{\tF}{F^*}
\newcommand{\tcF}{\widetilde{\mathcal{F}}}

\newcommand{\opvol}{{\operatorname{vol}}}

\newcommand{\mvol}{{\operatorname{vol}}}

\begin{document}
\title{
Lozenge tilings and the Gaussian free field on a cylinder
}
\author{Andrew Ahn}
\address{(Andrew Ahn) Columbia University, Department of Mathematics, 2990 Broadway, New York, New York, 10027}

 \author{Marianna Russkikh}
 \address{(Marianna Russkikh) Massachusetts Institute of Technology, Department of Mathematics, 77 Massachusetts Avenue, Cambridge, Massachusetts, 02139}
 
 \author{Roger Van Peski}
 \address{(Roger Van Peski) Massachusetts Institute of Technology, Department of Mathematics, 77 Massachusetts Avenue, Cambridge, Massachusetts, 02139}


\begin{abstract}
We use the periodic Schur process, introduced in~\cite{Bor07}, to study the random height function of lozenge tilings (equivalently, dimers) on an infinite cylinder distributed under two variants of the $q^{\mvol}$ measure. Under the first variant, corresponding to random cylindric partitions, the height function converges to a deterministic limit shape and fluctuations around it are given by the Gaussian free field in the conformal structure predicted by the Kenyon-Okounkov conjecture. Under the second variant, corresponding to an unrestricted dimer model on the cylinder, the fluctuations are given by the same Gaussian free field with an additional discrete Gaussian shift component. Fluctuations of the latter type have been previously conjectured for dimer models on planar domains with holes.
\end{abstract}

\maketitle

\tableofcontents

\section{Introduction}\label{sec:intro}

\subsection{Background} \label{sec:background}

Depending on one's viewpoint, an ordinary plane partition is either 
\begin{enumerate}
    \item an infinite array $\pi := (n_{i,j})_{i,j \in \Z_{\ge 1}}$ of nonnegative integers, only finitely many nonzero, such that $n_{i,j}$ is weakly decreasing in $i$ and $j$ (\Cref{Fig:lots_of_cubes} top left);
    \item a collection of a finite number of unit cubes stacked in a convex manner in the corner of an infinite room, where the numbers $n_{i,j}$ correspond to the height of the stacks (\Cref{Fig:lots_of_cubes} middle left);
    \item a tiling of the triangular lattice by lozenges $\hloz,\lloz,\rloz$ such that in each $120^\circ$ quadrant, all but finitely many lozenges are of one of the three types (also \Cref{Fig:lots_of_cubes} middle left). These tilings may equivalently be viewed as perfect matchings of the hexagonal lattice with vertices corresponding to the triangles in the triangular lattice, in which case the lozenges correspond to edges (dimers) of the matching.
\end{enumerate}

The second interpretation defines a stepped surface, which is specified by a \emph{height function} with values given by the array in the first one, and it is natural to study the height functions of random plane partitions. For many natural sequences of probability measures on plane partitions, the limits of these functions have been shown to converge to certain entropy-maximizers~\cite{CKP01,KOS06,KO07}. 
For some specific choices of boundary conditions the fluctuations have been shown to converge to the \emph{Gaussian free field}, a $2$ dimensional conformally invariant random field (see e.g.~\cite{werner2020lecture}), in a suitable conformal structure, see e.g.~\cite{Ken08,BF14,Pet15,BuG18,berestycki2019dimer,berestycki2020dimers,huang2020height}. Such convergence was conjectured to hold for a broad class of dimer models in~\cite{KO07}.

\begin{figure}[h]
\begin{center}
\begin{tabular}{c c c c}
\rotatebox{90}{$\quad$partition}
&\includegraphics[scale=0.22]{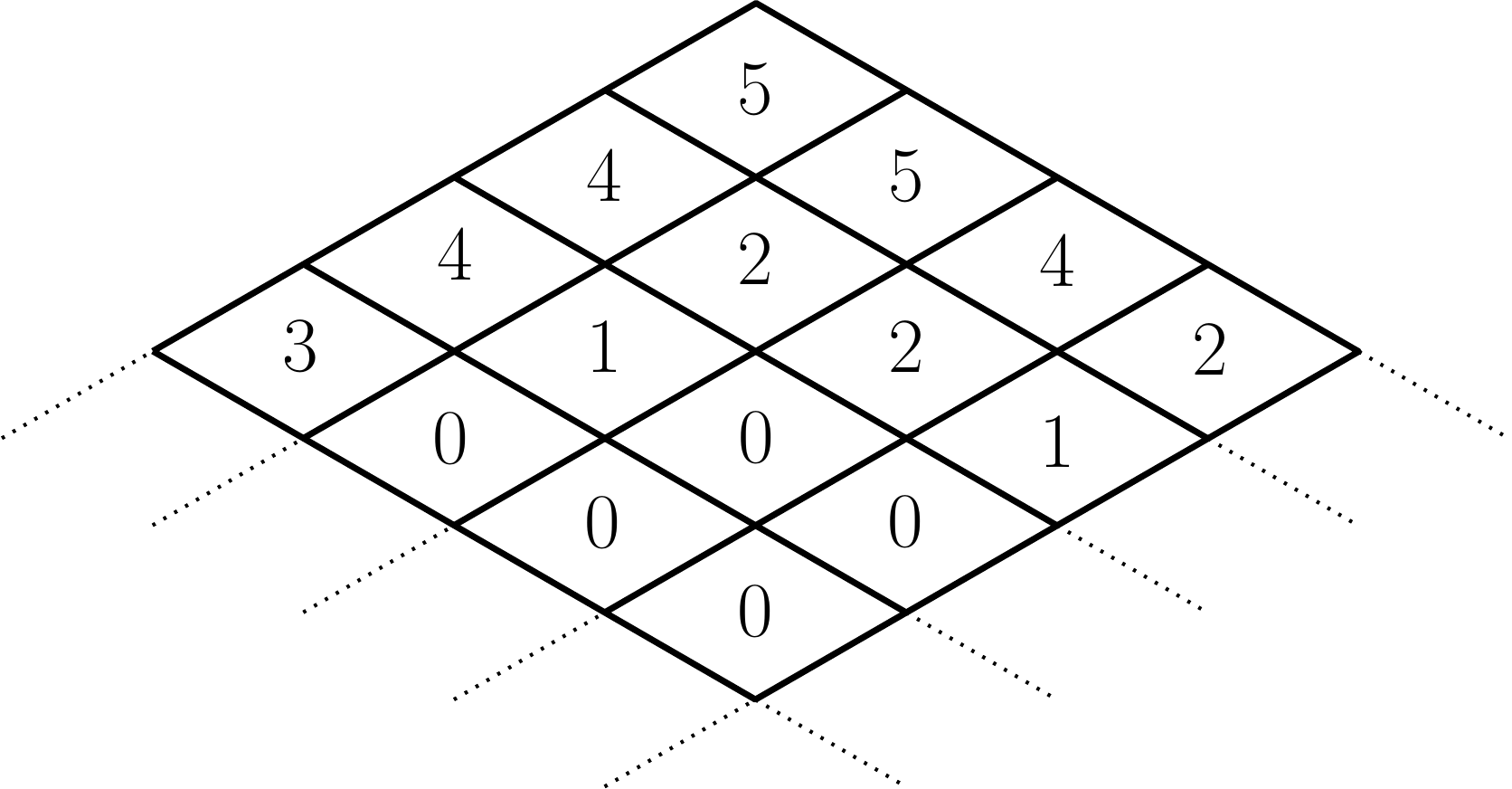}
$\quad\quad$
&\includegraphics[scale=0.22]{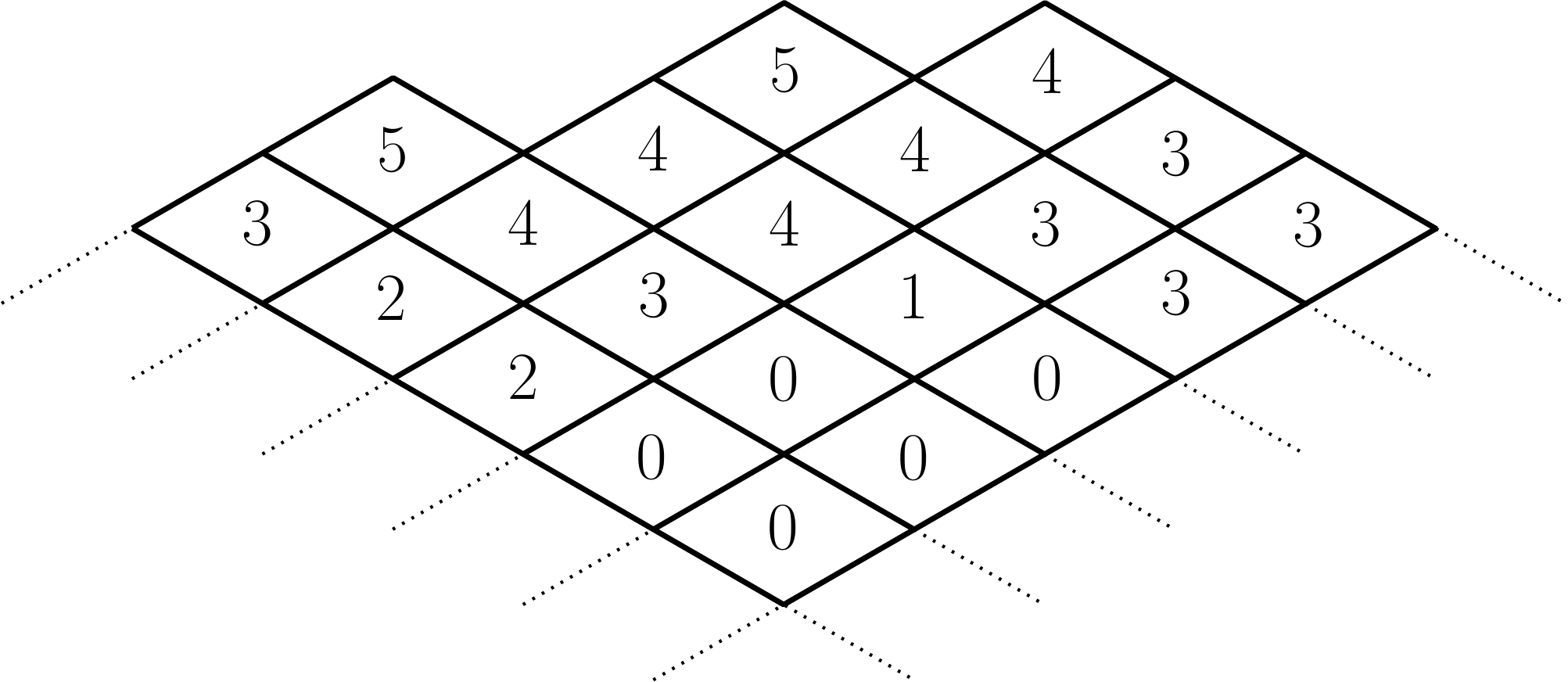}
$\quad\quad$
&\includegraphics[scale=0.22]{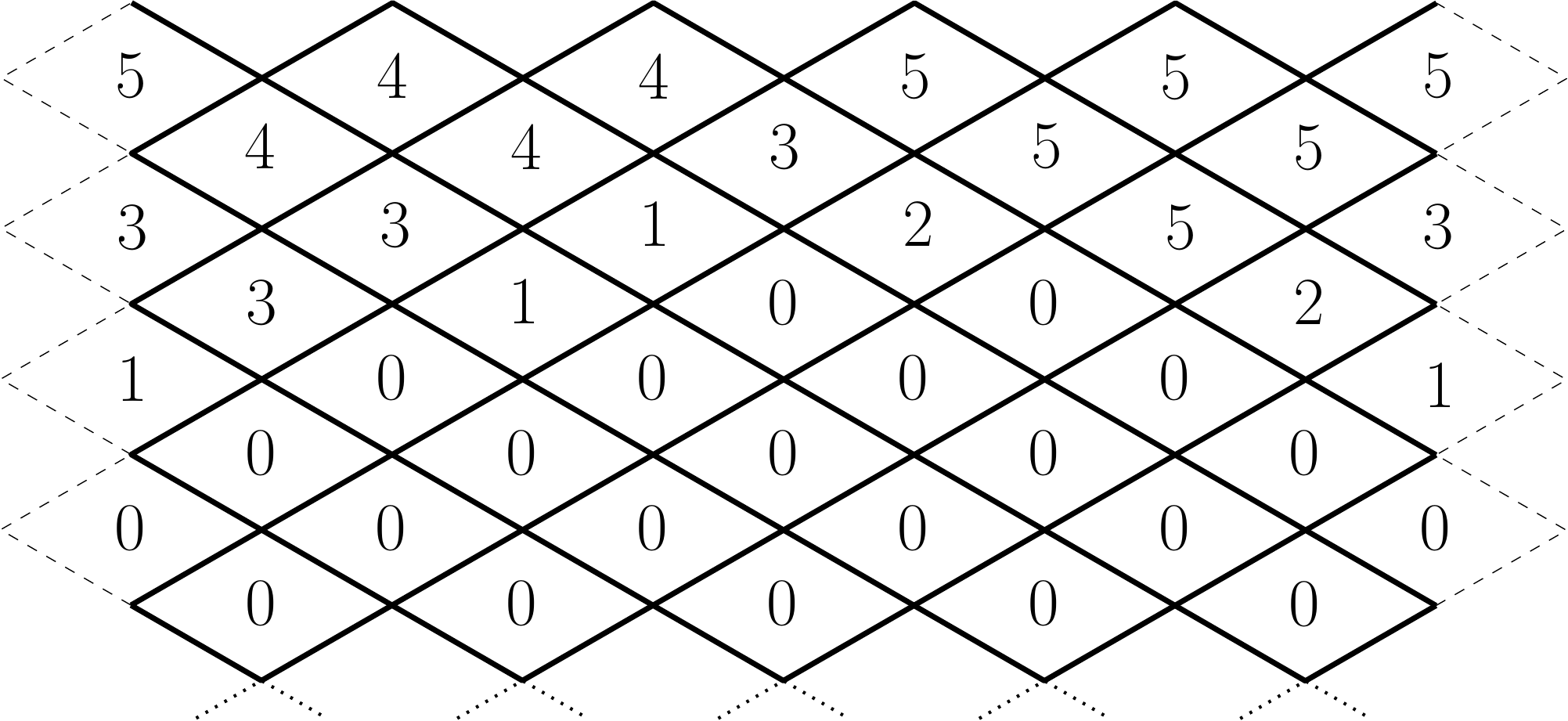}\\
\rotatebox{90}{$\quad\quad\quad\quad$lozenge tiling}
&\includegraphics[scale=0.22]{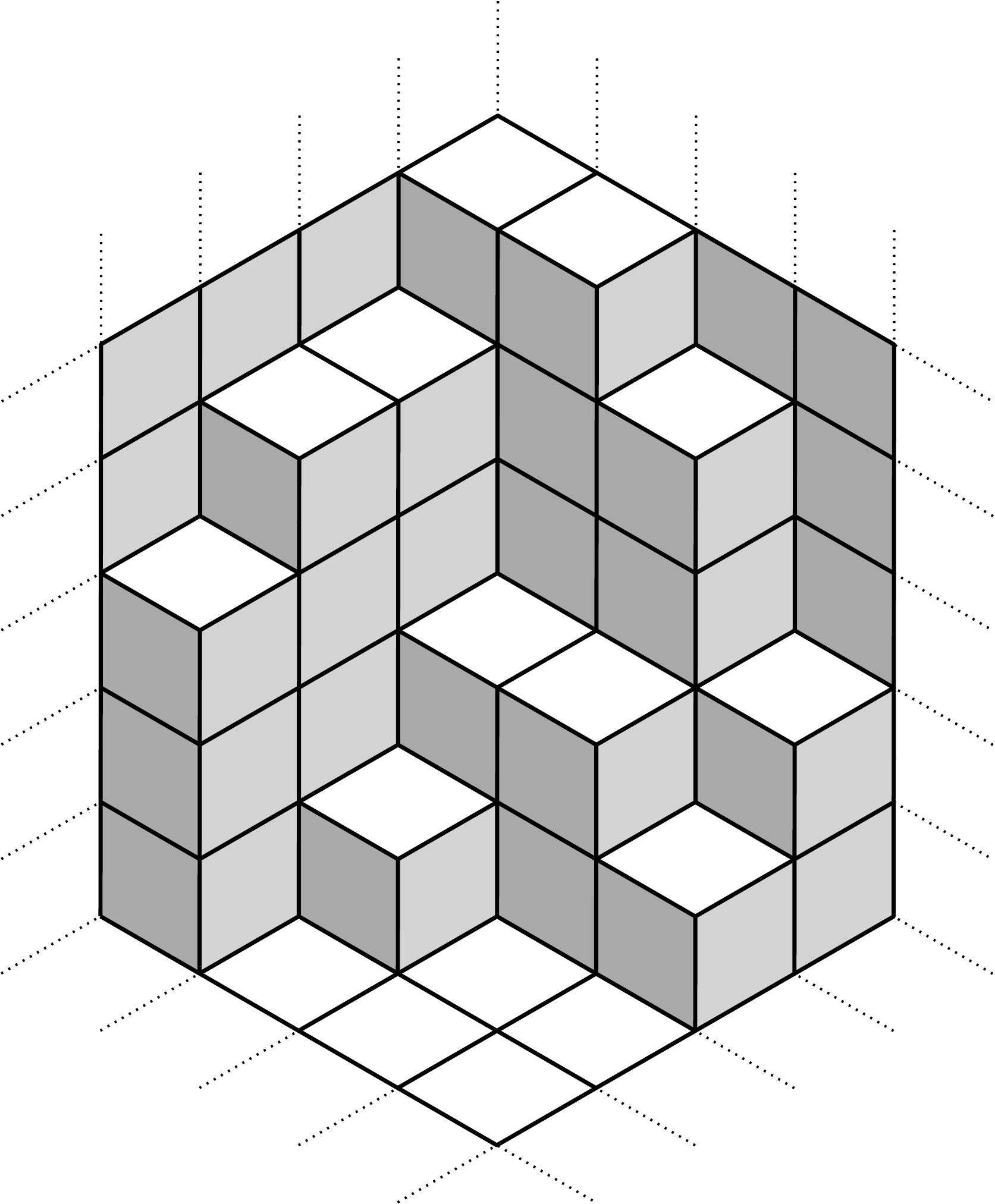}
$\quad\quad$
&\includegraphics[scale=0.22]{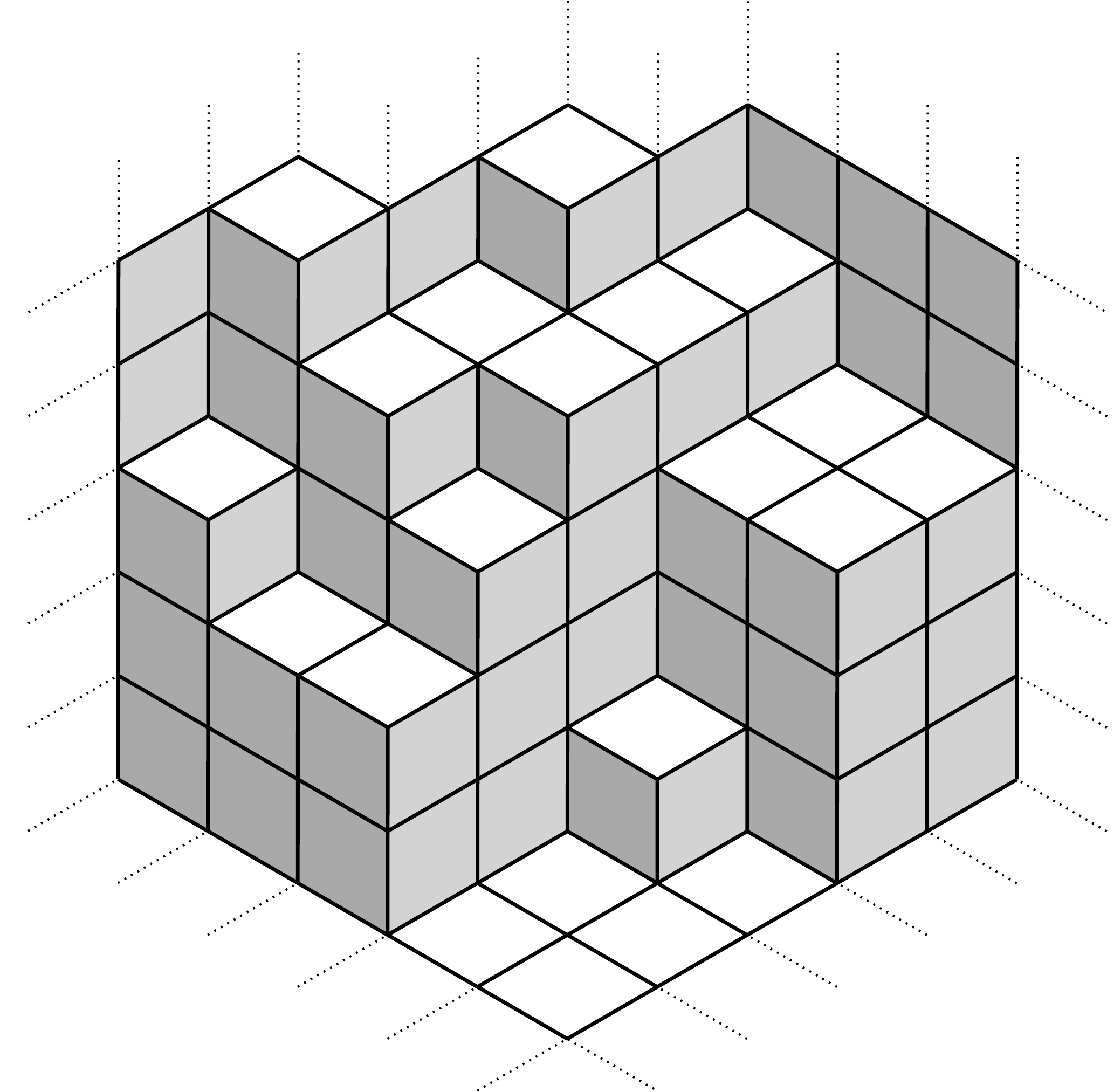}
$\quad\quad$
&\includegraphics[scale=0.22]{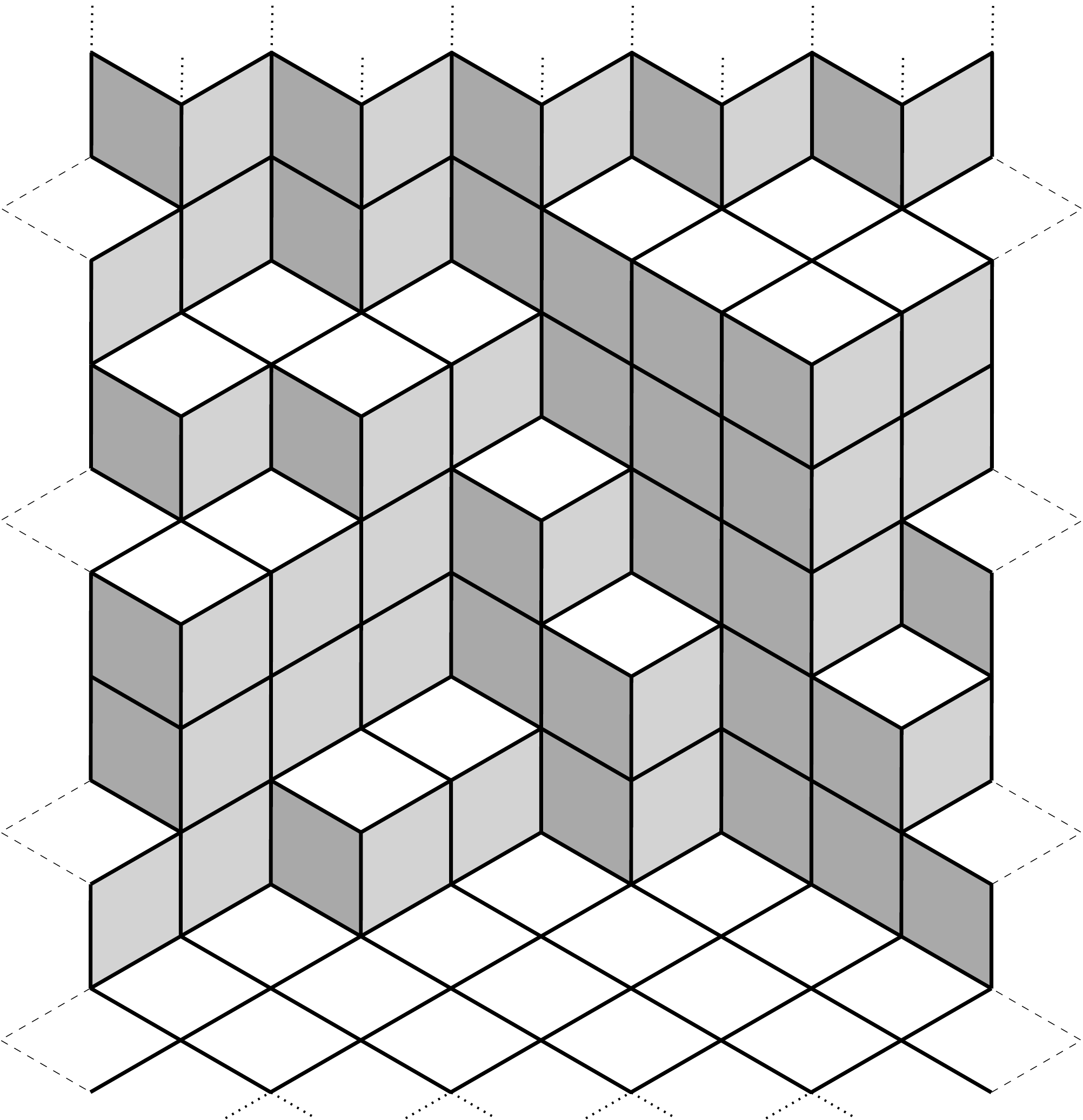} \\
\rotatebox{90}{$\quad\quad\quad\quad$empty room}
&\includegraphics[scale=0.22]{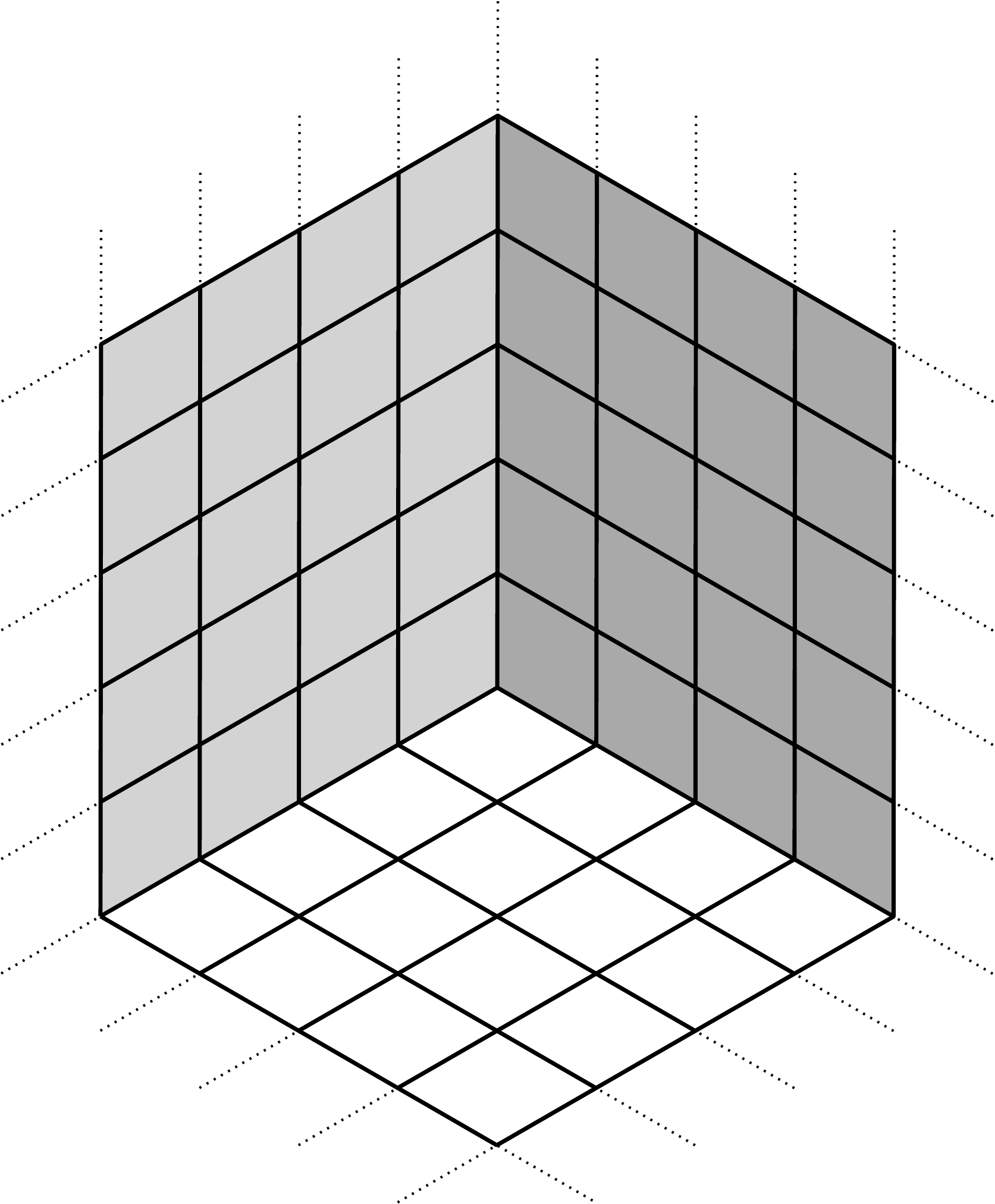}
$\quad\quad$
&\includegraphics[scale=0.22]{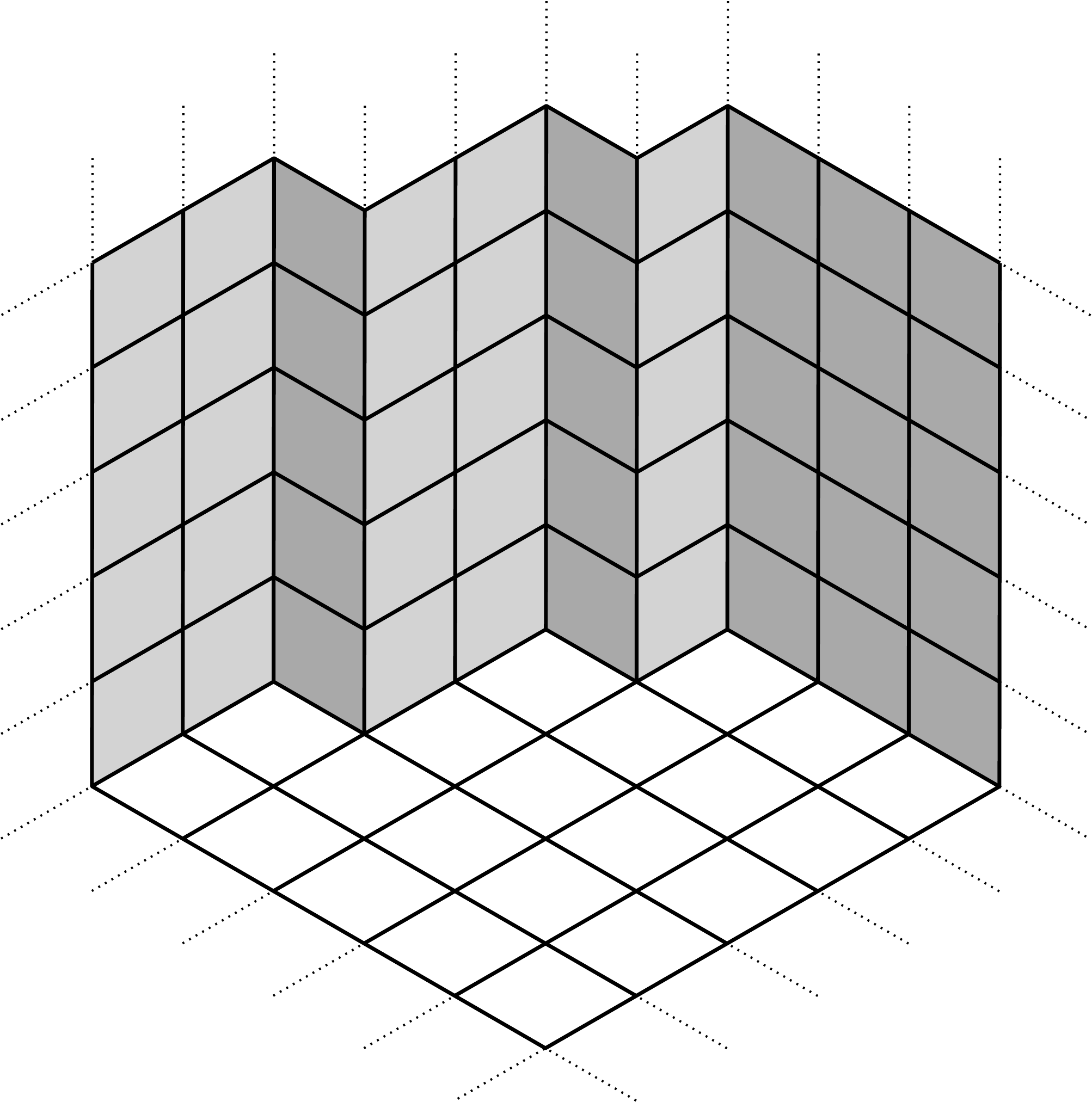}
$\quad\quad$
&\includegraphics[scale=0.22]{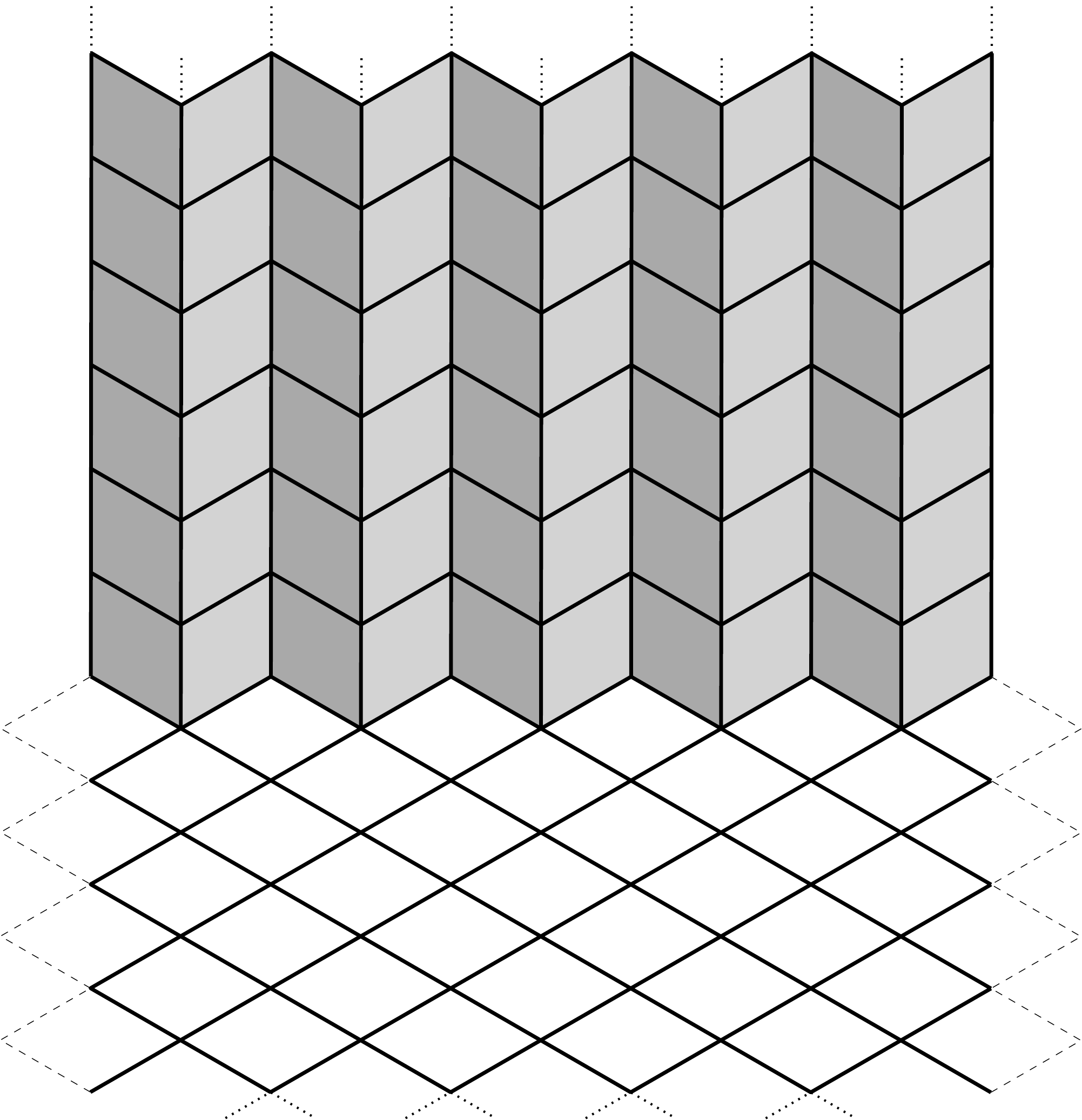} \\
\rotatebox{90}{ }
& Plane partition 
$\quad\quad$
& Skew plane partition  
$\quad\quad$
& Cylindric partition\\
\end{tabular}
\end{center}
 \caption{An example of a plane partition, a skew plane partition with a jagged wall, and a cylindric partition (first row from left to right), together with their stack of unit boxes representation (second row) in corresponding empty rooms (third row).}\label{Fig:lots_of_cubes}
\end{figure}


It is very natural to stack boxes in rooms with different geometries and consider the corresponding questions about limit shapes and fluctuations of the height function. For instance, one may consider a room with a jagged wall that has several corners, as in~\Cref{Fig:lots_of_cubes}, middle column; these are called \emph{skew plane partitions} and were considered in~\cite{OR07,BMRT12,Mkr11,Ahn20}. Taking liberties not only with the shape of the room but with the space it inhabits, one may also consider \emph{cylindric partitions}~\cite{gessel1997cylindric}, i.e. stacks of boxes in a room with a jagged wall and periodic boundary conditions, see \Cref{Fig:lots_of_cubes} right column. 
This gives rise to tilings of a multiply connected domain, topologically equivalent to an annulus. In this paper, we study the case of random cylindric partitions. We restrict our setting to cylindric partitions corresponding to a room with an alternating wall as in~\Cref{Fig:lots_of_cubes}, but more general choices of this `boundary profile' are accessible to the same methods, see~\cite{Bor07}.

Since the set of plane partitions in a given room is typically infinite, one cannot consider a uniform measure. One natural measure on the set of plane partitions $\pi$, with some choice of room, is the $q^{\mvol}$ measure, introduced for ordinary plane partitions in~\cite{vershik1996statistical}, which sets
\[ \PP(\pi) \propto q^{\mvol(\pi)}, \]
where $q \in (0,1)$ is fixed and $\mvol(\pi)$ is the sum of the numbers $n_{i,j}$, equivalently the actual volume of the stack of unit cubes in the corner. 
Here, $q$ is a soft volume constraint parameter which may be interpreted as a preference for plane partitions with volumes in a range determined by $q$. As $q$ approaches~$1$, the $q^{\mvol}$ measure tends to sample plane partitions with more volume. In the case of skew plane partitions, the limit shapes, local statistics and fluctuations of the $q^{\mvol}$ measure have been extensively studied~\cite{OR03,OR07,BMRT12,Ahn20}. Note that if the soft volume constraint is replaced with a strict one by considering uniformly random plane partitions with fixed volume, the limit shape does not change~\cite{CK01}.

A foundational paper~\cite{OR03} computed the local statistics for these measures by reinterpreting them as special cases of \emph{Schur processes}, a certain family of measures on sequences of integer partitions introduced in the same paper. Schur processes, later generalized to Macdonald processes in~\cite{BC14}, have proven to be a versatile tool for analyzing many probabilistic models. In particular, they were used in~\cite{Ahn20} to show that for the $q^\mvol$ measure and Macdonald process generalizations, fluctuations about the limit shape converge to the Gaussian free field. 

The Schur process admits another generalization, the \emph{periodic Schur process} introduced in~\cite{Bor07}, which allows measures on \emph{cylindric} plane partitions to be analyzed by similar techniques. This was used in~\cite{Bor07} and~\cite{betea2019periodic} to study the local bulk and edge fluctuations for cylindric partitions, but global fluctuations around the limit shape for any measure on cylindric partitions have so far not been studied.

In this paper we use the periodic Schur process machinery to establish the limit shape and Gaussian free field fluctuations for the $q^\mvol$ measure on cylindric partitions, in the conformal structure conjectured in~\cite{KO07}. There are actually two versions of the fluctuation result, with two different limit objects, because there are two natural versions of the $q^\mvol$ measure for cylindric partitions.  We introduce these next.

\subsection{Lozenge tilings and cylindric partitions}

Along with the periodic Schur process, another variant was introduced in~\cite{Bor07}, called the \emph{shift-mixed periodic Schur process}. The latter has the advantage that its correlation functions are determinantal. 
The usual periodic Schur process specializes to the $q^{\mvol}$ measure on cylindric partitions introduced above, while the shift-mixed version specializes to another measure on lozenge tilings of the cylinder, which we call the \emph{shift-mixed $q^{\mvol}$ measure}. 

To introduce this measure, consider the two lozenge tilings of the cylinder in \Cref{fig:shift_mixed}. The one on the right has the same tiles as the one on the left, but shifted down by $1$ unit. Viewing them as stacks of boxes in a cylindrical empty room, they define the same stack of boxes, but the stack on right is in a different room in which the boundary between wall and floor is shifted down by $1$ unit. Note that in the box-stacking viewpoint, one cannot transform one picture to the other by adding and removing any finite number of boxes.

The usual (unshifted) $q^{\mvol}$ measure is supported on cylindric partitions, while the shift-mixed $q^{\mvol}$ measure is supported on the larger set of all lozenge tilings of the cylinder such that all lozenges are of type $\hloz$ very far down, and all lozenges are of alternating types $\lloz,\rloz$ very far up. All such tilings are given by shifting a cylindric partition vertically by some integer $S$, as shown in \Cref{fig:shift_mixed} (right) for $S=-1$. The shift-mixed $q^{\mvol}$ measure assigns to each such lozenge tiling of the cylinder a probability proportional to 
\begin{equation}\label{eq:shift_probs}
    \left(u^S q^{N S^2}\right) q^\mvol,
\end{equation}
where $u>0$ is another parameter, $2N$ is the number of columns of the cylinder, and $S$ is the vertical shift of the wall-floor interface described above. The specific choice of weight $u^S q^{NS^2}$ modifying the $q^{\mvol}$ may at first appear ad hoc, but as mentioned it results in determinantal structure. In \Cref{Sec:Dimers} we show additionally that the shift-mixed measure comes from a dimer model with local edge weights, which may be viewed as the origin of this determinantal structure. A different but in some sense related explanation of the determinantal structure of the shift-mixed measure, in terms of a grand canonical ensemble of noninteracting fermions, is given in~\cite{betea2019periodic}.

\begin{figure}[ht]
    \centering
    \includegraphics[scale=0.8]{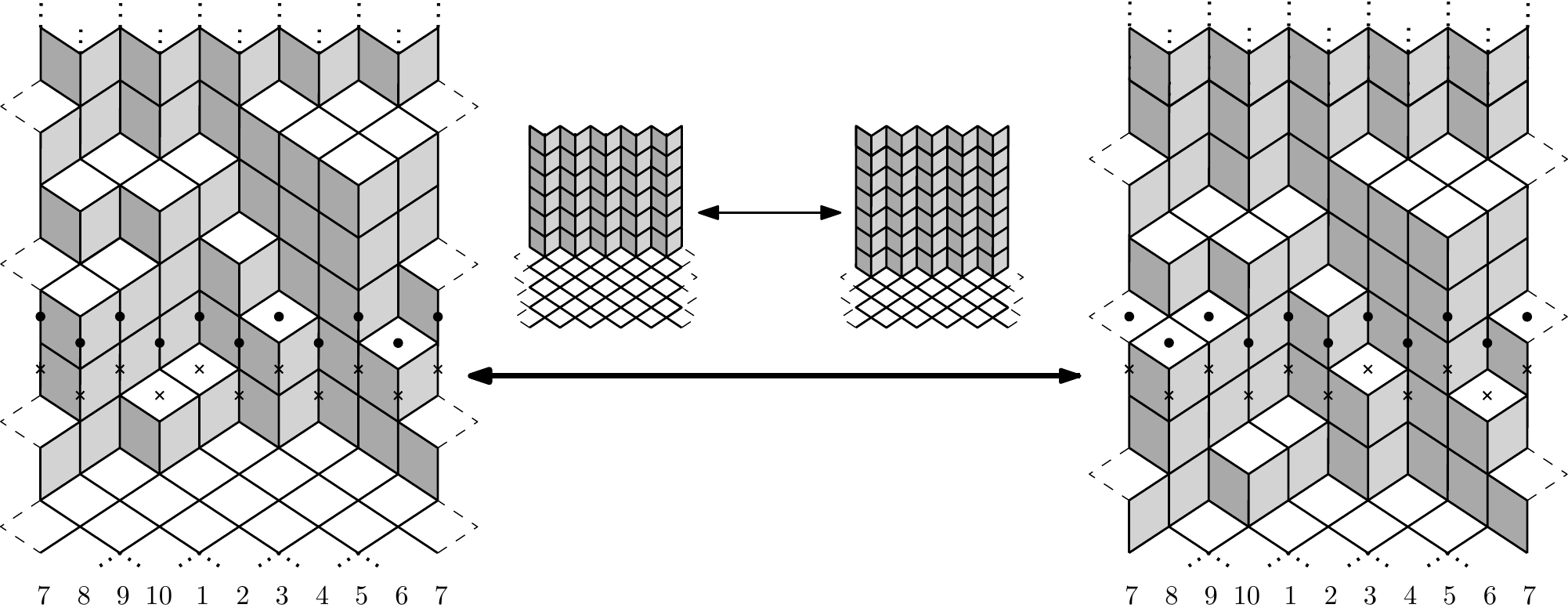}
    \caption{Two lozenge tilings of a cylinder which correspond to the same cylindric partition, but with shifts differing by $1$. For our coordinate system, the midpoint of each vertical edge of the underlying triangular lattice is given coordinates $(\tau^\sharp,y^\sharp),\tau^\sharp \in \{1,\ldots,2N\},$ $y^\sharp \in \Z+1/2$. The horizontal coordinate $\tau^\sharp$ is shown below the figure, and we indicate the vertical coordinate $y^\sharp$ by plotting $\bullet$ (resp. $\times$) at each point $(\tau^\sharp,1/2)$ (resp. $(\tau^\sharp,-1/2)$). The jagged horizontal coordinate is natural from the perspective of Schur processes, see \Cref{Sec:model}.}
    \label{fig:shift_mixed}
\end{figure}

In contrast, the unshifted measure does not correspond to a dimer model with local edge weights. We note that this feature of a random shift $S$ appears in the context of cylindric partitions because the interface between the wall and floor of the room is not a part of the `boundary' from the perspective of a tiling model. The shift does not appear for ordinary and skew plane partitions, because the wall-floor interface is `pinned' to the boundary at infinity.

As is typical in tiling/dimer models, our limit shape and fluctuation results will be expressed in terms of the \emph{height function} of a lozenge tiling of the cylinder. This is a function defined on lattice points of the cylinder, and its value at a point is the number of `absences of a horizontal lozenge' below this point, i.e.
\begin{equation} \label{eq:height_function_intro}
h(\tau^\sharp,y^\sharp) := \sum_{x^\sharp\in\mathbb{Z}+\tfrac{1}{2}: x^\sharp < y^\sharp} \bbone[\mbox{there is no lozenge of type $\hloz$ at $(\tau^\sharp,x^\sharp)$}],
\end{equation}
where the coordinates $(\tau^\sharp,x^\sharp)$ of horizontal lozenges are as described on~\Cref{fig:shift_mixed}. For a tiling of the cylinder satisfying our boundary conditions very far down and far up, the height function $h(\tau^\sharp,y^\sharp)$ vanishes for all sufficiently negative~$y^\sharp$ and $h(\tau^\sharp,y^\sharp) = y^\sharp + S + \tfrac{1}{2}$ for all sufficiently large positive~$y^\sharp$, for a shift~$S$ which depends on the tiling but not on the horizontal coordinate $\tau^\sharp$ (see \Cref{fig:loz_empt} later for an example). This~$S$ is exactly the vertical coordinate of the wall-floor interface alluded to earlier, and conditioning on $S=0$ yields the unshifted $q^{\mvol}$ measure. In particular, the behavior of $h$ at $\pm \infty$ is deterministic for the unshifted $q^{\mvol}$ measure, but has an additional random shift for the shift-mixed $q^{\mvol}$ measure.

\subsection{Main results}
\label{subsec:main_res}

Our main results give the height fluctuations for the unshifted and shift-mixed $q^\mvol$ measures. Before stating these, we give the limit shape, which was computed previously in~\cite[Proposition 7.1]{Bor07}. While concentration was not shown there, we expect our result could be deduced from that one without too much difficulty, though we give an independent argument in \Cref{sec:limit_shape}.

For this section fix $t\in (0, 1)$ and let $q := q(N) := t^{1/N}$. Denote by~$\hh$ the random height function (as in~\eqref{eq:height_function_intro}) 
associated to the $q^{\mvol}$ measure on the cylinder $\Z/(2N\Z) \times (\Z+1/2)$. 
Fix a parameter~$u \in \R_{>0}$ of the shift-mixed $q^{\mvol}$ measure~\eqref{eq:shift_probs}, and let ~$\hs$ denote the random height function associated to the corresponding shift-mixed model, as defined in the previous subsection.


We show that the rescaled random height function $\hh$ converges to an explicit deterministic function on the cylinder as $N$ tends to infinity ($q = t^{1/N}$ tends to $1$). Due to the rotational invariance (in distribution) of our model, the limiting height function must be invariant under rotation. Thus the limiting height function is determined by its values along vertical slices where we fix the coordinate corresponding to the periodic direction. 

\begin{definition}\label{def:limit_shape}
Define a function $\cH: \R \to \R$ by 
\[ \cH'(y) = \frac{2\arctan\left(\sqrt{4 t^{-2y} - 1}\right)}{\pi} \bbone(0 < t^y < 2) 
\quad\text{and}\quad
 \lim_{y \to -\infty} \cH(y) = 0. \]
\end{definition}

We use coordinates $(\tau, y) \in (0,1] \times \R$ for points on the cylinder by identifying $(0,1]$ with the circle, and typically use superscripts $\sharp$ on $\tau$ and $y$ as in the previous subsection for the discrete prelimit coordinates.

\begin{theorem}\label{thm:limit_shape_intro}
Let $t \in (0,1)$, $\tau \in (0,1]$ a horizontal coordinate on the cylinder, and $\cH$ be as above. Then the height function~$\hh$ defined above converges in probability to $\cH$ uniformly, i.e. 
\begin{align*} 
\lim_{N\to\infty} \PP\left( \sup_{y \in \frac{1}{2N} \Z} \left| \frac{1}{2N} \hh(\lfloor 2N \tau \rfloor, 2N y) - \cH(y) \right| \ge \e \right) = 0
\end{align*}
for each $\e > 0$.
\end{theorem}

\begin{remark}
  As already mentioned, the limit shape in \Cref{thm:limit_shape_intro} was computed in~\cite[Proposition~7.1]{Bor07}. In the notation~\cite{Bor07}, the asymptotic density of horizontal lozenges $\rho(\gamma)$ corresponds to our limit shape via
\[ \cH'(y) = 1 - \rho(\gamma) \]
with $t^y = e^{-\gamma}$. In particular, it is shown that
\[ \rho(\gamma) = \begin{cases}
\frac{2}{\pi} \arcsin\left( \frac{e^{-\gamma}}{2} \right) & \gamma \ge -\log 2 \\
1 & \gamma < - \log 2
\end{cases} \]
from which the proposed relation between $\cH'(y)$ and $\rho(\gamma)$ can be readily checked. 
\end{remark}

\begin{remark}\label{rmk:same_limit_shape}
The shift-mixed $q^{\mvol}$ measure has the same limit shape above, as the distribution of the shift is independent of the tiling and is finite-order independent of $N$.
\end{remark}

Our first main result is about the fluctuations of the random height function. To state this result, we introduce the region on which the fluctuations are nontrivial.

Let $p_{\hloz} := p_{\hloz}(\tau,y)$, $p_{\lloz} := p_{\lloz}(\tau,y), p_{\rloz} := p_{\rloz}(\tau,y)$ denote the asymptotic proportions of the lozenges of types $\hloz,\lloz,\rloz$ respectively in a neighborhood of size $o(N)$ around $(\tau,y)$. In~\Cref{Sec:Dimers} we interpret our model as a dimer model, and so the following definition of the liquid region coincides with that of~\cite{KOS06} in their classification of phases of dimer models.

\begin{definition} \label{def:liquid}
The \emph{liquid region} $\sL$ is the subset of $(\tau,y) \in (0,1]\times \R$ where $p_{\hloz}, p_{\lloz}, p_{\rloz}$ are all positive.
\end{definition}

Due to the symmetry in our model, we have $p_{\lloz} = p_{\rloz}$. The following relation holds
\[ \cH'(y) = 1 - p_{\hloz} \]
by \Cref{thm:limit_shape_intro} and because the vertical increments $h(\tau^\sharp,y^\sharp+1) - h(\tau^\sharp,y^\sharp)$ of the height function are $0$ exactly when there is a lozenge of type $\hloz$ at $(\tau^\sharp,y^\sharp)$ by~\eqref{eq:height_function_intro}. From \Cref{def:limit_shape}, we obtain the following description for the liquid region:
\begin{align}\label{eq:def_Liquid}
    \sL = \{(\tau,y) \in (0,1] \times \R: 0 < t^{2y} < 4 \} = \{(\tau,y) \in (0,1] \times \R: y > \tfrac{\log 2}{\log t} \}.
\end{align} 
To describe the limiting Gaussian free field on $\sL$, we map it onto the finite cylinder $\cC = (0,\tfrac{1}{2}) \times \R/\tfrac{|\log t|}{2\pi} \Z$, which we identify with its fundamental domain $(0,\frac{1}{2}) + \tfrac{|\log t|}{2\pi} \bi (0,1]$ in $\C$. This map is given by
\begin{equation}\label{eq:eta_intro}
    \eta(\tau,y) = \frac{1}{2\pi\bi} \log\left(t^\tau \frac{2 - t^{2 y} + \bi \sqrt{4 t^{2y} - t^{4y}}}{2}\right),
\end{equation}
where we choose the branch of the logarithm so that the imaginary part is in $(0,\pi]$. It is a natural identification of $\sL$ with $\cC$: for example, we note that $\eta$ takes vertical segments $\{\tau\} \times (\tfrac{\log 2}{\log t},\infty)$ in $\sL$ to the horizontal segment $(0,\frac{1}{2}) + \tfrac{\tau|\log t|}{2\pi} \bi \subset \cC$. 

We now briefly define the Gaussian free field on $\cC$; further details can be found in \Cref{GFF_def}. The Gaussian free field $\Phi$ on $\cC$ with $0$-Dirichlet boundary conditions is the random measure identified by the property that for any smooth test functions $\varphi_1,\ldots,\varphi_k$ on $\cC$, the random vector $(\langle \Phi, \varphi_1 \rangle,\ldots,\langle \Phi, \varphi_k \rangle)$ is a centered Gaussian vector with covariances
\[ \cov\left( \langle \Phi, \varphi_i \rangle, \langle \Phi, \varphi_j \rangle \right) = \int_{\cC} |dz|^2 \int_{\cC} |dw|^2 \varphi_i(z) \varphi_j(w) G(z,w) , \quad \quad i,j = 1,\ldots,k. \]
Here
\[ G(z,w) = -\frac{1}{2\pi} \log\left| \frac{\Theta(z - w\,|\,\omega)}{\Theta(z + \overline{w}\,|\,\omega)} \right| \]
is the Green's kernel for the Laplace operator on $\cC$ with $0$-Dirichlet boundary conditions, where
\[
\Theta(\eta\,|\,\omega) := \sum_{m=-\infty}^\infty (-1)^m e^{2 \pi \bi \left((m+1/2) \eta + \frac{m(m+1)}{2} \omega\right)}
\]
is the Jacobi theta function, and $\omega$ is defined by $t = e^{2\pi\bi \omega}$, see \Cref{sec:GFF_and_Green} for further details.

\begin{theorem}\label{thm:unshifted_gff_convergence_intro}
Fix $t \in (0,1)$. Then the normalized height function of the unshifted $q^{\mvol}$ measure converges as $N \to \infty$ to the $\eta$-pullback of the Gaussian free field $\Phi$ on the cylinder $\cC = (0,\tfrac{1}{2}) \times \R/\tfrac{|\log t|}{2\pi}\Z$ with $0$-Dirichlet boundary conditions, where $\eta: \sL \to \cC$ is given in~\eqref{eq:eta_intro}.

More precisely, for any $\tau_1,\ldots,\tau_n \in (0,1]$ and integers $k_1,\ldots,k_n > 0$, the random vector
\begin{align} \label{eq:prelimit_vector}
\left( \frac{1}{2N} \sum_{y \in \frac{1}{2N} (\Z+\frac{1}{2})} \Bigg(\hh(\floor{2N\tau_i},2Ny) - \E[\hh(\floor{2N\tau_i},2Ny)]\Bigg) t^{k_iy} \right)_{1 \le i \le n}
\end{align}
converges in distribution to the Gaussian random vector
\[ \left( \frac{1}{\sqrt{\pi}}\int_{\frac{\log 2}{\log t}}^\infty \Phi(\eta(\tau_i,y)) t^{k_iy} \, dy \right)_{1 \le i \le n} \]
as $N\to\infty$.
\end{theorem}

\begin{remark}\label{rmk:gff_slices}
We note that the Gaussian free field is defined as a distribution on a $2$-dimensional domain, and refer to \Cref{GFF_def} for an explanation of the meaning of integrating it along $1$-dimensional slices. Convergence to the Gaussian free field in the sense of $1$-dimensional slices was also considered previously, for example for tiling and random matrix models, see e.g.~\cite{borodin2014clt1,borodin2014clt2,BuG18}. We also note the similar setting of integrating the Gaussian free field over circular arcs was considered in~\cite{duplantier2011liouville}. In \Cref{GFF_def} we give more detail on the Gaussian free field, in particular why the above statement can be interpreted as convergence of the height function to it on a certain space of test function. 
\end{remark}

It was conjectured in~\cite{KO07} that the fluctuations of the height function of a planar bipartite biperiodic dimer model (which includes the case of tiling models) on a simply connected domain converge to the Gaussian free field on the liquid region with respect to an explicit conformal structure in terms of the limit shape. Previously, the conjecture has been shown for many cases of simply connected domains for lozenge tilings (e.g.~\cite{Ken08,BF14,Pet15,BuG18,berestycki2020dimers,huang2020height}) and domino tilings (e.g.~\cite{Ken01,Rus18,Rus20,bufetov2018asymptotics}). In~\cite{CLR20, chelkak2021bipartite}, Gaussian free field fluctuations are shown for a general family of planar bipartite graphs on simply connected domains, not necessarily regular as in~\cite{KO07}, where the conformal structure is given in terms of a new class of embeddings.

For our multiply connected setting, \Cref{thm:unshifted_gff_convergence_intro} establishes that the limiting height fluctuations of the unshifted $q^{\mvol}$ model are given by the Gaussian free field on the liquid region~$\sL$. Furthermore, we check in \Cref{sec:conf_str} that the conformal structure on $\sL$ is precisely the conformal structure predicted in~\cite{KO07}, verifying the natural extension of their conjecture to our domain. We mention also the analogous result of~\cite{BuG18} on Gaussian free field convergence for uniform lozenge tilings of a hexagonal domain with a hole (which like the cylinder is topologically equivalent to annulus), conditioned such that the height function takes a fixed value at the hole. In \Cref{sec:hol_hex} we will discuss more thoroughly the analogy between our results and those of~\cite{BuG18}.

Our final main result establishes the convergence of the height fluctuations of the shift-mixed $q^{\mvol}$ measure. For the shift-mixed model, there is an additional component to the height fluctuations given by a discrete Gaussian random variable; these were introduced in~\cite{lisman1972note}, see also~\cite{agostini2019discrete,kemp1997characterizations,szablowski2001discrete}.

\begin{definition}\label{def:disc_gauss}
A \emph{discrete Gaussian} $S \sim \cN_{\operatorname{discrete}}(m,C)$ is the $\Z$-valued random variable defined by 
\[
\Pr(S = x) \propto e^{-C(x-m)^2},
\]
where $m \in \R, C \in \R_{>0}$.
\end{definition} 

The parameters $m,C$ in the above definition are related to the mean and variance but do not correspond to them exactly as they do for usual Gaussians. The appearance of discrete Gaussian shifts is quite natural, as the probabilities 
\[
\frac{1}{Z}\left(u^S q^{N S^2/2}\right) q^\mvol
\]
of a shift-mixed $q^\mvol$ cylindric partition may be interpreted as giving an unshifted $q^\mvol$ plane partition, together with an independent random shift $S \in \Z$ with $\Pr(S=x) = u^x t^{x^2/2}$, i.e. 
\begin{equation}\label{eq:specific_disc_gaus}
  S \sim \cN_{\operatorname{discrete}}\left(\frac{\log u}{\log t}, \frac{|\log t|}{2}\right).  
\end{equation}

\begin{theorem}\label{thm:shifted_gff_intro}
Fix $u \in \R_{>0}$ and $t\in(0, 1)$, set $q := q(N) := t^{1/N}$. Then the normalized height function~$\hs$ of the shift-mixed $q^{\mvol}$ measure converges to the $\eta$-pullback of the Gaussian free field with a discrete Gaussian shift,
\[ \hs(\floor{2N\tau},2Ny) - \E[\hh(\floor{2N\tau},2Ny)] \xrightarrow{N \to \infty} \Phi(\eta(\tau,y))-S \cH'(y).\]
Here $S$ is as in~\eqref{eq:specific_disc_gaus}, $\hh$ is the unshifted height function, and $\eta,\cH$ are as in~\eqref{eq:eta_intro} and \Cref{def:limit_shape}. Specifically, for any $\tau_1,\ldots,\tau_n \in (0,1]$ and integers $k_1,\ldots,k_n > 0$, the random vector 
\begin{equation*}
    \left( \frac{1}{2N} \sum_{y \in \frac{1}{2N} (\Z+\frac{1}{2})} \Bigg(\hs\left(\lfloor 2N\tau_i\rfloor,y\right) - \E[\hh\left(\lfloor 2N\tau_i\rfloor,y\right)]\Bigg) t^{k_iy}\right)_{1 \leq i \leq n} 
\end{equation*}
converges in distribution as $N \to \infty$ to the random vector
\begin{equation*}
    \left( \int_{\frac{\log 2}{\log t}}^\infty \left( \frac{1}{\sqrt{\pi}}\Phi(\eta(\tau_i,y)) - S \cH'(y) \right) t^{k_iy} dy \right)_{1 \le i \le n}.
\end{equation*} 
\end{theorem}

As alluded to in \Cref{sec:background}, the unshifted $q^{\mvol}$ model does not correspond to a dimer model with local edge weights, in contrast with the shift-mixed model. From the perspective of dimer models, \Cref{thm:unshifted_gff_convergence_intro} and \Cref{thm:shifted_gff_intro} show that the height fluctuations of the conditioned (unshifted) model are given solely by the Gaussian free field whereas the unconditioned (shift-mixed) model has an additional discrete Gaussian component. Conjecturally, the height fluctuations of tiling models on multiply connected domains have discrete Gaussian components, one for each hole, see e.g.~\cite[Conjecture 24.1, Conjecture 24.2]{gorin_2021}. A proof for certain domains is expected to appear in~\cite{borot2021unpublished}. In fact, these conjectures predict an explicit value of the `variance' parameter $C$ of the discrete Gaussians which appear. In \Cref{sec:hol_hex} we show that \Cref{thm:shifted_gff_intro} confirms this prediction. 

To our knowledge, \Cref{thm:shifted_gff_intro} is the first result to establish limiting height fluctuations for an unconditioned dimer model with local edge weights on a domain topologically equivalent to an annulus (as mentioned, for a conditioned dimer model such results were shown previously in~\cite{BuG18}). The only other non-simply connected domain for which we are aware of such results is the torus, in which the limiting height fluctuations are also given by a Gaussian free field on the torus and an additional discrete Gaussian component coming from the height change from winding around the torus. Convergence of the discrete component to a discrete Gaussian was shown in~\cite{boutillier2009loop} for lozenge tilings, and the discrete and Gaussian free field components for domino tilings were studied previously in~\cite{dubedat2015dimers,dubedat2015asymptotics,berestycki2019dimer}. The limit pair given by the Gaussian free field and discrete component is often referred to as the \emph{compactified Gaussian free field}, see~\cite[Section 2.1.3]{dubedat2015dimers} (also~\cite[Lecture 1.4]{gawedzki1997lectures} and~\cite[Sections 6.3.5 and 10.4.1]{francesco2012conformal}), and the discrete component is called the \emph{instanton component}. 

It is worth noting that a discrete component was found much earlier in~\cite{pastur2006limiting} for random matrix models whose limiting spectral measure has disconnected support. These discrete shifts should appear for similar reasons, though they were not explicitly linked to the discrete Gaussian.

\subsection{Methods} 

Our approach is based on new formulas for joint exponential moments of the height function of periodic Schur processes, for both the shift-mixed and unshifted models. For the unshifted $q^{\mvol}$ model, the joint moment formulas are particularly amenable to asymptotic analysis. The asymptotics of these formulas form the basis of the proofs of \Cref{thm:limit_shape_intro,thm:unshifted_gff_convergence_intro}. For the shift-mixed $q^{\mvol}$ model, we deduce \Cref{thm:shifted_gff_intro} from \Cref{thm:unshifted_gff_convergence_intro} by a general argument which applies for any shift distribution satisfying a moment condition, see \Cref{thm:shift-mixed_fluctuations}. In particular, this argument does not rely on the shift being a discrete Gaussian.

The formulas for the joint moments were inspired by~\cite{koshida2020free}, which obtained formulas for observables for \emph{periodic Macdonald processes}. Here, the periodic Macdonald process generalizes the periodic Schur process in the same manner as the Macdonald process (see~\cite{BC14}) generalizes the Schur process. While the results of~\cite{koshida2020free} do not imply our formulas, there is agreement between the two in certain special cases.

Previous works as in~\cite{Bor07,betea2019periodic} on periodic Schur processes studied asymptotic questions via direct asymptotic analysis of contour integral formulas for the correlation kernel. For simply connected tiling models with determinantal structure, various works such as~\cite{duits2013gaussian,BF14,Pet15} have used similar methods to establish Gaussian free field fluctuations of the height function. Based on these two considerations, it seems that a natural way to access fluctuations for the shift-mixed $q^{\mvol}$ model is to use the correlation kernels, and indeed this was the first approach we considered.

However, due to the discrete shift term in the fluctuations, which is not present in tilings of simply connected domains, there are additional difficulties in this approach. If we consider the unshifted model instead, the fluctuations no longer have the discrete shift, but now the model is non-determinantal. We believe it should be possible to study the fluctuations of the unshifted $q^{\mvol}$ measure via direct analysis of the correlation functions, but that this approach would be considerably more technical than the moment-based one in this paper.

Our moment-based approach is closely related to those of~\cite{BG15,GZ18,Ahn20}, which establish Gaussian free field fluctuations for various models arising as degenerations of Macdonald process. In these works, there is no determinantal/free fermionic structure, and the analogous moment formulas are obtained through certain difference operators acting on symmetric functions. We expect that our formulas can be obtained via difference operators, using the methods from~\cite{BCGS16}. With this approach, the determinantal structure for the shift-mixed model manifests itself through the presence of a richer family of operators. However, we instead obtain these moment formulas using the generating series of the correlation kernel of the shift-mixed periodic Schur processes derived in~\cite{Bor07}, see \Cref{thm:shifted_observable_formula}. This is the only place that the determinantal structure enters our analysis: we then use these formulas and conditioning to obtain the corresponding formulas for the (non-determinantal) unshifted periodic Schur process in \Cref{thm:observable_formula}, which are the starting point for the asymptotic analysis. 

Since moment formulas like ours can be derived from determinantal structure but exist in settings where this structure does not, it might seem that they would be more difficult to use than directly studying the correlation kernel. However, for showing convergence to the Gaussian free field, we found that packaging the determinantal structure of the shift-mixed model into these moment formulas led to much simpler analysis than approaches which used the correlation kernel directly.

\subsection{Organization.} The paper is organised as follows. 
The definitions and basic properties of the Gaussian free field and Green's function are given in Section~\ref{sec:GFF_and_Green}.  
In Section~\ref{Sec:model} we introduce the periodic Schur process, an integral formula for its correlation kernel due to~\cite{Bor07}, and an equivalent dimer model. Section~\ref{sec:moment_formula} is devoted to a contour integral formula for certain Laplace transform observables of the~$q^{\mvol}$ height function.
In Section~\ref{sec:asymptotics} we compute asymptotics of the cumulants of these observables for the unshifted $q^{\mvol}$ measure.
Section~\ref{sec:main_proofs} uses these to prove the limit shape and fluctuations results.
Finally, in Section~\ref{sec:conf_str} we match our main results with existing conjectures on height fluctuations in dimer models.

\bigskip

\addtocontents{toc}{\protect\setcounter{tocdepth}{1}}
\subsection*{Acknowledgements}  The authors are deeply grateful to Alexei Borodin and Vadim Gorin for suggesting the problem and many helpful discussions, Leonid Petrov for answering questions about~\cite{Pet15}, and the anonymous referees for helpful comments.
MR~was supported by the Swiss NSF grants P400P2-194429 and P2GEP2-184555 and also partially supported by the NSF Grant DMS-1664619. RVP was supported by the National Science Foundation Graduate Research Fellowship under Grant No. \#$1745302$.
\addtocontents{toc}{\protect\setcounter{tocdepth}{2}}

\section{Gaussian free field and Green's function on cylinder}\label{sec:GFF_and_Green} 

\subsection{Gaussian free field} \label{GFF_def}

In this section we define the two-dimensional Gaussian free field on a cylinder with Dirichlet boundary conditions.

Consider the Laplace operator on a cylinder $\mathcal{C}:=(0,\alpha)\times \R/\beta\Z$
\[
\Delta=\frac{\partial^2}{\partial x^2}+\frac{\partial^2}{\partial y^2} = 4 \partial_z \partial_{\bar{z}}
\]
where for the latter definition we identify $\cC$ with $\{x+\bi y \in \C: 0 < x < \alpha\}/\bi \beta \Z$.

\begin{definition}\label{def:Green_1}
The Green's function $G(z_1,z_2)$ on a cylinder $\mathcal{C}$ with zero Dirichlet boundary conditions is a function $G:\mathcal{C}\times\mathcal{C}\to\mathbb{R}$, such that it is a kernel of the Laplace operator
\begin{align*}
     \Delta_{z_1}G(z_1,z_2)=\Delta_{z_2}G(z_1,z_2)=\delta(z_1=z_2)
\end{align*}
and it satisfies zero Dirichlet boundary conditions
\[ G(z_1,z_2) = 0, \quad \quad z_2 \in \partial \cC.
\]
\end{definition}

Before defining the Gaussian free field, we begin with some informal motivation; we stick to the cylinder for this discussion, but it holds also for general Riemann surfaces. The \emph{Gaussian free field on a cylinder $\mathcal{C}$} with Dirichlet boundary conditions is informally the random `function' $\Phi$ on $\mathcal{C}$ such that $\Phi(z)$ is a Gaussian random variable for every $z \in \mathcal{C}$, with covariances
\[ \cov(\Phi(z_1),\Phi(z_2)) = G(z_1,z_2), \]
for $z_1,z_2 \in \cC$, where $G$ is the Green's function on $\mathcal{C}$ with zero Dirichlet boundary conditions. To make rigorous sense of this definition, one defines $\Phi$ as a random distribution on a suitable class of test functions $\varphi: \cC \to \R$, by mandating that $\int_{\cC} \varphi(z) \Phi(z) \, dx \, dy$ is a centered Gaussian random variable with
\[ \var\left( \int_{\cC} \varphi(z) \Phi(z) \, dx \, dy \right) = \int_{\cC} \int_{\cC} \varphi(z_1) \varphi(z_2) G(z_1,z_2) \, dx_1 \, dy_1\,dx_2\, dy_2 \]
where $z = x + \bi y$, and similarly for $z_1,z_2$. Note that by the polarization identity, the above equation determines covariances as well, and below we will similarly only specify the variances.



A natural space of test functions on which to define $\Phi$ is the set of compactly supported smooth functions $\cD(\cC)$ on $\cC$. However, the statement of \Cref{thm:unshifted_gff_convergence_intro} corresponds to pairing the Gaussian free field with test `functions' which come from integrating exponential functions along ($1$-dimensional) vertical slices, and are actually distributions of the form $\varphi(x,y) = \delta(y=y_0)f(x)$ for fixed $y_0$ and smooth $f: \R \to \R$ which in our case are exponential functions. These live inside the negative order Sobolev space\footnote{We refer to the excellent exposition of the Gaussian free field in \cite[Chapter 1]{berestycki2015introduction} for definitions and an elaboration of our discussion.} $\mathbb{H}_0^{-1}(\cC)$, and in fact form a dense subset of it. Hence we want an appropriate notion of testing the Gaussian free field against elements of $\mathbb{H}_0^{-1}(\cC)$, which we do in the following manner. Let $(\Omega, \cF, \PP)$ denote the common probability space on which the random variables in $\left\{ \int_\cC \varphi(z) \Phi(z) \, dx \, dy \right\}_{\varphi \in \cD(\cC)}$ live. We view the Gaussian free field as a map from $\cD(\cC)$ to $L^2(\Omega,\cF,\PP)$ via $\varphi \mapsto \int_\cC \varphi(z) \Phi(z) \, dx \, dy$. Endowing $\cD(\cC)$ with the $\mathbb{H}_0^{-1}(\cC)$ norm, we see that $\Phi$ is an isometry and therefore extends to $\mathbb{H}_0^{-1}(\cC)$, by density of $\cD(\cC)$ in this Sobolev space. We now give a precise definition.

\begin{definition}\label{def:GFF}
The \emph{Gaussian free field $\Phi$ on a cylinder $\mathcal{C}$} with zero Dirichlet boundary conditions and test space $\mathbb{H}_0^{-1}(\cC)$ is a jointly Gaussian family of random variables indexed by the elements of $\mathbb{H}_0^{-1}(\cC)$, with covariances determined by
\[ \var\left( \int_{\cC} \varphi(z) \Phi(z) \, dx \, dy \right) = \int_{\cC} \int_{\cC} \varphi(z_1) \varphi(z_2) G(z_1,z_2) \, dx_1 \, dy_1\,dx_2\, dy_2 \]
where $z = x + \bi y$ and similarly for $z_1,z_2$, and $G$ is the Green's function on $\cC$ with zero Dirichlet boundary conditions.
\end{definition}

In addition, we can also consider pullbacks of the Gaussian free field. Given some map $\Omega:D \to \cC$, where~$D$ is some domain, then $\Phi \circ \Omega$ is informally the random Gaussian field on $D$ such that
\[ \cov\left( \Phi\circ \Omega(w_1), \Phi \circ \Omega(w_2) \right) = G(\Omega(w_1), \Omega(w_2)). \]
As above, we can formally define $\Phi \circ \Omega$ as a Gaussian family indexed by test functions on $D$.

Our main result for the unshifted $q^\mvol$ measure, \Cref{thm:unshifted_gff_convergence_intro}, was stated earlier in elementary terms as convergence of certain explicit random vectors to Gaussian ones with explicit covariance. However, it may be equivalently stated as the fact that the centered height function converges to the pullback of the Gaussian free field by the map $\eta: \sL \to \cC$ of \eqref{eq:eta_intro}, on the subset
\[
\cS = \mbox{linear combinations of }\{ \sL \ni (\tau,y) \mapsto \delta(\tau = \tau_0) t^{k y}: k \in \Z_{\ge 1}\}
\]
of $\mathbb{H}_0^{-1}(\sL)$. This explains the statement in \Cref{thm:unshifted_gff_convergence_intro} that the height function converges to a Gaussian free field; the analogous remark applies to \Cref{thm:shifted_gff_intro}, taking into account the discrete shift.

\begin{remark}
While the set $\cS$ above is clearly more restrictive than $\mathbb{H}_0^{-1}(\sL)$, we note that the Gaussian free field is still determined by this subset in the following sense. Suppose $\Psi_i: \mathbb{H}_0^{-1}(\sL) \to L^2(\Omega_i,\cF_i,\PP_i)$ for $i = 1,2$ are two continuous linear maps. Let us say that $\Psi_1$ and $\Psi_2$ \emph{agree in distribution on a subset} $A \subset \mathbb{H}_0^{-1}(\sL)$ if $\{\Psi_1(\varphi)\}_{\varphi \in A}$ and $\{\Psi_2(\varphi)\}_{\varphi \in A}$ agree in finite dimensional distributions. Since $\cS$ is dense in $\mathbb{H}_0^{-1}(\sL)$, if $\Psi_1$ and $\Psi_2$ agree in distribution on $\cS$, then $\Psi_1$ and $\Psi_2$ agree in distribution on $\mathbb{H}_0^{-1}(\sL)$.

Thus our main result proves convergence to the Gaussian free field on a dense subset of $\mathbb{H}_0^{-1}(\sL)$. While this does not imply the convergence on the entirety of $\mathbb{H}_0^{-1}(\sL)$, it does uniquely determine the limit by the reasoning above.
\end{remark}

For our applications, we will work with $\cC$ through its fundamental domain $(0,\alpha) + \bi (0,\beta] \subset \C$. Functions on $\cC$ lift to functions $f$ on the covering space $(0,\alpha) + \bi \R \subset \C$ satisfying the periodicity condition
\[ f(z) = f(z + \bi \beta). \]
We often identify such periodic $f$ with the restriction $f|_{(0,\alpha) + \bi (0,\beta]}$ to the fundamental domain.

\subsection{Computing the Green's function.}

In this section we give an explicit formula of the Green's function on a cylinder in terms of Jacobi theta functions. 
Let us first give a characterization which can be seen as an equivalent definition of the Green's function, see e.g.~\cite[Lemma 3.7--3.8]{werner2020lecture}.  

\begin{definition}\label{def:green}
The Green's function on a cylinder $\mathcal{C}:=(0,\alpha)\times \R/\beta\Z$ with zero Dirichlet boundary conditions is the unique function $G:\mathcal{C}\times\mathcal{C}\to\mathbb{R}$ satisfying the following conditions
\begin{enumerate}
    \item[(i)] $G$ is real-valued and $G(z_1,z_2)=G(z_2,z_1)$;
    \item[(ii)] Harmonicity: $\Delta_{z_1}G(z_1,z_2)=0$ for any $z_1\neq z_2$;
    \item[(iii)] Singularity: for a fixed $z_1\in \mathbb{C}^{+}$ one has $G(z_1,z_2)=-\frac{1}{2\pi}\cdot \log|z_1-z_2|+O(1)$ as $z_2\to z_1$;
    \item[(iv)] Boundary conditions: for a fixed $z_1\in \C^+$
    \[
    \begin{cases} 
    G(z_1,z_2+\bi \beta) = G(z_1,z_2) & \mbox{(periodicity)} \\
    G(z_1,\bi y) = G(z_1, \alpha+\bi y) = 0 & \mbox{($0$-Dirichlet)}.
    \end{cases}\]
\end{enumerate}
\end{definition}

We recall the infinite $q$-Pochhammer symbol, defined by
\[
(a;t)_\infty = \prod_{n \geq 1} (1-a t^n).
\]
For $t \in (0,1)$, let $\theta_1$ be the Jacobi theta function (see e.g.~\cite{erdelyibook}) defined as
\begin{align}
\theta_1(z;t) := \sum_{m=-\infty}^{\infty}(-1)^m t^{m(m+1)/2}z^{m+1/2}=(z^{1/2}-z^{-1/2}) (t;t)_\infty (tz;t)_\infty (t/z;t)_\infty, \label{eq:theta1}
\end{align}
which is holomorphic on $\C^\times.$  
Define
\begin{align}\label{def:Theta}
\Theta(\eta\,|\,\omega) := \theta_1(e^{2\pi\bi \eta},e^{2\pi \bi \omega}).
\end{align}

\begin{proposition}\label{prop:green}
Let $t=e^{2\pi\bi\omega} \in (0,1)$. The function $\widetilde G(\eta_1,\eta_2):=-\frac{1}{2\pi}\log \left|
\frac{\Theta(\eta_1-\eta_2\,|\,\omega)}
{\Theta(\eta_1+\overline\eta_2\,|\,\omega)}
\right|$ is the Green's function with zero boundary conditions on a cylinder $(0,\tfrac{1}{2}) \times \R/\tfrac{|\log t|}{2\pi}\Z$.
\end{proposition}

\begin{proof} Let us check that $\widetilde G$ satisfies all conditions listed in Definition~\ref{def:green}.
Note that $-\Theta(\eta\,|\,\omega)=\Theta(-\eta\,|\,\omega)$, therefore $\widetilde G(\eta_1,\eta_2)=\widetilde G(\eta_2,\eta_1).$  

Recall that $\theta_1(z;t)$ is holomorphic on $\C^\times$ as a function of $z$ with zeros at $z=t^n$, where $n\in\mathbb{Z}$. Hence the function $\Theta(\eta \,|\, \omega)$ is holomorphic with zeros at 
$\eta=n+\bi\tfrac{\log t}{2\pi}m$ with $n,m \in\mathbb{Z}$. This together with a definition of the function $\widetilde G$ imply that $\widetilde G(\eta_1,\eta_2)$ is harmonic (as a real part of holomorphic function) as a function of $\eta_1 \in \C^\times$ for any $\eta_1\neq \eta_2$ and has a correct singularity as $\eta_2\to\eta_1$. 

To check the periodicity it is enough to note that $\theta_1(t^{\pm 1}z;t)=-t^{-1/2}z^{\mp 1}\theta_1(z;t)$.
Finally, let us check that it satisfies zero Dirichlet boundary conditions. Note that $-\bi y=\overline{\bi y}$, therefore 
\[
\left|\frac{\Theta(\eta_1-\bi y\,|\,\omega)}
{\Theta(\eta_1+\overline{\bi y}\,|\,\omega)}\right|=1
\] 
and hence $G(z_1,\bi y)=0$. To see that $G(z_1,\tfrac12+\bi y)=0$
 it is enough to note that 
 \[\frac{\theta_1(\zeta_1/\zeta_2; t)}{\theta_1(\zeta_1/\overline{\zeta}_2; t)} = \frac{\Theta(\eta_1-\eta_2\,|\,\omega)}
{\Theta(\eta_1+\overline\eta_2\,|\,\omega)},
\]
where $\zeta_j := e^{2\pi\bi \eta_{j}}$ for $j=1,2$. Indeed, note that $\zeta_{2} = e^{2\pi\bi \left(\bi y + \tfrac12\right)}$ is real and therefore $\left|\tfrac{\theta_1(\zeta_1/\zeta_2; t) }{ \theta_1(\zeta_1/\overline{\zeta}_2; t)}\right| =1$.
\end{proof}

\section{The model}\label{Sec:model}

\subsection{Periodic Schur processes}

We begin with some preliminary definitions from symmetric function theory. We say that a sequence of nonnegative integers $\lambda = (\lambda_1 \ge \lambda_2 \ge \cdots)$ is a \emph{partition} if all but finitely many $\lambda_i$ are zero. 

The \emph{length} of the partition $\lambda$ is the number of nonzero elements in the sequence, denoted $\ell(\lambda)$. Let $\Y_N$ denote the set of partitions with length at most $N$, and $\Y$ denote the set of all partitions. Given partitions $\lambda$ of length $N$ and $\mu$ of length $N-1$ we write $\mu\prec\lambda$ if \[
\lambda_1\geq\mu_1\geq\lambda_2\geq\ldots\geq\mu_{N-1}\geq\lambda_N.
\] 
A \emph{skew Schur polynomial indexed by the skew shape $\lambda/\mu$} is defined by
\begin{align}\label{eq:skew}
s_{\lambda/\mu}(x_1,\ldots,x_{N-k}):=
\sum_{\substack{\lambda^{(k+1)},\ldots,\lambda^{(N-1)} \\
\mu=\lambda^{(k)}\prec\lambda^{(k+1)}\prec\ldots\prec\lambda^{(N)}=\lambda}}
x_1^{|\lambda^{(k+1)}|-|\lambda^{(k)}|}
\cdot\ldots\cdot
x_{N-k}^{|\lambda^{(N)}|-|\lambda^{(N-1)}|}
\end{align}
which is a homogeneous symmetric polynomial in $x_1,\ldots,x_N$ of degree $|\lambda| := \sum_{i \ge 1} \lambda_i$. If $\mu = (0,0,\ldots)$ is the zero partition, then $s_{\lambda/\mu}$ is the Schur function indexed by $\lambda$ given by the alternant formula
\[ s_\lambda(x_1,\ldots,x_N) = \frac{\det[x_i^{\lambda_j + N - j}]_{i,j=1}^N}{\det[x_i^{N-j}]_{i,j=1}^N}. \]
We write $\mu \subset \lambda$ if $\lambda_i \ge \mu_i$ for $i \ge 1$. From~\eqref{eq:skew}, observe that $s_{\lambda/\mu} = 0$ whenever $\mu \not \subset \lambda$.

If we denote by $\Lambda_N$ the algebra of symmetric polynomials in $N$ variables over $\C$, there is a natural projection map $\Lambda_{N+1} \to \Lambda_N$ defined by
\[ f(x_1,\ldots,x_{N+1}) \mapsto f(x_1,\ldots,x_N,0). \]
We may then define the algebra $\Lambda$, viewed informally as symmetric functions in infinitely many variables, by taking the inverse limits of each degree $d$ component of $\Lambda_N$ with respect to this projection and recombining them, see~\cite{Mac} for details. It is easy to see that the skew Schur polynomials satisfy the consistency relation
\[ s_{\lambda/\mu}(x_1,\ldots,x_M,0) = s_{\lambda/\mu}(x_1,\ldots,x_{M+1}), \quad \quad M \ge \ell(\lambda). \]
Therefore we may view $s_{\lambda/\mu}$ as an element of $\Lambda$.

Another distinguished set of symmetric functions that we consider are the Newton power sums
\[ p_k = \sum_{i \ge 1} x_i^k, \]
where $k \in \Z_{>0}$. More precisely, $p_k$ is the element of $\Lambda$ whose projection in $\Lambda_N$ is given by
\[ p_k(x_1,\ldots,x_N) = \sum_{i=1}^N x_i^k. \]
Then $\{p_k\}_{k \ge 0}$ forms an algebraic basis of $\Lambda$. In particular, if $\rho$ is \emph{specialization}, i.e. an algebra homomorphism from $\Lambda$ to $\C$, then $\rho$ is uniquely determined by its values of $p_k$. Notationally, we write $f(\rho)$ rather than $\rho(f)$ to denote the value of $\rho$ at $f \in \Lambda$.

We say that a specialization $\rho$ is \emph{Schur-positive} if $s_{\lambda/\mu}(\rho) \ge 0$ for all $\lambda,\mu$. The primary examples of Schur positive $\rho$ are given by the evaluation map
\[ s_{\lambda/\mu}(\rho) = s_{\lambda/\mu}(a_1,\ldots,a_k,0,0,\ldots), \]
where $a_1,\ldots,a_k \ge 0$. In this case, we simply write $\rho = (a_1,\ldots,a_k,0,0,\ldots)$.

Given Schur-positive specializations $a[1], b[2], a[3], \ldots, a[2N-1], b[2N]$  and $0 < t < 1$, define the \emph{periodic Schur process} to be a measure on sequences $\vec{\lambda}=(\lambda^{(1)},\ldots,\lambda^{(2N)})$ of $2N$ partitions with a periodic boundary condition
\begin{align}\label{eq:per_bdry_conditions}
\lambda^{(2N)} \supset \lambda^{(1)} \subset \lambda^{(2)} \supset \cdots \subset \lambda^{(2N-2)} \supset \lambda^{(2N-1)} \subset \lambda^{(2N)} 
\end{align}
with probability proportional to
\[ W(\vec{\lambda}) = t^{|\lambda^{(2N)}|}
\prod_{i=1}^N s_{\lambda^{(2i-2)}/\lambda^{(2i-1)}}(a[2i-1])\, s_{\lambda^{(2i)}/\lambda^{(2i-1)}}(b[2i]),\]
where $s_{\lambda/\mu}$'s are the skew Schur functions, and here and later we take the indices modulo $2N$ if needed. The partition function $\sum_{\vec{\lambda}} W(\vec{\lambda})$ can be explicitly computed (see~\cite[Proposition 1.1]{Bor07}) and is convergent under the assumption that $|t| < 1$ and the series 
\begin{equation}\label{eq:conv_abs}
     \sum_{n \ge 1} n p_n(a[k]) p_n(b[l])
\end{equation}
is absolutely convergent for $k=1,3,\ldots,2N-1$ and $l=2,4,\ldots,2N$ (see~\cite[Remark 1.2]{Bor07}).

Let $\Z' = \Z + \tfrac{1}{2}$. To any sequence of partitions $\vec{\lambda}=(\lambda^{(1)},\ldots,\lambda^{(2N)})$ as in \eqref{eq:per_bdry_conditions}, we associate the point configuration with coordinates of points given by the set
\begin{align} \label{eq:unshift_pt_process}
\bigcup_{\tau=1}^{2N} \{(\tau, \lambda^{(\tau)}_i - i + \tfrac{1}{2}): i \geq 1\}.
\end{align}
We therefore identify the periodic Schur process with a measure on such point configurations. Note that this set uniquely determines $\vec{\lambda}$.
Correlation functions of the periodic Schur process are defined as 
\begin{align} \label{def:rho}
\rho_k(\tau_1,x_1;\ldots;\tau_k,x_k) = \operatorname{Pr}\left( x_j\in \{ \lambda_i^{(\tau_j)} - i + \tfrac{1}{2} \}_{i\ge 1} \, \text{ for all } \, j=1,\ldots,k\right).
\end{align}

While the periodic Schur process is not a determinantal point process, if we mix the periodic Schur process with an independent discrete Gaussian shift variable, we obtain a point process which is determinantal. We now introduce this model. It will be convenient to work with another Jacobi theta function $\theta_3(z;t)$, defined by 
\begin{align}\label{eq:theta3}
\theta_3(z;t) :=\sum_{m=-\infty}^{\infty} z^m t^{m^2/2} = (t;t)_\infty \prod_{n=\frac{1}{2},\frac{3}{2},\ldots} (1 + t^n z)(1 + t^n/z).
\end{align}
Note that $\theta_1(z;t) = -z^{-1/2}\theta_3(-zt^{-1/2};t).$

By \Cref{def:disc_gauss} and~\eqref{eq:theta3}, for an auxiliary parameter $u > 0$ the discrete Gaussian
\[
S \sim \cN_{\operatorname{discrete}}\left(\frac{\log u}{\log t}, \frac{|\log t|}{2}\right).  
\]
has distribution explicitly given by
\begin{align} \label{eq:v_distr}
\operatorname{Pr}(S=x) = \frac{1}{\theta_3(u;t)} u^x t^{x^2/2}, \quad 
x\in\Z.
\end{align} 
The fact that the normalizing constant of a discrete Gaussian is given by a Jacobi theta function, and hence has an alternate product form, was observed and used by~\cite{kemp1997characterizations}, see also more recent works~\cite{agostini2019discrete,szablowski2001discrete} for further discussion.

If $S$ and $\vec{\lambda}$ are independent, then $(\vec{\lambda},S)$, viewed as the point process
\begin{align} \label{eq:pt_process}
\{ (1,S + \lambda_i^{(1)} - i + \tfrac{1}{2}) \}_{i\ge 1} \cup \cdots \cup \{(2N,S + \lambda_i^{(2N)} - i + \tfrac{1}{2})\}_{i \ge 1},
\end{align}
is a determinantal point process on $\{1,\ldots,N\} \times \Z'$  due to~\cite{Bor07} which is called the \emph{shift-mixed periodic Schur process} (all points are shifted by an independent integer-valued random variable $S$). The correlation functions are defined as before
\begin{align} \label{def:rho_shift}
\rho_k^{\operatorname{shift},u}(\tau_1,x_1;\ldots;\tau_k,x_k) = 
\operatorname{Pr}\left( x_j\in \{ S + \lambda_i^{(\tau_j)} - i + \tfrac{1}{2} \}_{i\ge 1} \, \text{ for all } \, j=1,\ldots,k\right),
\end{align}
and are given explicitly due to Borodin~\cite{Bor07}. We introduce some notation before stating the formula.

Given Schur-positive specializations $a[1], b[2], a[3], \ldots, a[2N-1], b[2N]$, for each $i \in \{1,2,\ldots,N\}$ and $k \in \Z$ let us define specializations~$\tilde{a}[\,\cdot\,]$ and~$\tilde{b}[\,\cdot\,]$ by
\begin{align}
\tilde{a}[2i-1 + kN] = a[2i-1] t^k, \quad  \tilde{b}[2i + kN] = b[2i] t^{-k},
\end{align}
where for specializations $\rho$ and $w \in \C$ we define $\rho \cdot w$ by 
\[
p_\ell(\rho \cdot w) = w^\ell p_\ell(\rho).
\]

Let $p_k$ be the Newton power sums. Following~\cite{Bor07} let us define
\begin{align}
&H(x,z) := \exp\left( \sum_{k=1}^\infty \frac{1}{k} p_k(x) z^k \right); \\
&F(\tau,z) 
:= \frac{\prod_{j \le \tau} H(\tilde{b}[i],z)}{\prod_{i > \tau} H(\tilde{a}[i],z^{-1})}, \quad \quad j \in 2\Z, \quad i \in 2\Z + 1 \label{eq:def_F}
\end{align}
where $x,z$ are formal variables and $\tau \in \Z$. Note that $F$ depends implicitly on a sequence of specializations. We may now state the contour integral formula for the correlation kernel of the shift-mixed periodic Schur process due to~\cite{Bor07}.

\begin{theorem}[{\cite{Bor07}}] \label{thm:kernel}
Let $t \in (0,1), u \in \R_{>0}$ and $a[k], b[l], 1 \le k,l,\le N$ be specializations such that $p_n(a[k]), p_n(b[l]) = O(R^n)$ for some $0 < R < 1$ and for $1 \le k,l \le N$.
Then the correlation functions of the corresponding shift-mixed periodic Schur process are determinantal, 
\[ \rho_k^{\operatorname{shift},u}(\tau_1,x_1;\ldots;\tau_k,x_k) = \det[K(\tau_i,x_i;\tau_j,x_j)]_{i,j=1}^k, \]
where the correlation kernel $K(\sigma,x;\tau,y)$ is defined by the generating series 
\begin{align} \label{eq:generating_function}
\sum_{x,y \in \Z'} K(\sigma,x;\tau,y) \zeta^x \eta^y = \begin{cases}
\displaystyle \frac{F(\sigma,\zeta)}{F(\tau,\eta^{-1})} \sum_{m\in \Z'} \frac{(\zeta\eta)^m}{1 + (u t^m)^{-1}}, & \sigma \le \tau, \\
\displaystyle -\frac{F(\sigma,\zeta)}{F(\tau,\eta^{-1})} \sum_{m\in \Z'} \frac{(\zeta\eta)^m}{1 + u t^m}, & \sigma > \tau.
\end{cases}
\end{align}
\end{theorem}

\begin{remark} \label{remark:generating_function}
The relation~\eqref{eq:generating_function} should be interpreted with $\zeta$ and $\eta$ as formal variables. We may also view~\eqref{eq:generating_function} as an analytic function via the equality
\[ (t;t)_\infty^3 \frac{-\sqrt{z} \theta_3(uz;t)}{\theta_3(-ut^{-\frac{1}{2}};t) \theta_3(u;t)} = \begin{cases}
\displaystyle\sum_{m \in \Z'} \frac{z^m}{1 + (ut^m)^{-1}}, & 1 < |z| < |t|^{-1}, \\
\displaystyle- \sum_{m \in \Z'} \frac{z^m}{1 + ut^m}, & |t| < |z| < 1,
\end{cases}\]
see~\cite[Remark 2.4]{Bor07}. Note also that the condition $p_n(a[k]), p_n(b[l]) = O(R^n)$ is given to ensure that the sum \eqref{eq:conv_abs} converges absolutely, as above.
\end{remark}

\begin{proof}[Proof of \Cref{thm:kernel}]
In our notation, \cite[Theorem 2.2]{Bor07} gives a formula for the correlation kernel of the point process
\[ \{(2,S + \lambda_i^{(2)} - i + \tfrac{1}{2})\}_{i \ge 1} \cup \{(4,S + \lambda_i^{(4)} - i + \tfrac{1}{2})\}_{i \ge 1} \cup \cdots \cup \{(2N,S + \lambda_i^{(2N)} - i + \tfrac{1}{2})\}_{i \ge 1}. \]
In words, this is the sub-point process of~\eqref{eq:pt_process} corresponding to $\lambda^{(j)}$ for even $j$. To obtain the correlation kernel on the full point process~\eqref{eq:pt_process}, we need to degenerate from a larger shift-mixed periodic Schur process in order to apply~\cite[Theorem 2.2]{Bor07}. With this in mind, we introduce the following:

Let $a[1],a[2],\ldots,a[2N], b[1],\ldots,b[2N]$ be Schur-positive specializations, and define $\tilde{a}[i],\tilde{b}[j]$ as above for $i,j \in \Z$.

Consider the shift-mixed periodic Schur process which we label as
\[ \tilde{\lambda}^{(2N)} = \tilde{\lambda}^{(0)} \supset \mu^{(1)} \subset \tilde{\lambda}^{(1)} \supset \cdots \subset \tilde{\lambda}^{(2N-1)} \supset \mu^{(2N)} \subset \tilde{\lambda}^{(2N)} = \tilde{\lambda}^{(0)} \]
where the probability on $(\vec{\lambda},\vec{\mu},S)$ is proportional to
\[ u^{S} t^{S^2/2}  \cdot t^{|\tilde{\lambda}^{(0)}|} \prod_{i=1}^{2N} s_{\tilde{\lambda}^{(i-1)}/\mu^{(i)}}(a[i]) s_{\tilde{\lambda}^{(i)}/\mu^{(i)}}(b[i]). \]
By~\cite[Theorem 2.2]{Bor07}, the point process
\begin{align} \label{extended_pt_process}
\{(1,S + \wt{\lambda}_i^{(1)} - i + \tfrac{1}{2})\}_{i \ge 1} \cup \cdots \cup \{(2N,S + \tilde{\lambda}_i^{(2N)} - i + \tfrac{1}{2})\}_{i \ge 1}
\end{align}
is determinantal, and the generating series of the correlation kernel $\wt{K}(\sigma,x;\tau,y)$ is given by
\[ \sum_{x,y \in \Z'} \wt{K}(\sigma,x;\tau,y) \zeta^x \eta^y = \begin{cases}
\displaystyle \frac{\wt{F}(\sigma,\zeta)}{\wt{F}(\tau,\eta^{-1})} \sum_{m\in \Z'} \frac{(\zeta\eta)^m}{1 + (z t^m)^{-1}}, & \sigma \le \tau, \\
\displaystyle -\frac{\wt{F}(\sigma,\zeta)}{\wt{F}(\tau,\eta^{-1})} \sum_{m\in \Z'} \frac{(\zeta\eta)^m}{1 + z t^m}, & \sigma > \tau.
\end{cases} \]
where
\[ \wt{F}(\tau,z) = \frac{\prod_{i \le \tau} H(\tilde{b}[i],z)}{\prod_{i > \tau} H(\tilde{a}[i],z^{-1})}, \quad \quad i \in \Z.  \]
If we take $a[i]$ for even $i$ and $b[j]$ for odd $j$ to be trivial, then~\eqref{extended_pt_process} degenerates to~\eqref{eq:pt_process} and~$\wt{F}(\tau,z)$ to~$F(\tau,z)$. This completes the proof.
\end{proof}

\subsection{Lozenge tilings}\label{Sec:Dimers} 
We focus on the case where the specializations are given by a single-variable alpha specialization. In other words, $a[i],b[j]$ are the specializations corresponding to evaluation at $(a_i,0,\ldots,0)$, $(b_j,0,\ldots,0)$ respectively for some $a_i,b_j > 0$. With this choice, $\lambda^{(2i-1)} \prec \lambda^{(2i)} \succ \lambda^{(2i+1)}$ for each $i$, so the periodic Schur process is a measure on the set
\begin{align}
    \cP^{(2N)} := \{(\lambda^{(1)},\ldots,\lambda^{(2N)}) \in \Y^{2N}: \lambda^{(1)} \prec \lambda^{(2)} \succ \cdots \prec \lambda^{(2N)} \succ \lambda^{(1)} \}.
\end{align}
We will typically denote elements of $\cP^{(2N)}$ by $\vec{\lambda}$, and often identify them with point configurations on $\{1,\ldots,2N\} \times \Z'$ as in~\eqref{eq:unshift_pt_process}. These may also be identified with tilings of an infinite cylinder by lozenges of three types $\lloz,\rloz,\hloz$. Let us introduce this perspective in greater detail.

Consider the triangular lattice on a cylinder with finite circumference and biinfinite vertical height, such that each triangle has one vertical edge. The vertical edges, viewed as line segments, are all contained in a finite union of disjoint parallel lines along the cylinder; let the number of such lines be $2N$. A lozenge is a union of two triangles sharing an edge in this lattice, which can be of three types $\lloz,\rloz,\hloz$. Note that a tiling of a domain by lozenges of three types is determined by the positions of its horizontal lozenges. Let $\mathcal{T}_N$ denote the set of tilings in this domain such that all lozenges are eventually of type $\hloz$ sufficiently far down, and sufficiently far up all lozenges are alternating types $\lloz$ and $\rloz$ in odd and even columns respectively, see \Cref{fig:loz_empt}. If we view a point configuration on $\{1,\ldots,2N\} \times \Z'$ as in~\eqref{eq:unshift_pt_process} as giving the positions of horizontal lozenges, with the horizontal lozenges having coordinates as in \Cref{fig:loz_empt}, this determines a lozenge tiling on this cylinder. Thus any measure on interlacing point configurations on $\{1,\ldots,2N\} \times \Z'$, in particular the periodic Schur process with single-variable alpha specializations and its shift-mixed version, defines a measure on $\cT_N$.

\begin{figure}[ht]
    \centering
    \includegraphics[scale=0.3]{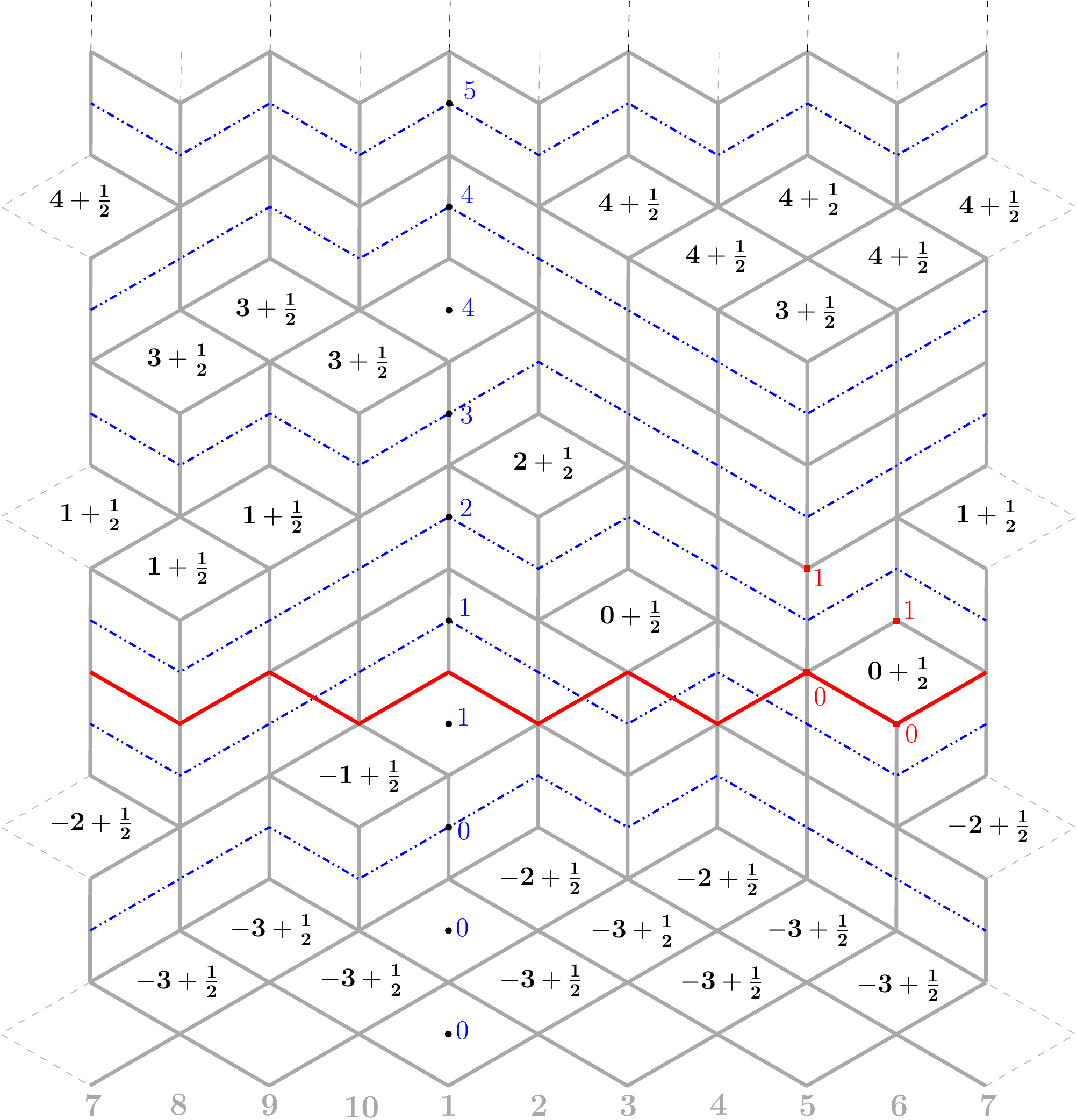}
    $\quad\quad\quad\quad\quad\quad$
     \includegraphics[scale=0.3]{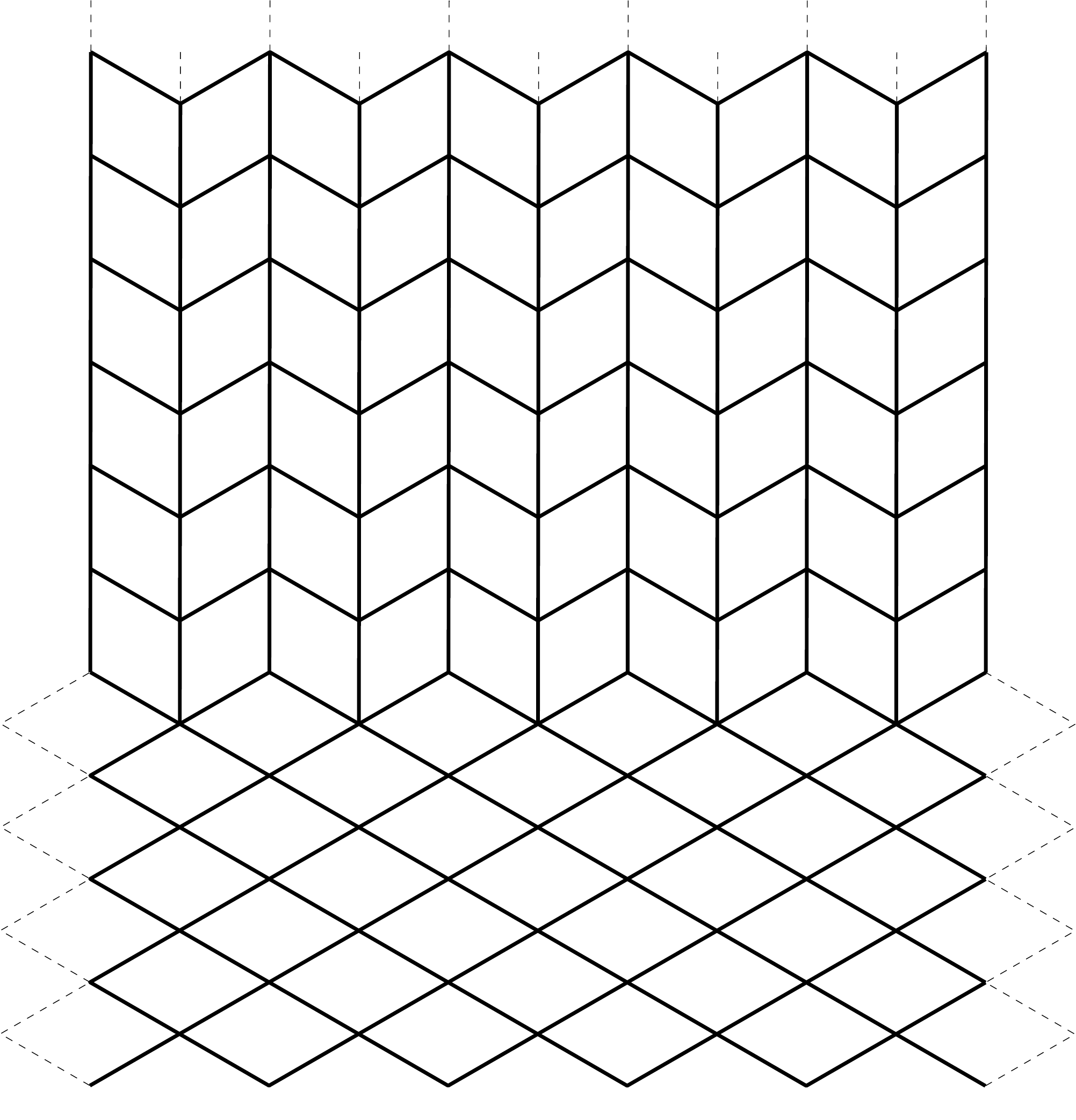}
    \caption{An example of a lozenge tiling of a cylinder (with $2N$ columns labeled in grey) and an empty room configuration (on the right). This example corresponds to $S=0.$ The vertical positions $x_i^j$ of the horizontal lozenges are shown in black, to get a corresponding pair $(\vec{\lambda},S)$ one should set $\lambda_i^{(j)}=x^j_i-\tfrac12-S+i$. We define coordinates of a horizontal lozenge at a position $x_i^j$ in column $j$ by $(j,x_i^j)$. The level lines of the height function shown in blue; the values of the height function are shown (in blue) in the first column. The red line on the left indicates the $y=0$ `axis', and the red $0$'s and $1$'s indicate points of $y$-coordinate $0$ and $1$ to highlight this.
    }
    \label{fig:loz_empt}
\end{figure}

However, the measure coming from the periodic Schur process without shift-mixing is in fact supported on a specific subset of $\cT_N$, as not all elements of $\cT_N$ correspond to elements of $\cP^{(2N)}$. The set $\cT_N$ is in fact in bijection with shifted cylindric partitions $\cP^{(2N)} \times \Z$, for reasons we now describe. Define the \emph{empty room} $M_{\operatorname{empty}} \in \cT_N$ to be the tiling in which all horizontal lozenges with negative vertical coordinate are present and no others, equivalently the tiling corresponding to the sequence of zero partitions $\vec{\lambda} = (\vec{0},\ldots,\vec{0}) \in \cP^{(2N)}$, shown on the right in \Cref{fig:loz_empt}. Then for each $M \in \cT_N$, there exists a constant $S(M)$ such that for any column $C$ one has
\[|\{\hloz\in (M \cap C) \,|\,\hloz\not\in (M_{\operatorname{empty}} \cap C )\}|-|\{\hloz\in (M_{\operatorname{empty}} \cap C) \,|\,\hloz\not\in (M \cap C) \}|=S(M)\]
independent of $C$. We invite the reader to check this on several columns of the tiling on the left of \Cref{fig:loz_empt}, for which $S(M) = 0$. If instead one considered the tiling given by translating this tiling up by $k$ units, the result would be a tiling with $S(M) = k$. Given a tiling $M$, viewed as a point configuration $\{x_i^j\}_{(i,j) \in \{1,\ldots,2N\} \times \Z_{\geq 0}}$ given by the positions $x_i^j$ of horizontal lozenges in the $j\tth$ column, there is a unique $\vec{\lambda} = (\lambda^{(1)},\ldots,\lambda^{(2N)}) \in \cP^{(2N)}$ such that 
\[
\lambda_i^{(j)} - i + \tfrac{1}{2} + S(M) = x_i^j
\]
for every $i,j$, so $S(M)$ is identified with the shift $S$ of an element $(\vec{\lambda},S) \in \cP^{(2N)} \times \Z$. Hence the map $M \mapsto (\vec{\lambda},S(M))$ defines a map $ \cT_N \to \cP^{(2N)} \times \Z$. Conversely, any $(\vec{\lambda},S) \in \cP^{(2N)} \times \Z$ specifies the positions of horizontal lozenges of an element of $\cT_N$ by~\eqref{eq:pt_process}. This defines a bijection $ \cT_N \leftrightarrow \cP^{(2N)} \times \Z$. 

Similarly, letting 
\[
\cT_N^0 := \{M \in \cT_N: S(M) = 0\},
\]
the above yields a bijection
\begin{align*}
    \cT_N^0 & \leftrightarrow \cP^{(2N)}.
\end{align*}
Hence the periodic Schur process with single-variable alpha specializations corresponds to a measure on $\cT_N^0$, while the shift-mixed periodic Schur process yields a measure on $\cT_N$. From the perspective of lozenge tilings of the cylinder, the set $\cT_N$, and hence the shift-mixed periodic Schur process, are arguably the more natural objects.

In what follows we often view elements of $\cT_N$ through their height functions, defined on midpoints of vertical edges of the triangular lattice, see \Cref{fig:loz_empt}. 
\begin{definition}\label{def:height_fn}
Given a fixed lozenge tiling $M \in \cT_N$, the corresponding \emph{height function} $h: \{1,\ldots,2N\} \times \Z' \to \Z_{\geq 0}$ is defined by
\begin{align} \label{eq:height_function}
h(\tau^\sharp,y^\sharp) := \sum_{x^\sharp\in\mathbb{Z}': x^\sharp < y^\sharp} \bbone[ \mbox{there is no lozenge of type $\hloz$ at $(\tau^\sharp,x^\sharp)$}],
\end{align}
where the coordinates of lozenges are as in \Cref{fig:loz_empt}.
\end{definition}

By the above discussion, the height function has the property that as $y^\sharp \to \infty$, $h(\tau^\sharp,-y^\sharp)$ stabilizes to~$0$ and $h(\tau^\sharp,y^\sharp) - y^\sharp - 1/2$ stabilizes to $S(M)$; this in fact gives an alternative definition of $S(M)$, which was how we introduced $S(M)$ in \Cref{sec:intro}. 

We now show the fact mentioned in \Cref{sec:intro} that the shift-mixed periodic Schur process with single-variable alpha specializations is equivalent to a local dimer model. Each lozenge is just a pair of adjacent equilateral triangles on a triangular lattice, and so we may identify these triangles with vertices of a hexagonal lattice on a cylinder, and their edges with the edges of this graph. Under this identification, $\cT_N$ corresponds to the set $\cM$ of dimer covers such that sufficiently far down on the cylinder every dimer is horizontal, and sufficiently far up, all dimers are of the other two diagonal types, northwest and northeast in odd and even columns respectively. The `empty room' corresponds to the dimer cover $M_{\empty}$ in which all horizontal edges~$e(\sigma^\sharp,x^\sharp)$ with $x^\sharp<0$, in the coordinates of \Cref{fig:dimer_weights}, are present, and all other edges are northwest and northeast in odd and even columns respectively. The dimer covers $\cM$ may be equivalently defined as those which differ from $M_{\empty}$ by a finite number of edges.

\begin{figure}[ht]
    \centering
    \includegraphics[scale=0.8]{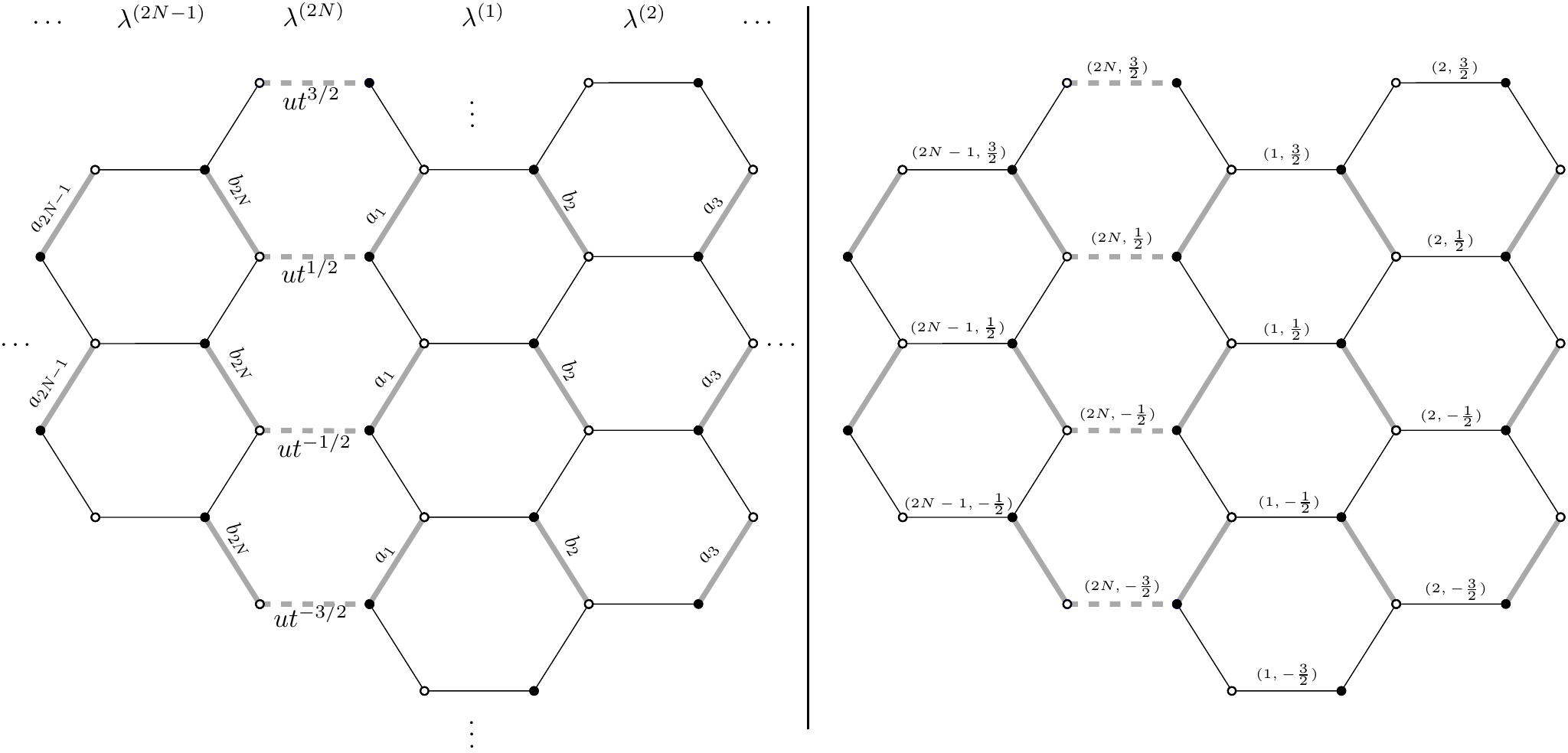}
    \caption{(Left) Dimer weights on honeycomb lattice; all weights not shown are equal to~$1$. (Right) Coordinates of horizontal edges on honeycomb lattice.}
    \label{fig:dimer_weights}
\end{figure}

Let $a_1,b_2,\ldots,a_{2N-1},b_{2N},t\in(0,1)$ and $u\in\mathbb{R}_{>0}$. Denote by $G_{2N} = G_{2N}(a_1,b_2,\ldots,a_{2N-1},b_{2N},t,u)$ the bipartite graph with edge weights as in \Cref{fig:dimer_weights}, and define the weight of a dimer cover $M \in \mathcal{M}$ to be
\begin{equation}\label{eq:weights}
    \nu(M) = \prod_{\text{edges }e \in G_{2N}} \nu(e)^{\bbone(e \in M) - \bbone(e \in M_{\operatorname{empty}})}.
\end{equation}
Given these weights, we define a probability measure on $\cM$ by $\Pr(M) = \frac{1}{Z}\nu(M)$ where 
\[Z = Z(a_1,a_2,\ldots,a_{2N-1},b_{2N},t,u) = \sum_{M \in \cM} \nu(M).\] 

\begin{remark}
Typically, probability measures in dimer models are defined in terms of weights $\prod_{e \in M}\operatorname{wt}(e)$ for some edge weights $\operatorname{wt}(e)$, but such probability measures are unchanged if all weights are multiplied by the same nonzero constant. Informally, our weights~\eqref{eq:weights} should be viewed as the weights
\[
\prod_{e \in M}\nu(e),
\]
which are divergent infinite products, regularized by the overall constant
\[
\prod_{e \in M_{\operatorname{empty}}}\nu(e)
\]
which is also a divergent infinite product.
\end{remark}

It is easy to check from \Cref{fig:dimer_weights} and~\eqref{eq:skew} that the weight $\nu(M)$ of a dimer configuration corresponding to $(\vec{\lambda},S) \in \cP^{(2N)} \times \Z$ is exactly
\begin{align*}
& u^{S}t^{|\lambda^{(2N)}|+S^2/2} s_{\lambda^{(2N)}/\lambda^{(1)}}(a_1)s_{\lambda^{(2)}/\lambda^{(1)}}(b_2) \cdots s_{\lambda^{(2N)}/\lambda^{(2N-1)}}(b_{2N}) \\
& \quad \quad =  u^{S}t^{|\lambda^{(2N)}|+S^2/2} \prod_{i=1}^N a_{2i-1}^{|\lambda^{(2i - 2)}| - |\lambda^{(2i - 1)}|} b_{2i}^{|\lambda^{(2i)}| - |\lambda^{(2i-1)}|}.
\end{align*}
Hence the dimer model on the hexagonal lattice on a cylinder with edge weights shown on~\Cref{fig:dimer_weights} defines the same measure on cylindric partitions as the shift-mixed periodic Schur process with the corresponding parameters. We note that for dimer models on finite domains, Kasteleyn theory yields that the edges of a dimer cover form a determinantal point process~\cite{kenyon1997local}, see also~\cite{kenyon2009lectures,gorin_2021}. 

\subsection{{$q^{\mvol}$} Measure on cylindric partitions}

In the remainder of the paper, we will be concerned with the $q^{\mvol}$ measure on cylindric partitions and its shift-mixed
version, which turns out to be very special cases of the (shift-mixed) periodic Schur process with single-variable alpha specializations. We will identify cylindric partitions $M \in \cT_N$ with sequences of partitions $(\vec{\lambda},S) \in \cP^{(2N)} \times \Z$, and identify $M \in \cT_N^0$ with $\vec{\lambda} \in \cP^{(2N)}$. We define the volume of a cylindric partition by 
\[ \operatorname{vol}(\vec{\lambda},S) := \sum_{i=1}^{2N} \left(|\lambda^{(i)}| + S^2/2 \right) \]
and write 
\[
\opvol(\vec{\lambda}) := \opvol(\vec{\lambda},0) =  \sum_{i=1}^{2N} |\lambda^{(i)}|
\]
for the volume of a cylindric partition with zero shift.

\begin{definition}
Let $q \in (0,1)$. The \emph{$q^{\mvol}$ measure} is the probability measure on $\cP^{(2N)}$ with
\[
\Pr(\vec{\lambda}) \propto q^{\mvol(\vec{\lambda})}.
\]
For $u \in \R_{>0}$, the \emph{shift-mixed $q^{\mvol}$ measure} with shift parameter~$u$ is the probability measure on $\cP^{(2N)} \times \Z$ with
\[
\Pr(\vec{\lambda},S) \propto u^S q^{\mvol(\vec{\lambda},S)}.
\]
\end{definition}

The usual and shift-mixed $q^\mvol$ measures are special cases of the usual and shift-mixed periodic Schur process. The analogue of this fact for ordinary plane partitions, and the tools it furnished for studying the $q^{\mvol}$ measure, was in fact the first application of Schur processes when they were introduced in~\cite{Oko00}.

\begin{lemma}\label{thm:q^vol}
Fix $q \in (0,1)$ and $u \in \R_{>0}$ as above. Then the $q^{\mbox{vol}}$ measure is the periodic Schur process with specializations
\[ a_{2i-1} = q^{2i - \frac{3}{2}}, \quad b_{2i} = q^{-(2i - \frac{1}{2})}, \quad t = q^{2N} \quad \quad i = 1,\ldots,N, \]
and the shift-mixed $q^{\mbox{vol}}$ measure with shift parameter $u$ is the shift-mixed periodic Schur process with the above specializations and shift parameter $u$.
\end{lemma}
\begin{proof}
Direct computation by the branching rule~\eqref{eq:skew} for Schur functions.
\end{proof}

In light of the previous subsection, it follows from \Cref{thm:q^vol} that the shift-mixed $q^{\mvol}$ measure is also given by a dimer model, though we do not use this fact. A fact that we will use, however, is that with the specializations in \Cref{thm:q^vol} one has
\begin{align} \label{eq:F_function}
F(\tau,z) = \frac{\prod_{i > \tau} (1 - q^i z^{-1})}{\prod_{j < \tau} (1 - q^{-j} z)}, \quad \quad i \in 2\Z + \tfrac{1}{2}, \quad j \in 2\Z - \tfrac{1}{2}
\end{align}
where $F$ is as in~\eqref{eq:def_F}.

\section{Moment formula}\label{sec:moment_formula}

In this section we derive a contour integral formula for certain Laplace transform observables of the~$q^{\mvol}$ height function, see \Cref{thm:analytic_r_moments}, using the general cylindric Schur process correlation kernel given previously in \Cref{thm:kernel}.
For a partition $\lambda$ and an integer $S$, define
\[
\tcF_r(\lambda,S) = \sum_{i \geq 1} r^{\lambda_i - i + 1 + S}
\]
and
\[
\tcF_r(\lambda) = \sum_{i \geq 1} r^{\lambda_i - i + 1 }
\]
if no $S$ is present in the argument.

\begin{proposition} \label{thm:shifted_observable_formula}

Let $u,t$ be real with $u \neq 0, -t^{\pm \frac{2\ell+1}{2}}$ for $\ell \in \Z$ and $t\in(0,1)$. Let {$a[1],b[2],\ldots,a[2N-1],$} $b[2N]$ be each finite sequences of real parameters in $(0,1)$. Then there exists a constant $\delta > 0$ in terms of these parameters such that, for all $r_1,\ldots,r_n \in \C$ with $1 < |r_i| < 1+\delta$,
\begin{align}\label{eq:shift-mixed_formal_moments}
\begin{split}
& \E_{\operatorname{shift}}\left[ \prod_{i=1}^n \tcF_{r_i}(\lambda^{(\tau_i)},S) \right] =  \left( \prod_{i=1}^n \frac{(t;t)_\infty^2}{(r_i t;t)_\infty (r_i^{-1}t;t)_\infty (1 - r_i^{-1})} \right) \\
& \quad \quad \times \frac{\theta_3(u \prod_{i=1}^n r_i; t)}{\theta_3(u;t)} \frac{1}{(2 \pi \bi)^n} \oint \cdots \oint \left(\prod_{1 \le i < j \le n} \frac{\theta_1(z_i/z_j;t) \theta_1(r_jr_i^{-1}z_i/z_j;t)}{\theta_1(r_jz_i/z_j;t) \theta_1(r_i^{-1}z_i/z_j;t)} \prod_{i=1}^n \frac{F(\tau_i,z_i)}{F(\tau_i,r_i^{-1}z_i)}\right) \frac{dz_i}{z_i}
\end{split}
\end{align}
where the contours are chosen such that 
\begin{align}\label{eq:contour_conditions}
\begin{split}
    \max(|r_j^{-1}z_j|, |r_i z_j|) < |z_i| < t^{-1} \min(|r_j^{-1}z_j|,|r_iz_j|) \quad \quad \quad \text{for all }1 \leq i < j \leq n\\
    |r_i| \max_{\substack{j > \tau_i \\ a \in a[j]}} a  < |z_i| < \min_{\substack{j \leq  \tau_i\\ b \in b[j]}}  b^{-1} \quad \quad \quad \text{for all }1 \leq i \leq n.
    \end{split}
\end{align}
\end{proposition}

\begin{remark}
    \noindent(i) The reason for the condition $1 < |r_i| < 1+\delta$ is that if $|r_i| \leq 1$ the series defining $\tcF$ does not converge, and if the $|r_i|$ are too large then there may not exist contours satisfying~\eqref{eq:contour_conditions}. It would be possible to write more explicit sufficient conditions on how large $\delta$ can be in terms of the parameters based on conditions used in the proof, but since we only need existence of $\delta$ we have avoided this for simplicity.\\
    \noindent(ii) The reason for the conditions~\eqref{eq:contour_conditions} is to ensure that certain Laurent series expansions encountered in the proof are valid.\\
    \noindent(iii) In the condition $\max(|r_j^{-1}z_j|, |r_i z_j|) < |z_i| < t^{-1} \min(|r_j^{-1}z_j|,|r_iz_j|) $ it is readily apparent that the maximum and minimum are always achieved by the second and first argument respectively when $|r_i|,|r_j| > 1$. The reason for this apparent redundancy is that later, in \Cref{thm:observable_formula} we will analytically continue to $r_i$ with $|r_i| \leq 1$, and the other conditions will become relevant.
\end{remark}

\begin{proof}[Proof of {\Cref{thm:shifted_observable_formula}}]
We first prove the statement for the shift-mixed Schur process in the restricted setting $1 \leq \tau_1 < \cdots < \tau_n \leq 2N$. It is straightforward from the definition of $\tcF$ that
\[
\tcF_{r_i}(\lambda^{(\tau_i)},S) = \sum_{x_i \in \Z'} \bbone(\tau_i,x_i) r_i^{x_i+1/2}
\]
where we use the shorthand $\bbone(\tau_i,x_i) = \bbone(\text{there exists a horizontal lozenge at }(\tau_i,x_i))$. Then 
\begin{align}\label{eq:observable_to_cor_fn}
\begin{split}
\E_{\operatorname{shift}}\left[ \prod_{i=1}^n \tcF_{r_i}(\lambda^{(\tau_i)},S) \right] &= \sum_{x_1,\ldots,x_n \in \Z'} \E_{\operatorname{shift}}\left[\prod_{i=1}^n r_i^{x_i+1/2}\bbone(\tau_i,x_i)\right]\\ 
&= \sum_{x_1,\ldots,x_n \in \Z'} \prod_{i=1}^n r_i^{x_i+1/2} \rho^{shift}_n(\tau_1,x_1; \ldots; \tau_n,x_n)
\end{split}
\end{align}
where in the first equality we have used that the sum is absolutely convergent.
 
Expanding the correlation function in terms of the correlation kernel by \Cref{thm:kernel}, this is equal to 
\begin{align*}
& \sum_{x_1,\ldots,x_n \in \Z'} \prod_{i=1}^n r_i^{x_i+1/2} \det\left(K(\tau_i,x_i; \tau_j,x_j)\right)_{1 \leq i,j \leq n} \\
&= \sum_{\sigma \in S_n} \sgn(\sigma) \sum_{x_1,\ldots,x_n \in \Z'} \prod_{i=1}^n r_i^{x_i+1/2} K(\tau_i,x_i; \tau_{\sigma(i)},x_{\sigma(i)}) \\
 &= \sqrt{r_1 \cdots r_n}\sum_{\sigma \in S_n} \sgn(\sigma) \left(\prod_{i=1}^n \sum_{y_i,x_i \in \Z'} z_i^{x_i} (r_{\sigma(i)}z_{\sigma(i)}^{-1})^{y_{\sigma(i)}} K(\tau_i,x_i; \tau_{\sigma(i)}, y_{\sigma(i)})\right)[z_1^0 \cdots z_n^0]
 \end{align*}
 where $[z_1^0 \cdots z_n^0]$ denotes the constant term in the Laurent series. By \Cref{thm:kernel} and \Cref{remark:generating_function}, we have that if $\zeta,\eta$ satisfy the conditions
 \begin{equation}\label{eq:anal_conditions_theta}
 \begin{cases}
 1 < |\zeta \eta| < |t|^{-1} & \sigma \leq \tau \\
 |t| < |\zeta \eta| < 1 & \sigma > \tau 
 \end{cases}
  \end{equation}
  and 
\begin{align}\label{eq:zeta_F_restriction}
 |\zeta| < b^{-1} \text{ for all }b \in b[j], j \leq \sigma \\
 |\eta| < a^{-1} \text{ for all }a \in a[j], j > \tau \label{eq:eta_F_restriction}
\end{align}
then
 \[
 \sum_{x,y \in \Z'} K(\sigma,x;\tau,y)\zeta^x \eta^y = -\frac{F(\sigma,\zeta)}{F(\tau,\eta^{-1})} (t;t)_\infty^3 \frac{\sqrt{\zeta \eta} \theta_3(u\zeta \eta;t)}{\theta_3(-t^{-1/2}\zeta \eta;t) \theta_3(u;t)}.
 \]
 The analytic conditions above are required for certain Laurent series expansions to be valid. Specifically,~\eqref{eq:anal_conditions_theta} guarantees the Laurent series expansion of \Cref{remark:generating_function} for 
 \[
 \frac{\sqrt{\zeta \eta} \theta_3(u\zeta \eta;t)}{\theta_3(-t^{-1/2}\zeta \eta;t) \theta_3(u;t)}
 \]
while~\eqref{eq:zeta_F_restriction} (resp.~\eqref{eq:eta_F_restriction}) guarantees that $\zeta$ (resp. $\eta^{-1}$) lies outside a circle containing the poles of $F(\sigma,\zeta)$ (resp. $F(\tau,\eta^{-1})^{-1})$).

Substituting this in with $\zeta = z_i, \eta = r_{\sigma(i)}z_{\sigma(i)}^{-1}$ for $z_i$ complex numbers satisfying the hypotheses~\eqref{eq:contour_conditions} on $|z_i|$, one has that $\zeta,\eta$ satisfy the conditions above. For $\delta$ sufficiently small so that $|r_i|$ is sufficiently close to $1$, it is easy to see that such $|z_i|$ must exist---though again we note as a caution that if the $r_i$ are too large this may \emph{not} be true, hence the necessity to have $|r_i| < 1+\delta$. Hence we obtain 
 \begin{align}
 &\sum_{\sigma \in S_n} \sgn(\sigma) \frac{1}{(2 \pi \bi)^n} \oint \cdots \oint \prod_{i=1}^n \frac{F(\tau_i,z_i)}{F(\tau_{\sigma(i)},r_{\sigma(i)}^{-1}z_{\sigma(i)})} \frac{-r_i (t;t)_\infty^3\theta_3(u r_{\sigma(i)}z_i/z_{\sigma(i)};t)}{\theta_3(-t^{-1/2}r_{\sigma(i)}z_i/z_{\sigma(i)};t)\theta_3(u;t) } \frac{dz_i}{z_i} \\ \label{eq:substitute_det}
 &= (-1)^n (t;t)_\infty^{3n} \frac{1}{(2 \pi \bi)^n} \oint \cdots \oint \prod_{i=1}^n r_i \frac{F(\tau_i,z_i)}{F(\tau_i,r_i^{-1}z_i)} \det\left(\frac{\theta_3(u r_j z_i/z_j;t)}{\theta_3(-t^{-1/2}r_j z_i/z_j;t)\theta_3(u;t)}\right)_{1 \leq i,j \leq n} \frac{dz_i}{z_i}
 \end{align}
 where the contours are as in the proposition statement. 
 By the elliptic Cauchy determinant identity of Frobenius~\cite{frobenius1882ueber} as used in~\cite[p. 18]{Bor07}, we have
 \begin{align*}
     &\det\left(\frac{\theta_3(u  r_jz_i/z_j;t)}{\theta_3(-t^{-1/2}r_jz_i/z_j;t)\theta_3(u;t)}\right)_{1 \leq i,j \leq n} \\
     & \quad \quad \quad \quad 
     =  \frac{(-1)^n\theta_3(u \prod_{i=1}^n r_i; t)}{\theta_3(u;t) (r_1 \cdots r_n)^{1/2}} \prod_{1 \leq i < j \leq n} \frac{\theta_1(z_i/z_j;t) \theta_1(\frac{r_i}{r_j}z_j/z_i;t)}{\theta_1(r_jz_i/z_j;t)\theta_1(r_iz_j/z_i;t)} \prod_{i=1}^n \frac{1}{\theta_1(r_i;t)}.
 \end{align*}
Substituting this into~\eqref{eq:substitute_det} we obtain
 \begin{align}\label{eq:final_shifted_observable}
 \begin{split}
     \E_{\operatorname{shift}}\left[ \prod_{i=1}^n \tcF_{r_i}(\lambda^{(\tau_i)},S) \right] &= \frac{\theta_3(u \prod_{i=1}^n r_i; t)}{\theta_3(u;t)} \left( \prod_{i=1}^n \frac{(t;t)_\infty^3 r_i^{1/2}}{\theta_1(r_i;t)} \right) \\
     & \times \frac{1}{(2 \pi \bi)^n} \oint \cdots \oint \prod_{i=1}^n \frac{F(\tau_i,z_i)}{F(\tau_i,r_i^{-1}z_i)}  \prod_{1 \leq i < j \leq n} \frac{\theta_1(z_i/z_j;t) \theta_1(\frac{r_i}{r_j}z_j/z_i;t)}{\theta_1(r_jz_i/z_j;t)\theta_1(r_iz_j/z_i;t)} \frac{dz_i}{z_i}.
     \end{split}
 \end{align}
 The final result~\eqref{eq:shift-mixed_formal_moments} follows upon noting that 
 \[
 \frac{(t;t)_\infty^3 r_i^{1/2}}{\theta_1(r_i;t)} = \frac{(t;t)_\infty^2}{(r_it;t)_\infty (r_i^{-1}t;t)_\infty (1-r_i^{-1})}
 \]
 and applying $\theta_1(z^{-1};t) = -\theta_1(z;t)$ to half of the $\theta_1$ functions in the product over $i<j$.
Having shown~\eqref{eq:shift-mixed_formal_moments} when $\tau_1 < \ldots < \tau_n$, we illustrate how to extend to the case where $\tau_\ell = \tau_{\ell+1}$ for some $1 \leq \ell \leq n-1$ and all other $\tau_i$ are distinct. The general case $\tau_1 \leq \cdots \leq \tau_n$ follows by induction with essentially the same manipulation. Let
 \[
 \ttau_i = \ttau_i(\ell) = 
 \begin{cases}
 \tau_i & 1 \leq i \leq \ell \\
 \tau_i + 2 & \ell+1 \leq i \leq n
 \end{cases}.
 \]
 Recall that $F(\sigma,z)$ depends implicitly on the sequence of specializations $a[1],b[2],\ldots,a[2N-1],b[2N]$, and let $\tF(\sigma,z)$ be the same function defined with specializations padded with two extra zeroes at the $\ell\tth$ place, i.e. specializations $a[1],b[2],\ldots,a[\tau_\ell],0,0,b[\tau_\ell+1],\ldots,a[2N-1],b[2N]$ if $\ell$ is odd and $a[1],b[2],\ldots,b[\tau_\ell],0,0,a[\tau_\ell+1],\ldots,a[2N-1],b[2N]$ if $\ell$ is even. Since $H(0,z)=1$, it follows from the definition of $F$ that
 \[ \tF(\ttau_i,z) = F(\tau_i,z).\]
 Note that the functions $F$ are the only part of~\eqref{eq:final_shifted_observable} which depend on the sequence of specializations $a[1],b[2],\ldots,a[2N-1],b[2N]$. The proof that~\eqref{eq:shift-mixed_formal_moments} holds for $1 \leq \tau_1 < \cdots < \tau_\ell = \tau_{\ell+1} < \cdots < \tau_n \leq 2N$ is now complete upon noticing that $\lambda^{(\tau_\ell)} = \lambda^{(\tau_\ell+2)}$ when the sequence of specializations contains zeros in the places $\ell+1,\ell+2$ as above. The general case where many $\tau_i$ are equal follows by the same padding with trivial specializations argument.
 \end{proof}
 
 It is easy to deduce from \Cref{thm:shifted_observable_formula} the analogue for the unshifted Schur process.
 
 \begin{corollary}\label{thm:observable_formula}
 Let partitions $\lambda^{(i)}, 1 \leq i \leq 2N$ be distributed by the unshifted Schur process with specializations as in \Cref{thm:shifted_observable_formula}. Then 
\begin{align}\label{eq:unshifted_formal_moments}
\begin{split}
& \E\left[ \prod_{i=1}^n \tcF_{r_i}(\lambda^{(\tau_i)}) \right] =  \left( \prod_{i=1}^n \frac{(t;t)_\infty^2}{(r_i t;t)_\infty (r_i^{-1}t;t)_\infty (1 - r_i^{-1})} \right) \\
& \quad \quad \times \frac{1}{(2 \pi \bi)^n} \oint \cdots \oint  \left(\prod_{1 \le i < j \le n} \frac{\theta_1(z_i/z_j;t) \theta_1(r_jr_i^{-1}z_i/z_j;t)}{\theta_1(r_jz_i/z_j;t) \theta_1(r_i^{-1}z_i/z_j;t)} \prod_{i=1}^n \frac{F(\tau_i,z_i)}{F(\tau_i,r_i^{-1}z_i)}\right) \frac{dz_i}{z_i}
\end{split}
\end{align}
with the same contours as in \Cref{thm:shifted_observable_formula}.
 \end{corollary}
 \begin{proof}
 
By~\cite[Prop. 2.1]{Bor07} we have 
 \begin{equation*}
     \rho_n(\tau_1,x_1; \ldots; \tau_n,x_n) = (\text{constant term in $u$ of }\theta_3(u;t)\rho_n^{shift}(\tau_1,x_1; \ldots; \tau_n,x_n)).
 \end{equation*}
 Hence for any $1 \leq \tau_1 \leq \cdots \leq \tau_n \leq 2N$, by~\eqref{eq:observable_to_cor_fn} (which holds for both shift-mixed and unshifted measures) we have
\[
 \E\left[ \prod_{i=1}^n \tcF_{r_i}(\lambda^{(\tau_i)}) \right] = \left(\text{constant term in $u$ of }\theta_3(u;t)\E_{\operatorname{shift}}\left[ \prod_{i=1}^n \tcF_{r_i}(\lambda^{(\tau_i)}) \right]\right).
 \]
 Substituting~\eqref{eq:shift-mixed_formal_moments} in for $\E_{\operatorname{shift}}\left[ \prod_{i=1}^n \tcF_{r_i}(\lambda^{(\tau_i)}) \right]$ and noting that the constant term in~$u$ of~$\theta_3(u \prod_{i=1}^n r_i; t)$ is~$1$ yields~\eqref{eq:unshifted_formal_moments}.

\end{proof}

We wish to use the observables $\tcF$ to study the height function of a random $q^{\mvol}$ cylindric partition along vertical slices. However, it will be more convenient to extend the integral formula~\eqref{eq:unshifted_formal_moments} to the case when $|r_i|<1$. In this range the power series defining $\tcF$ no longer converge, but we may work around this issue by defining a slightly different version. Let 
\begin{equation}\label{eq:def_cF}
 \cF_r(\lambda) := \sum_{i=1}^{\ell(\lambda)} r^{\lambda_i - i + 1} + \frac{r^{-\ell(\lambda)}}{1 - r^{-1}}.    
\end{equation}
 
Clearly $\cF_r(\lambda) = \tcF_r(\lambda)$ when $|r|>1$, but $\cF_r(\lambda)$ makes sense when $|r| < 1$.
The following lemma shows that these observables are essentially a certain exponential transform of the height function.
\begin{lemma} \label{thm:height_function_observable}
Let $\hh(\tau,x)$ be a height function and $\cF_r$ defined as above. Then for  $\tau\in[\![2N]\!]$ and $0 < r < 1$ one has
\[ \sum_{x \in \Z'} \hh(\tau,x) r^x = - \frac{r^{1/2}}{1 - r} \cF_r(\lambda^{(\tau)}). \]
\end{lemma}
\begin{proof}
Summation by parts gives us
\[\sum_{x \in \Z'} \hh(\tau,x) r^x = 
\left(
\hh(\tau,x)\frac{r^x}{1-r}\Bigg|_{x=-\infty}
-\hh(\tau,x)\frac{r^{x+1}}{1-r}\Bigg|_{x=+\infty}
\right)
+\sum_{x \in \Z'}\frac{r^{x}}{1-r}\big(1-\bbone(\tau,x-1)\big),\]
where 
$\bbone(\tau,x)=
\bbone\big[
x=\lambda_i^{(\tau)} - i + \tfrac12 
\text{  for some } i 
\big].$ Note that for $x<0$ small enough $\hh(\tau,x)=0$ and for $r<1$  one has $\lim\limits_{x \to +\infty} h(\tau,x)r^{x+1}=0$, therefore 
\[
\hh(\tau,x)\frac{r^x}{1-r}\Bigg|_{x=-\infty}
-\hh(\tau,x)\frac{r^{x+1}}{1-r}\Bigg|_{x=\infty}
=0.
\] Let us compute the remaining term:
\begin{align*}
&\sum_{x \in \Z'}\frac{r^{x}}{1-r}\big(1-\bbone(\tau,x-1)\big)\\
&\quad\quad\quad\quad
= \frac{1}{1-r}\sum_{x=-\ell\left(\lambda^{(\tau)}\right)+\tfrac32}^\infty \big(1-\bbone(\tau,x-1)\big)r^{x}
= \frac{1}{1-r}
\left(
\sum_{x=-\ell\left(\lambda^{(\tau)}\right)+\tfrac32}^\infty r^x - \sum_{i=1}^{\ell\left(\lambda^{(\tau)}\right)} r^{\lambda_i^{(\tau)} - i + \tfrac32}
\right),
\end{align*}
where the sum starts from $-\ell(\lambda^{(\tau)}) - \tfrac{3}{2}$ because $1 - \bbone(\tau,x) = 0$ for $x \in \{\lambda_i^{(\tau)} - i + \tfrac{1}{2}\}_{i\ge 1}$ and we have $\lambda_i^{(\tau)} = 0$ for $i > \ell(\lambda^{(\tau)})$. Now,  note that 
\[\sum_{x=-\ell\left(\lambda^{(\tau)}\right)+\tfrac32}^\infty r^x 
- \sum_{i=1}^{\ell\left(\lambda^{(\tau)}\right)} r^{\lambda_i^{(\tau)} - i + \tfrac32} = -r^{\tfrac12}
\left(
\frac{r^{-\ell\left(\lambda^{(\tau)}\right)}}{1 - r^{-1}}
+\sum_{i=1}^{\ell\left(\lambda^{(\tau)}\right)} r^{\lambda^{(\tau)}_i - i + 1}
\right)\]
to conclude. \end{proof}

We conclude with a particularly nice specialization of \Cref{thm:observable_formula} to the case of the $q^{\mvol}$ model. The key observation is that
\begin{align} \label{eq:F_quotient}
\frac{F(\tau,z)}{F(\tau,q^{-2k}z)} = \prod_{\tau < i \le \tau + 2k} (1 - q^i z^{-1}) \prod_{\tau \le j < \tau + 2k} (1 - q^{-j} z), \quad \quad i \in 2\Z + \tfrac{1}{2}, \quad j \in 2\Z - \tfrac{1}{2}
\end{align}
is entire for any positive integer $k$, where we recall the expression~\eqref{eq:F_function} for $F$.

\begin{corollary}\label{thm:analytic_r_moments}
Suppose $\lambda^{(2N)} \succ \lambda^{(1)} \prec \cdots \succ \lambda^{(2N-1)} \prec \lambda^{(2N)} $ is distributed as the unshifted $q^{\mvol}$ model with $q \in (0,1)$. If $r_i = q^{2k_i}$ for some positive integer $k_i$ and $1 \le i \le n$, then
\begin{align} \label{eq:analytic_r_moments}
\begin{split}
& \E\left[ \prod_{i=1}^n \cF_{r_i}(\lambda^{(\tau_i)}) \right] = \left( \prod_{i=1}^n \frac{(t;t)_\infty^2}{(r_i t;t)_\infty (r_i^{-1}t;t)_\infty (1 - r_i^{-1})} \right) \\
& \quad \quad \times \frac{1}{(2\pi\bi)^n} \oint \cdots \oint \prod_{1 \le i < j \le n} \frac{\theta_1(z_i/z_j;t) \theta_1(r_jr_i^{-1}z_i/z_j;t)}{\theta_1(r_jz_i/z_j;t) \theta_1(r_i^{-1}z_i/z_j;t)} \prod_{i=1}^n \frac{F(\tau_i,z_i)}{F(\tau_i,r_i^{-1}z_i)} \frac{dz_i}{z_i}
\end{split}
\end{align}
where the contours are concentric positively oriented circles around $0$ such that
\[ r_j^{-1}|z_j| < |z_i| < |t|^{-1} r_i|z_j|, \quad \quad 1\le i < j \le n. \] 
\end{corollary}

\begin{remark}
While the integral formula in \Cref{thm:observable_formula,thm:analytic_r_moments} look identical, there is one key difference. The main point of \Cref{thm:analytic_r_moments} is to choose $r_i$ with special values ($r_i = q^{2k_i}$) so that $\frac{F(\tau_i,z)}{F(\tau_i,r_i^{-1}z)}$ is entire, as in~\eqref{eq:F_quotient}. The entireness then leads to fewer contour conditions in \Cref{thm:analytic_r_moments} compared to \Cref{thm:observable_formula} which will be important for our analysis. We note also that the $|r_i|$ are $>1$ in \Cref{thm:observable_formula} and are always $<1$ in \Cref{thm:analytic_r_moments}.
\end{remark}

\begin{proof}[Proof of \Cref{thm:analytic_r_moments}]
We proceed by an analytic continuation argument. We first suppose $\lambda^{(2N)} \succ \lambda^{(1)} \prec \cdots \succ \lambda^{(2N-1)} \prec \lambda^{(2N)}$ is distributed as the unshifted periodic Schur process with parameters $a_1,b_2,a_3,\ldots,$ $a_{2N-1},b_{2N}$. Then
\[ F(\tau,z) = \frac{\prod_{i > \tau} (1 - a_iz^{-1})}{\prod_{i \le \tau} (1 - b_iz)}, \]
so \Cref{thm:observable_formula} yields
\begin{align} \label{eq:observable_formula_proof}
\begin{split}
& \E\left[ \prod_{i=1}^n \cF_{r_i}(\lambda^{(\tau_i)}) \right] = \left( \prod_{i=1}^n \frac{(t;t)_\infty^2}{(r_i t;t)_\infty (r_i^{-1}t;t)_\infty (1 - r_i^{-1})} \right) \\
& \quad \quad \times \frac{1}{(2\pi\bi)^n} \oint \cdots \oint \prod_{1 \le i < j \le n} \frac{\theta_1(z_i/z_j;t) \theta_1(r_jr_i^{-1}z_i/z_j;t)}{\theta_1(r_jz_i/z_j;t) \theta_1(r_i^{-1}z_i/z_j;t)} \prod_{i=1}^n \frac{F(\tau_i,z_i)}{F(\tau_i,r_i^{-1}z_i)} \frac{dz_i}{z_i}
\end{split}
\end{align}
whenever $1 < |r_i| < 1+\delta$ for some small $\delta > 0$ and $1 \leq i \leq n$.
Here, we may choose the contours to satisfy
\begin{align} \label{eq:cross_term_condition}
\max(r_j^{-1}|z_j|, r_i|z_j|) < |z_i| < |t|^{-1} \min(r_j^{-1}|z_j|, r_i|z_j|), \quad \quad 1 \le i < j \le n
\end{align}
and
\begin{align} \label{eq:F_quotient_condition}
\max_{j > \tau_i} (r_i a_j) < |z_i| < \min_{j \le \tau_i} b_j^{-1}, \quad \quad 1 \le i \le n.
\end{align}
Recall the~\eqref{eq:cross_term_condition} comes from analyticity conditions on the $\theta$ functions and~\eqref{eq:F_quotient_condition} comes from analyticity conditions on $\frac{F(\tau_i,z_i)}{F(\tau_i,r_i^{-1}z_i)}$. Observe that we can always ensure these conditions hold, by decreasing the $b_j$'s to be very small.

Since~\eqref{eq:observable_formula_proof} is analytic in the $a_i,b_j$ and $r_\ell$ wherever we may make sense of these expressions, we can analytically continue in these variables. Ultimately, we want $a_i,b_j$ to take the values corresponding to the $q^{\mvol}$ measure and have $r_\ell = q^{2k_\ell}$ so that $\frac{F(\tau_i,z_i)}{F(\tau_i,r_i^{-1}z_i)}$ is entire. Having that $\frac{F(\tau_i,z_i)}{F(\tau_i,r_i^{-1}z_i)}$ is entire then allows us to drop condition~\eqref{eq:F_quotient_condition}. However, to carry out this analytic continuation, we must do so carefully.

We may choose $a_{2j-1} = q^{2j - \frac{3}{2}}$ for $j = 1,\ldots,N$, as it is for the $q^{\mvol}$ measure, and vary $r_1,\ldots,r_n$ by analytic continuation so that $r_i = q^{2k_i}$ for each $1 \le i \le n$. Then
\[ \frac{F(\tau_i,z_i)}{F(\tau_i,r_i^{-1}z_i)} = \prod_{\tau_i < j \le \tau_i + 2k_i} (1 - q^j z^{-1}) \prod_{j \le \tau_i} \frac{1 - b_j r_i^{-1} z_i}{1 - b_j z_i} \]
where the first product is restricted over $j \in 2\Z + \tfrac{1}{2}$ and the second product is restricted over $j \in 2\Z$. In particular, we note that there is a massive cancellation of poles and zeros with this particular choice of $r_i$ and $a_j$. Thus the condition~\eqref{eq:F_quotient_condition} may be relaxed to
\[ |z_i| < \min_{j \le \tau_i} b_j^{-1}, \quad \quad 1 \le i \le n. \]
We may now vary $b_j$ by analytic continuation to take the value $q^{-(2j - \frac{1}{2})}$ as it is for the $q^{\mvol}$ measure, for each $j = 1,\ldots,N$. The end result is that~\eqref{eq:cross_term_condition} is the only condition on the contours because~\eqref{eq:F_quotient} tells us that $\frac{F(\tau_i,z_i)}{F(\tau_i,r_i^{-1}z_i)}$ is entire. Since $r_1,\ldots,r_n < 1$, we obtain the contour conditions in the statement of the corollary. This completes the proof. 
\end{proof}

\section{Asymptotics}\label{sec:asymptotics}

In this section we compute the mean, variance, and higher cumulants of the observables given in \Cref{thm:analytic_r_moments}. Recall that
$ \Theta(\eta\,|\,\omega) = \theta_1(e^{2\pi\bi \eta},e^{2\pi \bi \omega}). $
Throughout this section, we fix $0 < t < 1$ and set $q = t^{\frac{1}{2N}}$.

\begin{lemma} \label{thm:theta-1}
Fix $\delta,\omega > 0$ and suppose $K$ is a compact subset of $\C \setminus \{0\}$ which avoids $\omega\Z$. Then
\[ \frac{\Theta(\alpha\,|\,\omega)\Theta(\alpha + v_1 - v_2\,|\,\omega)}{\Theta(\alpha - v_2\,|\,\omega)\Theta(\alpha + v_1\,|\,\omega)} - 1 = - v_1v_2 \partial^2 \log \Theta(\alpha\,|\,\omega) + O(N^{-3\delta}) \]
uniformly for $|v_1|,|v_2| < N^{-\delta}$ and $\alpha \in K$, where $\partial^2 \log \Theta$ is the second derivative of $f(z) = \log \Theta(z\,|\,\omega)$.
\end{lemma}
\begin{proof}
The proof is an elementary computation, starting from a Taylor expansion of each $\Theta$ term at $\alpha$.
\end{proof}

\begin{proposition} \label{thm:moment_asymptotics}
For any $0 < \tau \le 1$ and any integer $k > 0$, we have
\[ \lim_{N\to\infty} \E \left[ \frac{1}{2N} \sum_{y \in \tfrac{1}{2N} \Z'} \frac{\hh(\lfloor 2N \tau \rfloor, 2N y)}{2N} t^{2k y} \right] = \frac{1}{(2k \log t)^2} \binom{2k}{k}. \]
\end{proposition}

\begin{proof}
Set $\tau^\sharp := \lfloor 2N \tau \rfloor$ and $r = t^{k/N}$. By \Cref{thm:height_function_observable} and \Cref{thm:analytic_r_moments}, we have
\begin{align*}
\E \left[ \frac{1}{(2N)^2} \sum_{y^\sharp \in \Z'} \hh(\tau^\sharp,y^\sharp) t^{ky^\sharp/N} \right] &= - \frac{1}{(2N)^2} \frac{r^{1/2}}{1 - r} \E \left[ \cF_r(\lambda^{(\tau^\sharp)}) \right] \\
&= \frac{1}{(2N)^2} \frac{r^{3/2}}{(1 - r)^2} \frac{(t;t)_\infty^2}{(rt;t)_\infty(r^{-1}t;t)_\infty} \frac{1}{2\pi\bi} \oint \frac{F(\tau^\sharp,z)}{F(\tau^\sharp,r^{-1}z)} \frac{dz}{z}
\end{align*}
where the $z$ contour is some simple contour positively oriented around $0$. Note that $1 - r \sim k (\log t)/N$ and
\[ \lim_{N\to\infty} \frac{F(\tau^\sharp,z)}{F(\tau^\sharp,r^{-1}z)} = (1 - t^\tau z)^k (1 - t^{-\tau} z)^k = \frac{(1 - t^{-\tau} z)^{2k}}{(-t^{-\tau}z)^k}, \]
where the first equality follows by~\eqref{eq:F_quotient} since $t=q^{2N}$. Hence our expectation converges to 
\[ \frac{1}{(2k \log t)^2 2\pi\bi} \oint \frac{(1 - t^{-\tau} z)^{2k}}{(-t^{-\tau}z)^k} \frac{dz}{z}. \]
Observing that the contour integral evaluates to the central binomial coefficient $\binom{2k}{k}$ gives us the desired.
\end{proof}

\begin{proposition} \label{thm:covariance_asymptotics}
For any $0 < \tau_1 \le \tau_2 \le 1$ and any integers $k_1,k_2 > 0$, we have
\begin{align*}
& \lim_{N\to\infty} \cov \left( \frac{1}{2N} \sum_{y\in\tfrac{1}{2N}\Z'} \bar{h}_N(\lfloor 2N\tau_1 \rfloor,2Ny) t^{2k_1 y}, \frac{1}{2N} \sum_{y \in \tfrac{1}{2N}\Z'} \bar{h}_N(\lfloor 2N \tau_2 \rfloor,2N y)t^{2k_2y}\right) \\
& \quad \quad = - \frac{1}{4k_1k_2(2\pi\bi \log t)^2} \int_{-\frac{1}{2} + \bi c_1}^{\frac{1}{2} + \bi c_1} \int_{-\frac{1}{2} + \bi c_2}^{\frac{1}{2} + \bi c_2} \partial^2 \log \Theta(\eta_1 - \eta_2\,|\,\omega) \prod_{i=1}^2  (1 - t^{\tau_i} e^{-2\pi\bi \eta_i})^{k_i} (1 - t^{-\tau_i} e^{2\pi\bi \eta_i})^{k_i} d\eta_i,
\end{align*}
where $t = e^{2\pi\bi \omega}$, $c_1,c_2$ are any real numbers such that $c_2 - \tfrac{|\log t|}{2\pi} < c_1 < c_2$, and $\partial^2 \log \Theta$ is the second derivative of $f(z) = \log \Theta(z\,|\,\omega)$.
\end{proposition}

\begin{proof}
Set $\tau_i^\sharp := \lfloor 2N \tau_i \rfloor$ and $r_i = t^{k_i/N}$ for $i = 1,2$. We are interested in computing the asymptotics of
\[ \E\left[ \prod_{i=1}^2 \frac{1}{2N} \sum_{y^\sharp\in\Z'} \hh(\tau_i^\sharp,y^\sharp) t^{2k_i\frac{y^\sharp}{2N}} \right] - \prod_{i=1}^2 \E\left[ \frac{1}{2N} \sum_{y^\sharp\in\Z'} \hh(\tau_i^\sharp,y^\sharp) t^{2k_i\frac{y^\sharp}{2N}} \right] \]
as $N\to\infty$. By \Cref{thm:height_function_observable}, this is equal to
\[ \frac{1}{4N^2} \frac{r_1^{1/2} r_2^{1/2}}{(1 - r_1)(1 - r_2)} \left( \E\left[ \cF_{r_1}(\lambda^{(\tau_1^\sharp)}) \cF_{r_2}(\lambda^{(\tau_2^\sharp)}) \right] - \E\left[ \cF_{r_1}(\lambda^{(\tau_1^\sharp)}) \right] \E\left[ \cF_{r_2}(\lambda^{(\tau_2^\sharp)}) \right] \right). \]
Recalling that $t=q^{2N}$, \Cref{thm:analytic_r_moments} yields
\begin{align*}
& \frac{1}{4N^2} \frac{r_1^{3/2} r_2^{3/2}}{(1 - r_1)^2(1 - r_2)^2} \left( \prod_{i=1}^2 \frac{(t;t)_\infty^2}{(r_it;t)_\infty (r_i^{-1}t;t)_\infty} \right) \\
& \quad \quad \times \frac{1}{(2\pi\bi)^2} \oint \oint \left( \frac{\theta_1(z_1/z_2;t) \theta_1(r_2r_1^{-1}z_1/z_2;t)}{\theta_1(r_2z_1/z_2;t)\theta_1(r_1^{-1}z_1/z_2;t)} - 1 \right) \prod_{i=1}^2 \frac{F(\tau_i^\sharp,z_i)}{F(\tau_i^\sharp,r_i^{-1}z_i)} \frac{dz_i}{z_i}
\end{align*}
where the contours are positively oriented concentric circles centered at $0$ such that
\[ r_2^{-1}|z_2| < |z_1| < t^{-1} r_1|z_2|. \]
We choose the circles so that $|z_2| < |z_1| < t^{-1} |z_2|$ independent of $N$. Then the contour conditions above are satisfied for $N$ large enough since $r_i \to 1$ for $i = 1,2$. It is convenient to change coordinates $e^{2\pi\bi\eta_i} = z_i$ where we choose the branch so that $\Re \eta_i$ varies from $-1/2$ to $1/2$. Then we have
\begin{align*}
& \frac{1}{4N^2} \frac{r_1^{3/2} r_2^{3/2}}{(1 - r_1)^2(1 - r_2)^2} \left( \prod_{i=1}^2 \frac{(t;t)_\infty^2}{(r_it;t)_\infty (r_i^{-1}t;t)_\infty} \right) \\
& \times \int_{-\frac{1}{2} + \bi c_1}^{\frac{1}{2} + \bi c_1} \int_{-\frac{1}{2} + \bi c_2}^{\frac{1}{2} + \bi c_2} \left( \frac{\Theta(\eta_1 - \eta_2\,|\,\omega) \Theta(\tfrac{k_2-k_1}{2\pi\bi N} \log t + \eta_1 - \eta_2\,|\,\omega)}{\Theta(\tfrac{k_2}{2\pi\bi N} \log t + \eta_1 - \eta_2\,|\,\omega) \Theta(-\tfrac{k_1}{2\pi\bi N} \log t + \eta_1 - \eta_2\,|\,\omega)} - 1 \right) \prod_{i=1}^2 \frac{F(\tau_i^\sharp,e^{2\pi\bi \eta_i})}{F(\tau_i^\sharp,r_i^{-1}e^{2\pi\bi \eta_i})} d\eta_i
\end{align*}
where $c_1$ and $c_2$ are any real numbers such that $c_2 - \tfrac{|\log t|}{2\pi} < c_1 < c_2$.

To estimate this integral, note that
\[
1 - r_i \sim -k_i(\log t)/N \cdot (1 + o(1)),
\]
and
\[ \lim_{N\to\infty} \frac{F(\tau_i^\sharp,e^{2\pi\bi \eta_i})}{F(\tau_i^\sharp,r_i^{-1}e^{2\pi\bi \eta_i})} = (1 - t^{\tau_i} e^{-2\pi\bi \eta_i})^{k_i} (1 - t^{-\tau_i} e^{2\pi\bi \eta_i})^{k_i}. \]
Additionally,
\[ \frac{\Theta(\eta_1 - \eta_2\,|\,\omega) \Theta(\tfrac{k_2-k_1}{2\pi\bi N} \log t + \eta_1 - \eta_2\,|\,\omega)}{\Theta(\tfrac{k\,|\,2}{2\pi\bi N} \log t + \eta_1 - \eta_2\,|\,\omega) \Theta(-\tfrac{k_1}{2\pi\bi N} \log t + \eta_1 - \eta_2\,|\,\omega)} - 1 = -\frac{k_1k_2(\log t)^2}{(2\pi \bi N)^2} \partial^2 \log \Theta(\eta_1 - \eta_2\,|\,\omega) + o(N^{-2}) \]
by applying \Cref{thm:theta-1} with the compact set $K$ containing both contours, $\alpha = \eta_1-\eta_2$, a$v_i = -\tfrac{k_i}{2 \pi \bi N} \log t$.

Hence our double integral converges to
\[ - \frac{1}{4k_1k_2(2\pi\bi \log t)^2} \int_{-\frac{1}{2} + \bi c_1}^{\frac{1}{2} + \bi c_1} \int_{-\frac{1}{2} + \bi c_2}^{\frac{1}{2} + \bi c_2} \partial^2 \log \Theta(\eta_1 - \eta_2\,|\,\omega) \prod_{i=1}^2  (1 - t^{\tau_i} e^{-2\pi\bi \eta_i})^{k_i} (1 - t^{-\tau_i} e^{2\pi\bi \eta_i})^{k_i} d\eta_i \]
as desired.
\end{proof}

\begin{proposition} \label{thm:cumulant_asymptotics}
If $m \ge 3$, then for any $0 < \tau_1 \le \ldots \le \tau_m \le 1$ and any integers $k_1,\ldots,k_m > 0$, we have
\[ \lim_{N\to\infty} \kappa_m\left( \frac{1}{2N} \sum_{y\in\tfrac{1}{2N}\Z'} \bar{h}_N(\lfloor 2N\tau_1 \rfloor,2Ny) t^{2k_1y}, \ldots, \frac{1}{2N} \sum_{y \in \tfrac{1}{2N}\Z'} \bar{h}_N(\lfloor 2N \tau_m \rfloor,2N y)t^{2k_my}\right) = 0, \]
where $\kappa_m$ is the $m\tth$ joint cumulant.
\end{proposition}

\begin{proof}
Set $\tau_i^\sharp := \lfloor 2N \tau_i \rfloor$ and $r_i := t^{k_i/N}$ for $1 \le i \le m$. Recall that the order $\ge 2$ cumulants are invariant under additive shifts of the random variables by a constant. Then the cumulant we want to consider is
\begin{align*}
\kappa_m\left( \frac{1}{2N} \sum_{y^\sharp\in\Z'} \hh(\tau_1^\sharp,y^\sharp) t^{k_1y^\sharp/N}, \ldots, \frac{1}{2N} \sum_{y^\sharp\in\Z'} \hh(\tau_m^\sharp,y^\sharp)t^{k_my^\sharp/N}\right),
\end{align*}
where we rewrote the sum to be over $\Z'$. By \Cref{thm:height_function_observable}, we may express this as
\[ \left( \frac{1}{(2N)^m} \prod_{i=1}^m \frac{r_i^{1/2}}{r_i - 1} \right) \kappa_m\left( \cF_{r_1}(\lambda^{(\tau_1^\sharp)}), \ldots, \cF_{r_m}(\lambda^{(\tau_m^\sharp)}) \right). \]
Since $r_i - 1 = k_i/N \log t + O(1/N^2)$, we want to show that
\[ \lim_{N\to\infty} \kappa_m\left( \cF_{r_1}(\lambda^{(\tau_1^\sharp)}), \ldots, \cF_{r_m}(\lambda^{(\tau_m^\sharp)}) \right) = 0. \]

We can express the joint cumulants in terms of the joint moments (see e.g.~\cite[C3.2]{PT11}) to obtain
\begin{align*}
& \kappa_m\left( \cF_{r_1}(\lambda^{(\tau_1^\sharp)}), \ldots, \cF_{r_m}(\lambda^{(\tau_m^\sharp)}) \right) \\
& \quad \quad = \sum_{\substack{d > 0 \\ \{S_1,\ldots,S_d\} \in \cP_{[[1,m]]}}} (-1)^{d-1} (d-1)! \prod_{\ell=1}^d \E\left[ \prod_{i \in S_\ell} \cF_{r_i}(\lambda^{(\tau_i^\sharp)}) \right],
\end{align*}
where $\cP_{[[1,m]]}$ is the collection of set partitions of $[[1,m]]$. We may replace the expectations by \Cref{thm:analytic_r_moments} to obtain
\[ \left( \prod_{i=1}^m \frac{(t;t)_\infty^2}{(r_i t;t)_\infty (r_i^{-1}t;t)_\infty (1 - r_i^{-1})} \right) \frac{1}{(2\pi\bi)^m} \oint \cdots \oint \cC(z_1,\ldots,z_m) \prod_{i=1}^n \frac{F(\tau_i^\sharp,z_i)}{F(\tau_i^\sharp,r_i^{-1}z_i)} \frac{dz_i}{z_i} \]
where
\begin{equation}\label{eq:C_def}
    \cC(z_1,\ldots,z_m) := \sum_{\substack{d > 0 \\ \{S_1,\ldots,S_d\} \in \cP_m}} (-1)^{d-1} (d-1)! \prod_{\ell=1}^d \prod_{\substack{i < j \\ i,j \in S_\ell}} \frac{\theta_1(z_i/z_j;t) \theta_1(r_jr_i^{-1}z_i/z_j;t)}{\theta_1(r_jz_i/z_j;t) \theta_1(r_i^{-1}z_i/z_j;t)}
\end{equation} 
and the contours are positively oriented concentric circles centered at $0$ such that
\[ r_j^{-1}|z_j| < |z_i| < t^{-1} r_i|z_j|, \quad \quad 1 \le i < j \le m. \]
Choose circles such that
\[ |z_j| < |z_i| < t^{-1} |z_j|, \quad \quad 1 \le i < j \le m \]
independent of $N$. Then the contour conditions above are satisfied for $N$ large enough since $r_i \to 1$ for $1 \le i \le m$.

Let $S \subset [[1,m]]$, let $\cT(S)$ denote the set of undirected simple graphs with vertex set $S$, and let $\cL(S) \subset \cT(S)$ denote the subset of connected graphs. Given a graph $\Omega$, denote by $E(\Omega)$ its edge set. We claim that
\begin{align} \label{eq:mobius}
\cC(z_1,\ldots,z_m) = \sum_{\Omega \in \cL([[1,m]])} \prod_{\substack{(i,j) \in E(\Omega) \\ i < j}} \left( \frac{\theta_1(z_i/z_j;t) \theta_1(r_jr_i^{-1}z_i/z_j;t)}{\theta_1(r_jz_i/z_j;t) \theta_1(r_i^{-1}z_i/z_j;t)} - 1 \right).
\end{align} 
Define
\begin{align*}
\cK(S) & := \sum_{\Omega \in \cL(S)} \prod_{\substack{(i,j) \in E(\Omega) \\ i < j}} \left( \frac{\theta_1(z_i/z_j;t) \theta_1(r_jr_i^{-1}z_i/z_j;t)}{\theta_1(r_jz_i/z_j;t) \theta_1(r_i^{-1}z_i/z_j;t)} - 1 \right) \\
\cE(S) &:= \sum_{\Omega \in \cT(S)} \prod_{\substack{(i,j) \in E(\Omega) \\ i < j}} \left( \frac{\theta_1(z_i/z_j;t) \theta_1(r_jr_i^{-1}z_i/z_j;t)}{\theta_1(r_jz_i/z_j;t) \theta_1(r_i^{-1}z_i/z_j;t)} - 1 \right)
\end{align*}
where we note the only distinction between the two definitions is set that we sum over. These are related by
\[ \cE(S) = \sum_{\substack{d > 0 \\ \{S_1,\ldots,S_d\} \in \cP_S}} \prod_{\ell=1}^d \cK(S_\ell) \]
where $\cP_S$ is the collection of set partitions of $S$. By generalized M\"obius inversion (see~\cite[Lemma A.5]{Ahn20}, \cite[p964]{GZ18}), we have
\begin{equation}\label{eq:K_to_E}
    \cK(S) = \sum_{\substack{d > 0 \\ S_1,\ldots,S_d}} (-1)^{d-1} (d-1)! \prod_{\ell=1}^d \cE(S_\ell). 
\end{equation}
Furthermore, we have
\begin{equation}\label{eq:E_alternate} 
\cE(S) = \prod_{\substack{i < j \\ i,j \in S}} \frac{\theta_1(z_i/z_j;t) \theta_1(r_jr_i^{-1}z_i/z_j;t)}{\theta_1(r_jz_i/z_j;t) \theta_1(r_i^{-1}z_i/z_j;t)},
\end{equation}
which may be seen by writing the terms in the above product as 
\[
\left(\frac{\theta_1(z_i/z_j;t) \theta_1(r_jr_i^{-1}z_i/z_j;t)}{\theta_1(r_jz_i/z_j;t) \theta_1(r_i^{-1}z_i/z_j;t)} -1\right) + 1
\]
and expanding. Substituting~\eqref{eq:E_alternate} into~\eqref{eq:K_to_E} yields~\eqref{eq:C_def}, while the definition of $\cK$ yields RHS~\eqref{eq:mobius}, proving the claim.

Thus our cumulant becomes
\begin{align*}
& \left( \prod_{i=1}^m \frac{(t;t)_\infty^2}{(r_i t;t)_\infty (r_i^{-1}t;t)_\infty (1 - r_i^{-1})} \right) \\
& \quad \quad \times
\sum_{\Omega \in \cL([[1,m]]} \frac{1}{(2\pi\bi)^m} \oint \cdots \oint \prod_{\substack{(i,j) \in E(\Omega) \\ i < j}}\left( \frac{\theta_1(z_i/z_j;t) \theta_1(r_jr_i^{-1}z_i/z_j;t)}{\theta_1(r_jz_i/z_j;t) \theta_1(r_i^{-1}z_i/z_j;t)} - 1 \right) \prod_{i=1}^m \frac{F(\tau_i^\sharp,z_i)}{F(\tau_i^\sharp,r_i^{-1}z_i)} \frac{dz_i}{z_i}.
\end{align*}
We have $1/(1 - r_i^{-1}) = O(N)$ for $1 \le i \le m$, and
\[ \frac{\theta_1(z_i/z_j;t) \theta_1(r_jr_i^{-1}z_i/z_j;t)}{\theta_1(r_jz_i/z_j;t) \theta_1(r_i^{-1}z_i/z_j;t)} - 1 = O(1/N^2), \quad \quad 1 \le i < j \le m \]
by \Cref{thm:theta-1} where we recall $\theta_1(e^{2\pi\bi \eta};e^{2\pi\bi \omega}) = \Theta(\eta\,|\,\omega)$. Therefore, the summand indexed by $\Omega$ has decay of order $O(N^{m - 2|E(\Omega)|})$. Since $\Omega$ is connected on $m$ vertices, $|E(\Omega)| \ge m - 1$ from which we conclude that
\[ \kappa_m\left( \cF_{r_1}(\lambda^{(\tau_1^\sharp)}), \ldots, \cF_{r_m}(\lambda^{(\tau_m^\sharp)}) \right) = O(N^{2-m}) \]
which is $o(1)$ for $m \ge 3$.
\end{proof}

\section{Limit shape and fluctuations}\label{sec:main_proofs}

In this section we establish the limit shape, then prove the main results stated earlier in \Cref{subsec:main_res}, namely the Gaussian free field convergence for the $q^{\mvol}$ measure and its shift-mixed version. 

\subsection{Limit shape} \label{sec:limit_shape}
The result below shows the limit shape convergence stated earlier as \Cref{thm:limit_shape_intro}, and a statement on convergence of observables which will be used later. Recall the limit shape $\cH$ given in \Cref{def:limit_shape}, and the liquid region $\sL$ given by~\eqref{eq:def_Liquid}.

\begin{proposition} \label{thm:limit_shape}
For each $0 < \tau \le 1$ and $k \in \Z_{>0}$, we have the convergence
\begin{align} \label{eq:weak_convergence}
\lim_{N\to\infty} \frac{1}{2N} \sum_{y \in \frac{1}{2N} \Z'} \frac{1}{2N} \hh(\lfloor 2N \tau \rfloor, 2N y) t^{ky} = \int \cH(y) t^{ky} \, dy
\end{align}
in probability. Moreover, the height function converges uniformly in probability on vertical slices, i.e.
\begin{align} \label{eq:uniform_convergence}
\lim_{N\to\infty} \PP\left( \sup_{y \in \frac{1}{2N} \Z} \left| \frac{1}{2N} \hh(\lfloor 2N \tau \rfloor, 2N y) - \cH(y) \right| \ge \e \right) = 0
\end{align}
for each $\e > 0$.
\end{proposition}

Before the proof, we collect various facts related to the limit shape. Note that there exists a unique solution $\zeta := \zeta(\tau,y)$ in $\C^+$ of
\begin{equation}\label{eq:zeta_defining}
     (1 - t^{-\tau} \zeta^{-1})(1 - t^{\tau} \zeta) = t^{2y}
\end{equation} 
if and only if $y \in (\tfrac{\log 2}{\log t},\infty)$, explicitly given by
\[ \zeta = t^{-\tau} \frac{2 - t^{2 y} + \bi \sqrt{4 t^{2y} - t^{4y}}}{2}. \]

\begin{proposition} \label{thm:complex_coordinate}
The map $\zeta$ is a diffeomorphism from $\R \times (\tfrac{\log 2}{\log t},\infty)$ onto the upper half-plane. Additionally,
\begin{align} \label{eq:zeta_proportion}
\cH'(y) = \frac{\pi - \arg \zeta(\tau,y)}{\pi} \bbone\left(y > \tfrac{\log 2}{\log t}\right)
\end{align}
is independent of $\tau \in \R$, where $\cH$ is as defined in \Cref{def:limit_shape}.
\end{proposition}

\begin{remark}
Recall that the height function differs between $(\tau,y)$ and $(\tau,y+1)$ exactly when there is no horizontal lozenge at $(\tau,y)$ by~\eqref{eq:height_function}. It follows that $1-\cH'(y) = \frac{\arg \zeta(\tau,y)}{\pi}$ may be interpreted as the asymptotic proportion of $\hloz$ locally near $(\tau,y)$ in the liquid region $\sL$.
\end{remark}

\begin{proof}
[Proof of {\Cref{thm:complex_coordinate}}]
It is clear that $0 < t^{2y} < 4$ is equivalent to $y > \tfrac{\log 2}{\log t}$. Observe that for $y > \tfrac{\log 2}{\log t}$,
\begin{align*}
\cH'(y) &:=\frac{2\arctan\left(\sqrt{4 t^{-2y} - 1}\right)}{\pi} \\
&= \frac{2 \arg(1 + \bi \sqrt{4 t^{-2y} - 1})}{\pi} \\
&= \frac{\arg\left((1 + \bi \sqrt{4t^{-2y} - 1})^2\right)}{\pi} \\
&= \frac{\arg\left(2 - 4t^{-2y} + 2\bi \sqrt{4t^{-2y} - 1}\right)}{\pi}.
\end{align*}
Since the argument is independent of modulus and $\arg(a + \bi b) = \pi - \arg(-a + \bi b)$ for $a \in \R$ $b > 0$, the above is equal to
\[ \frac{\pi - \arg \zeta(\tau,y)}{\pi} \]
for $(\tau,y) \in \R \times (\tfrac{\log 2}{\log t},\infty)$, which shows~\eqref{eq:zeta_proportion}. We note that this equality holds independent of $\tau \in \R$ because~$\tau$ only modifies the modulus of $\zeta$. Therefore~\eqref{eq:zeta_proportion} holds.

To check that $\zeta$ is a diffeomorphism onto the upper half-plane, it suffices to show that $\zeta$ is a bijection onto the upper half-plane. Observe that for fixed $\tau$, $\zeta$ maps bijectively onto the semicircle $\{|z| = t^{-\tau}\} \cap \C^+$ in the upper half-plane. By varying $\tau \in \R$, we see that $\zeta$ is a bijection onto the upper half-plane.
\end{proof}

We recall the map $\eta$ from~\eqref{eq:eta_intro}, which is naturally defined in terms of $\zeta$.

\begin{definition}\label{def:eta}
We define $\eta: \sL \to \frac{1}{2}(0,1) + \bi \frac{|\log t|}{2\pi} (0,1]$ by
\begin{equation}
    \eta(\tau,y) := \frac{1}{2\pi\bi} \log t^{2\tau}\zeta(\tau,y) = \frac{1}{2\pi\bi} \log\left(t^\tau \frac{2 - t^{2 y} + \bi \sqrt{4 t^{2y} - t^{4y}}}{2}\right),
\end{equation}
where we choose the branch of the logarithm so that the imaginary part is in $(0,\pi]$.
\end{definition}

\begin{remark} \label{remark:eta}
Since $|\zeta(\tau,y)| = t^{-\tau}$ for $(\tau,y) \in \sL$, we have
\[ \eta\left(\{\tau\} \times (\tfrac{\log 2}{\log t},\infty)\right) = \tfrac{1}{2}(0,1) + \tfrac{\tau |\log t|}{2\pi} \bi.
\]
\end{remark}

\begin{remark}
In the above we have been viewing $(\tau,y)$ as living in a fundamental domain of a covering space $\R \times (\tfrac{\log t}{\log 2},\infty) \to (\R/\Z) \times (\tfrac{\log t}{\log 2},\infty)$. We could have equivalently worked with the cover instead, which removes the apparent discontinuity at the boundary of the fundamental domain. However, we found that for our purposes, it is more convenient to work with the fundamental domain. 
\end{remark}

The following computation will be used in the proof of \Cref{thm:limit_shape}.

\begin{lemma}\label{thm:rho_moments}
For $k \in \Z_{\geq 1}$,
\begin{align} \label{eq:rho_moments}
\int_{-\infty}^\infty \cH'(y) t^{2ky} \, dy = -\frac{1}{2k \log t} \binom{2k}{k}.
\end{align}
\end{lemma}
\begin{proof}
Plugging in our expression for $\cH$, the left hand side of \eqref{eq:rho_moments} becomes
\[ \int_{\frac{\log 2}{\log t}}^\infty \frac{2 \arctan(\sqrt{4 t^{-2y} - 1})}{\pi} t^{2ky} \, dy. \]
Changing variables $x = t^y/2$ and integrating by parts, we see that the desired equality \eqref{eq:rho_moments} is equivalent to
\[ \int_0^1 \frac{x^{2k}}{\sqrt{1 - x^2}} \, dx = \frac{1}{2^{2k+1}} \binom{2k}{k} \pi. \]
Note that this is the Wallis integral, indeed the left hand side becomes
\[ \int_0^{\pi/2} \sin^{2k}(\phi) \, d\phi \]
after changing variables $x = \sin \phi$.
\end{proof}

\begin{proof}[Proof of \Cref{thm:limit_shape}]

In this proof, we consider a linear interpolation of the rescaled height function which is convenient for stating intermediate weak convergence results that lead to the desired convergences~\eqref{eq:weak_convergence} and~\eqref{eq:uniform_convergence}. We first use \Cref{thm:moment_asymptotics} to obtain moment asymptotics for these linear interpolations. We then show that the moments of the measure with density $\cH'(y)\,dy$ coincide with the limiting moments, which immediately gives us~\eqref{eq:weak_convergence}. We then show that the moment convergence implies weak convergence which in turn implies~\eqref{eq:uniform_convergence} as a consequence of our linear interpolations being $1$-Lipschitz.

\textbf{Step 1.} Fix $0 < \tau \le 1$. Let $\wt{h}_N(\tau,y)$ be the linear interpolation of the function $y \mapsto \tfrac{1}{2N} \hh(\lfloor 2N \tau \rfloor, 2Ny)$ on $\tfrac{1}{2N}\Z'$. The goal of this step is to show
\begin{align} \label{eq:limit_linear_interpolation}
\int_{-\infty}^\infty \wt{h}_N(\tau,y) t^{2ky} \to \frac{1}{(2k \log t)^2} \binom{2k}{k}
\end{align}
in probability as $N\to\infty$.

Observe that $\wt{h}_N(\tau,y)$ is $1$-Lipschitz and monotonically increasing in $y$. Beyond these properties, the main utility for considering this linear interpolation is that our height function observables can be expressed as integrals of $\wt{h}_N$:
\begin{align} \label{eq:linear_interpolation}
\int_{-\infty}^\infty \wt{h}_N(\tau,y) t^{2ky} dy = \left( \frac{t^{\frac{k}{2N}} - t^{-\frac{k}{2N}}}{2k \log t} \right)^2 \sum_{y \in \frac{1}{2N}\Z'} \hh(\lfloor 2N\tau\rfloor,2Ny) t^{2ky}.
\end{align}
We now prove~\eqref{eq:linear_interpolation}.

By integrating by parts, we have
\begin{align*}
\frac{2k \log t}{t^{\frac{k}{N}} - 1} \int_{-\infty}^\infty \wt{h}_N(\tau,y) t^{2ky} \, dy &= \frac{1}{t^{\frac{k}{N}} - 1} \wt{h}_N(\tau,y) t^{2ky} \Bigg|_{-\infty}^\infty - \frac{1}{t^{\frac{k}{N}} - 1} \int_{-\infty}^\infty \partial_2 \wt{h}_N(\tau,y) t^{2ky} \, dy.
\end{align*}
Note that the boundary term vanishes. Setting $x_i := x_i^{(\tau)} := \tfrac{\lambda^{(\lfloor 2N \tau \rfloor)}_i - i + \frac{1}{2}}{2N}$, we have
\[ \partial_2 \wt{h}_N(\tau,y) = 1 - \sum_{i=1}^\infty \bbone(y \in [x_i^{(\tau)},x_i+\tfrac{1}{2N}]) \]
for all $y \not \in \frac{1}{2N}\Z'$.

Thus the above is equal to
\[ -\frac{1}{t^{\frac{k}{N}} - 1} \int_{-\infty}^\infty \! \Bigg( 1 - \sum_{i =1}^\infty \bbone(y \in [x_i^{(\tau)},x_i+\tfrac{1}{2N}]) \Bigg) t^{2ky} \, dy. \]
If $\ell$ is the length of $\lambda_i^{(\lfloor 2N \tau\rfloor)}$, then $x_{\ell+1}+\tfrac{1}{2N} = \tfrac{-\ell+\frac{1}{2}}{2N}$ and we can rewrite this integral as
\[ -\frac{1}{t^{\frac{k}{N}} - 1} \int_{\frac{-\ell+\frac{1}{2}}{2N}}^\infty \! \Bigg( 1 - \sum_{i=1}^\ell \bbone(y \in [x_i^{(\tau)},x_i+\tfrac{1}{2N}]) \Bigg) t^{2ky} \, dy \]
which evaluates to
\[ \frac{1}{2k \log t} \left( \frac{t^{\frac{k}{N}(-\ell+\frac{1}{2})}}{t^{\frac{k}{N}} - 1} + \sum_{i=1}^\ell t^{2kx_i} \right) = \frac{t^{-\frac{k}{2N}}}{2k \log t} \cF_{t^{\frac{k}{N}}}(\lambda^{(\lfloor 2N \tau \rfloor)})  \]
where $\cF$ is as in~\eqref{eq:def_cF}. Thus
\[ \int_{-\infty}^\infty \wt{h}_N(\tau,y) t^{2ky} dy = \frac{t^{\frac{k}{2N}} - t^{-\frac{k}{2N}}}{(2k \log t)^2} \cF_{t^{\frac{k}{N}}}(\lambda^{(\lfloor 2N \tau \rfloor)}). \]
We obtain~\eqref{eq:linear_interpolation} by applying \Cref{thm:height_function_observable} to the RHS.

Since
\[ \lim_{N\to\infty} (2N)^2\left(\frac{t^{\frac{k}{2N}} - t^{-\frac{k}{2N}}}{2k \log t}\right)^2 = 1, \]~\eqref{eq:linear_interpolation} implies that
\begin{align} \label{eq:same_limit}
\int_{-\infty}^\infty \wt{h}_N(\tau,y) t^{2ky} dy \quad \quad \mbox{and} \quad \quad \frac{1}{2N} \sum_{y \in \frac{1}{2N}\Z'} \frac{\hh(\lfloor 2N\tau\rfloor,2Ny)}{2N} t^{2ky}
\end{align}
share the same limit. By \Cref{thm:moment_asymptotics} and \Cref{thm:covariance_asymptotics} (the latter gives concentration by the usual Chebyshev argument), we obtain~\eqref{eq:limit_linear_interpolation}.

\textbf{Step 2.}
 Set
\[ \rho(y) := \cH'(y) = \frac{2\arctan\left(\sqrt{4 t^{-2y} - 1}\right)}{\pi} \bbone(0 < t^{2y} < 4). \]
Applying integration by parts on~\eqref{eq:limit_linear_interpolation} and using \Cref{thm:rho_moments}, we obtain
\begin{align} \label{eq:derivative_convergence}
\lim_{N\to\infty} \int_{-\infty}^\infty \partial_2 \wt{h}_N(\tau,y) t^{2ky} \, dy = \int_{-\infty}^\infty \rho(y) t^{2ky} \, dy.
\end{align}
This immediately implies~\eqref{eq:weak_convergence} by applying integration by parts on both sides of \eqref{eq:derivative_convergence} and noting that
\[ \cH(y) = \int_{-\infty}^y \rho(s) \, ds. \]

\textbf{Step 3.} Let $d\nu_N(x) := \partial_2 \wt{h}_N(\tau,\tfrac{\log x}{2\log t}) \, dx$ and $d\nu(x) := \rho(\tfrac{\log x}{2\log t}) \, dx$. Our next goal is to establish
\begin{align} \label{eq:derivative_transformed_convergence}
  d\nu_N(x) \to d\nu(x)
\end{align}
weakly in probability. We will need this weak convergence in the next step to prove pointwise convergence in probability of height functions, which we upgrade to the full uniform convergence statement~\eqref{eq:uniform_convergence} in Step 5. We note that it is more convenient to work in the transformed variables $y = \tfrac{\log x}{2 \log t}$ as $\rho(\tfrac{\log x}{2\log t})$ is compactly supported and the observables $t^{2ky}$ become moments.

Changing variables in~\eqref{eq:derivative_convergence}, we have
\[ \lim_{N\to\infty} \int_0^\infty x^{k-1} \, d\nu_N x = \int_0^\infty x^{k-1} \, d\nu(x) \]
for integers $k > 0$. Using the fact that
\[ \var\left( \int_0^\infty x^{k-1} d\nu(x) \right) \to 0 \]
by \Cref{thm:covariance_asymptotics}, the fact that $\nu$ is supported in $[0,4]$, and the bound
\[ \left| \int_0^\infty x^{k-1} d\nu(x) \right| = \left| 2 \log t\int_{-\infty}^\infty \rho(y) t^{2ky} \, dy \right| = \left| \frac{1}{k} \binom{2k}{k} \right| \le 4^k \]
for $k \in \Z_{>0}$, the desired weak convergence in probability follows exactly in the proof of~\cite[Lemma 2.1.7]{AGZ10}.

\textbf{Step 4.} In this step, we prove that 
\begin{align} \label{eq:convergence_in_prob}
\wt{h}_N(\tau,y) \to \cH(y)
\end{align}
in probability as $N\to\infty$, for all $0 < \tau \le 1$ and $y \in \R$.

The convergence in~\eqref{eq:convergence_in_prob} is equivalent to the statement that 
\begin{align} \label{eq:1/x_test}
-\int_{t^{2y_0}}^\infty \partial_2 \wt{h}_N(\tau,\tfrac{\log x}{2 \log t}) \, \frac{dx}{2 x \log t} \to -\int_{t^{2y_0}}^\infty \rho(\tfrac{\log x}{2 \log t}) \, \frac{dx}{2 x \log t}
\end{align}
in probability as $N\to\infty$ for each $y_0 \in \R$, because the LHS of~\eqref{eq:1/x_test} is equal to
\[ \int_{-\infty}^{y_0} \partial_2 \wt{h}_N(\tau,y) \, dy = \wt{h}_N(\tau,y_0) \]
and the RHS is equal to
\[ \int_{-\infty}^{y_0} \rho(y) \, dy = \cH(y_0) \]
by change of variables $x = t^{2y}$. To prove~\eqref{eq:1/x_test} we would like to be able to invoke the weak convergence result from Step 3. However, there is a slight issue here since the test function in~\eqref{eq:1/x_test} has only $1/x$ decay and therefore~\eqref{eq:1/x_test} is not directly implied by weak convergence in probability. We address this issue by splitting the LHS of~\eqref{eq:1/x_test} into a main term and tail term and showing convergence on each separately.

The weak convergence in probability~\eqref{eq:derivative_transformed_convergence} implies that for each $R > 0$,
\begin{equation}\label{eq:conv_without_tail}
    \int_{t^{2y_0}}^R \partial_2 \wt{h}_N(\tau,\tfrac{\log x}{2 \log t}) \frac{dx}{2x \log t} \to \int_{t^{2y_0}}^R \rho(\tfrac{\log x}{2 \log t}) \frac{dx}{2x \log t}.
\end{equation}
Thus~\eqref{eq:1/x_test} will follow if we can show convergence of the tail term
\begin{align} \label{eq:conv_with_tail}
\int_R^\infty \partial_2 \wt{h}_N(\tau,\tfrac{\log x}{2 \log t}) \frac{dx}{2x \log t} \to 0 = \int_R^\infty \rho(\tfrac{\log x}{2 \log t}) \frac{dx}{2x \log t}
\end{align}
in probability as $N \to \infty$, for some $R > 4$; recall $\rho(\tfrac{\log x}{2 \log t})$ is supported in $[0,4]$, which is why the RHS is $0$. We will show~\eqref{eq:conv_with_tail} for $R > \max\left(4,\tfrac{1}{2|\log t|}\right)$.

Since $R > \tfrac{1}{2|\log t|}$ we may dominate the LHS of~\eqref{eq:conv_with_tail} by
\begin{equation}\label{eq:tail_to_0}
\abs*{\int_R^\infty \partial_2 \wt{h}_N(\tau,\tfrac{\log x}{2 \log t}) \frac{dx}{2x \log t}} = -\int_R^\infty \partial_2 \wt{h}_N(\tau,\tfrac{\log x}{2 \log t}) \frac{dx}{2x \log t} \le \int_R^\infty \partial_2 \wt{h}_N(\tau,\tfrac{\log x}{2 \log t}) \, dx,
\end{equation}
hence it suffices to show
\[ \int_R^\infty \partial_2 \wt{h}_N(\tau,\tfrac{\log x}{2 \log t}) \, dx \to 0 \]
in probability as $N\to\infty$. This follows from the convergences
\begin{align*}
\int_0^\infty \partial_2 \wt{h}_N(\tau,\tfrac{\log x}{2 \log t}) \, dx = \int_{-\infty}^\infty \partial_2 \wt{h}_N(\tau,y) t^{2y} \, dy \to \int_{-\infty}^\infty \rho(y) t^{2y} \, dy = \int_0^\infty \rho(\tfrac{\log x}{2 \log t}) \, dx 
\end{align*}
and 
\begin{align*}
\int_0^R \partial_2 \wt{h}_N(\tau,\tfrac{\log x}{2 \log t}) \, dx \to \int_0^R \rho(\tfrac{\log x}{2 \log t}) \, dx
\end{align*}
in probability as $N\to\infty$, where the former uses~\eqref{eq:weak_convergence} (shown in Step 2) and latter uses~\eqref{eq:derivative_transformed_convergence}.

\textbf{Step 5.} In this step we use the pointwise convergence in probability of~\eqref{eq:convergence_in_prob} to deduce uniform convergence~\eqref{eq:uniform_convergence}.

The convergence~\eqref{eq:convergence_in_prob} implies
\[ \lim_{N \to \infty} \PP\left(\sup_{0 \le i \le m} \left|\cH(y_i) - \wt{h}_N(\tau,y_i)\right| \ge \e/9 \right) = 0 \]
for any $\e > 0$, $y_0,\ldots,y_m \in \R$, and $m \in \Z_{>0}$. Fix $\e > 0$ arbitrarily small and $R > 0$ arbitrarily large. By taking $m = \lfloor 2R/(\e/9) \rfloor$ and $y_i = -R + i\tfrac{\e}{9}$ for $i = 0,\ldots,m$, we have that for each $y \in [-R,R]$ there exists some $0 \le i \le m$ such that $|y - y_i| < \e/9$. Since $\cH$ and $\wt{h}_N$ are $1$-Lipschitz
\begin{align*}
|\cH(y) - \wt{h}_N(\tau,y)| &\le |\cH(y) - \cH(y_i)| + |\cH(y_i) - \wt{h}_N(\tau,y_i)| + |\wt{h}_N(\tau,y_i) - \wt{h}_N(\tau,y)| \\
&< \frac{2\e}{9} + |\cH(y_i) - \wt{h}_N(\tau,y_i)|,
\end{align*}
which implies
\[ \PP\left(\sup_{y \in [-R,R]} |\cH(y) - \wt{h}_N(\tau,y)| \ge \e/3\right) \le \PP\left(\sup_{0 \le i \le m} |\cH(y_i) - \wt{h}_N(\tau,y_i)| \ge \e/9\right). \]

From the convergence of the RHS to $0$ as $N\to\infty$, to show uniform convergence in probability~\eqref{eq:uniform_convergence} it suffices to show that
\begin{align} \label{eq:uniform_convergence_inequality}
\PP\left(\sup_{y \in \R} |\cH(y) - \wt{h}_N(\tau,y)| \ge \e \right) \le \PP\left(\sup_{y \in [-R,R]} |\cH(y) - \wt{h}_N(\tau,y)| \ge \e/3\right)
\end{align}
for $R$ sufficiently large. For this, we need to take $R$ large enough so that
\begin{align} \label{eq:H_y<<0}
\cH(y) = 0 \quad \quad \text{for all } y \in (-\infty,-R]
\end{align}
and
\begin{align} \label{eq:H_y>>0}
0 \le \cH(y) - y \le \e/3 \quad \quad \text{for all } y \in [R,\infty).
\end{align}
The existence of such an $R$ is guaranteed by the explicit formula for $\cH(y)$.

Since $\wt{h}_N(\tau,y)$ is a nonnegative and increasing function of $y$, if
\[ |\cH(-R) - \wt{h}_N(\tau,-R)| = \wt{h}_N(\tau,-R) < \e/3, \]
then
\[ |\cH(y) - \wt{h}_N(\tau,y)| = \wt{h}_N(\tau,y) < \e/3 \]
for all $y \in (-\infty,-R]$, where we used~\eqref{eq:H_y<<0}. Thus
\begin{align} \label{eq:y<<0_bound}
\PP\left( \sup_{y \in (-\infty,R]} |\cH(y) - \wt{h}_N(\tau,y)| \ge \e/3 \right) \le \PP\left(\sup_{y \in [-R,R]} |\cH(y) - \wt{h}_N(\tau,y)| \ge \e/3\right).
\end{align}
Similarly, we can check that $h_N(\tau^\sharp,y^\sharp) - y^\sharp - \tfrac{1}{2}$ is nonnegative and decreasing. Indeed, we have
\[ h_N(\tau^\sharp,y^\sharp + 1) - h_N(\tau^\sharp,y^\sharp) = -\bbone(\hloz~\mbox{at $(\tau,y)$}) \]
which establishes the monotonicity, and we know that $h_N(\tau^\sharp,y^\sharp) = y^\sharp + \tfrac{1}{2}$ for $y^\sharp$ large enough which establishes the nonnegativity. This implies that $\wt{h}_N(\tau,y) - y$ is nonnegative and decreasing. Thus, if
\[ |\cH(R) - \wt{h}_N(\tau,R)| = |(\cH(R) - R) - (\wt{h}_N(\tau,R) - R)| < \e/3, \]
then for $y \in (R,\infty)$, \eqref{eq:H_y>>0} implies that
\[ 0 \le \wt{h}_N(\tau,y) - y \le \wt{h}_N(\tau,R) - R \le 2\e/3 \]
so that
\[ |\cH(y) - \wt{h}_N(\tau,y)| \le |\cH(y) - y| + |\wt{h}_N(\tau,y) - y| < \e. \]
Therefore
\begin{align} \label{eq:y>>0_bound}
\PP\left( \sup_{y \in \R} |\cH(y) - \wt{h}_N(\tau,y)| \ge \e \right) \le \PP\left( \sup_{y \in (-\infty,R]} |\cH(y) - \wt{h}_N(\tau,y)| \ge \e/3 \right).
\end{align}
Combining~\eqref{eq:y<<0_bound} and~\eqref{eq:y>>0_bound} yields~\eqref{eq:uniform_convergence_inequality}, which completes the proof.
\end{proof}

\subsection{Gaussian free field convergence for the unshifted model} 
Let $\bar{h}_N$ be the centered height function of the $q^{\mvol}$ measure and $\eta$ be as in Definition~\ref{def:eta}.

\begin{proof}[Proof of {\Cref{thm:unshifted_gff_convergence_intro}}]
Set
\[ \bar{h}_N(\tau^\sharp,y^\sharp) := \hh(\tau^\sharp,y^\sharp) - \E \hh(\tau^\sharp,y^\sharp). \]
By \Cref{thm:covariance_asymptotics,thm:cumulant_asymptotics}, we know that the vector 
\[
\left( \frac{1}{2N} \sum_{y \in \frac{1}{2N} (\Z+\frac{1}{2})} \bar{h}_N(\floor{2N\tau_i},2Ny) t^{k_iy} \right)_{1 \le i \le n}
\]
from \eqref{eq:prelimit_vector} converges to a Gaussian vector in distribution. Therefore, it suffices to show that
\begin{align} \label{eq:covariance_gff}
\begin{split}
& \lim_{N\to\infty} \cov \left( \frac{1}{2N} \sum_{y\in\tfrac{1}{2N}\Z'} \bar{h}_N(\lfloor 2N\tau_1 \rfloor,2Ny) t^{2k_1y}, \frac{1}{2N} \sum_{y \in \tfrac{1}{2N}\Z'} \bar{h}_N(\lfloor 2N \tau_2 \rfloor,2N y)t^{2k_2y}\right) \\
& \quad \quad = \frac{1}{\pi} \int\limits_{\tfrac{\log 2}{\log t}}^{\infty}\int\limits_{\tfrac{\log 2}{\log t}}^{\infty} G(\eta(\tau_1,y_1),\eta(\tau_2,y_2))t^{2k_1y_1}t^{2k_2y_2}\, dy_1dy_2,
\end{split}
\end{align}
where $G(\eta_1,\eta_2)=- \frac{1}{2\pi} \log \left(
\frac{\Theta(\eta_1-\eta_2 \, | \, \omega)}
{\Theta(\eta_1+\overline\eta_2 \, | \, \omega)}
\right)$ is the Green's function, and $t = e^{2 \pi \bi \omega} \in (0,1)$.

We prove the statement for $\tau_1 \le \tau_2$ without loss of generality. By \Cref{thm:covariance_asymptotics}, for any real $c_1$ and $c_2$ satisfying $c_2 - \tfrac{|\log t|}{2\pi} < c_1 < c_2$,
\begin{align}\label{eq:covariance_limit}
\begin{split}
& \lim_{N\to\infty} \cov \left( \frac{1}{2N} \sum_{y\in\tfrac{1}{2N}\Z'} \bar{h}_N(\lfloor 2N\tau_1 \rfloor,2Ny) t^{2k_1y}, \frac{1}{2N} \sum_{y \in \tfrac{1}{2N}\Z'} \bar{h}_N(\lfloor 2N \tau_2 \rfloor,2N y)t^{2k_2y}\right) \\
& \quad = \frac{1}{4k_1k_2(2\pi\bi\log t)^2}
\int\limits_{-\tfrac12+\bi c_1}^{\tfrac12+\bi c_1}
\int\limits_{-\tfrac12+\bi c_2}^{\tfrac12+\bi c_2}
-\partial^2\log\Theta(\eta_1-\eta_2\,|\,\omega)
\prod_{i=1}^2(1-t^{\tau_i} e^{-2\pi\bi\eta_i})^{k_i}(1-t^{-\tau_i} e^{2\pi\bi\eta_i})^{k_i}\,d\eta_i
\end{split}
\end{align}
where $t = e^{2\pi\bi\omega}$. It suffices, by continuity, to establish the equality of the RHS of~\eqref{eq:covariance_limit} with the RHS of~\eqref{eq:covariance_gff} for $\tau_1 < \tau_2$.

Note that the RHS of~\eqref{eq:covariance_limit} can be rewritten as
\begin{align*}
&\frac{1}{4k_1k_2(2\pi\bi \log t)^2}
\left(
\int\limits_{0+\bi c_1}^{1/2+\bi c_1}
\int\limits_{0+\bi c_2}^{1/2+\bi c_2}
+
\int\limits_{0+\bi c_1}^{1/2+\bi c_1}
\int\limits_{-1/2+\bi c_2}^{0+\bi c_2} 
+
\int\limits_{-1/2+\bi c_1}^{0+\bi c_1}
\int\limits_{0+\bi c_2}^{1/2+\bi c_2}
+
\int\limits_{-1/2+\bi c_1}^{0+\bi c_1}
\int\limits_{-1/2+\bi c_2}^{0+\bi c_2}
\right)
\\
& \quad \quad \quad \quad \quad \quad \quad \quad \quad \quad \quad \quad \quad \quad \quad \quad \quad 
\left(
-\partial^2\log\Theta(\eta_1-\eta_2\, |\,\omega)
\prod_{i=1}^2(1-t^{\tau_i} e^{-2\pi\bi\eta_i})^{k_i}(1-t^{-\tau_i} e^{2\pi\bi\eta_i})^{k_i}\,d\eta_i
\right)\\
& = \sum_{\diamond_1,{\diamond_2}\in\{+,-\}}
\frac{1}{4k_1k_2(2\pi\log t)^2}
\int\limits_{0+\bi c_1}^{1/2+\bi c_1}
\int\limits_{0+\bi c_2}^{1/2+\bi c_2}
\partial^2\log\Theta(\eta_1^{\diamond_1}-\eta_2^{\diamond_2}\, |\,\omega)
\prod\limits_{i=1}^2(1-t^{\tau_i} e^{-2\pi\bi\eta_i^{\diamond_i}})^{k_i}(1-t^{-\tau_i} e^{2\pi\bi\eta_i^{\diamond_i}})^{k_i}\,d\eta_i,
\end{align*}
where $\eta^{+}_i=\eta_i$ and $\eta^{-}_i=-\overline\eta_i$.

By definition of $\eta$ and~\eqref{eq:zeta_defining}, the following holds
\[ -\frac{(1 - t^{-\tau} e^{2\pi\bi\eta(\tau,y)})^2}{t^{-\tau} e^{2\pi\bi\eta(\tau,y)}} = (1-t^{\tau}e^{-2\pi \bi \eta(\tau,y)})(1-t^{-\tau}e^{2\pi \bi \eta(\tau,y)})=t^{2y}.\]
Because $e^{2 \pi \bi(-\overline{\eta})} = \overline{e^{2 \pi \bi \eta}}$, the same holds with $\eta(\tau,y)$ replaced by $-\overline{\eta(\tau,y)}$. Therefore 
\[ -\frac{(1 - t^{-\tau} e^{2\pi\bi\eta(\tau_i,y_i)})^{2k_i}}{t^{-k_1 \tau_1} e^{2\pi\bi k_i\eta(\tau_i,y_i)}} = t^{2 k_iy_i}
\]
and similarly with $\eta_i(\tau_i,y_i)$ replaced by $-\overline{\eta_i(\tau_i,y_i)}$.

By \Cref{remark:eta}, we have $\eta$ maps $\{\tau\} \times (\tfrac{\log 2}{\log t},\infty)$ onto $\tfrac{1}{2}(0,1)+ \tfrac{\tau|\log t|}{2\pi} \bi$. Thus if we choose 
\[ c_i = \frac{\tau_i|\log t|}{2\pi}, \quad\quad i = 1,2, \]
where we note $c_2 - \tfrac{|\log t|}{2\pi} < c_1 < c_2$, then we may change variables $\eta_i$ to $y_i$ where $\eta_i = \eta(\tau_i,y_i)$ for $i = 1,2$ to obtain 
\begin{align*}
\sum_{\diamond_1,{\diamond_2}\in\{+,-\}}
\frac{(-1)^{1-\bbone_{\diamond_1=\diamond_2}} }{4k_1k_2(2\pi\log t)^2}\int\limits_{\tfrac{\log 2}{\log t}}^{\infty}\int\limits_{\tfrac{\log 2}{\log t}}^{\infty}
\partial^2\log\Theta(\eta^{\diamond_1}(\tau_1,y_1)-\eta^{\diamond_2}(\tau_2,y_2)\, |\,\omega)
\prod\limits_{i=1}^2
t^{2k_iy_i}
\frac{\partial\eta^{\diamond_1}(\tau_i,y_i)}{\partial y_i}
\,dy_i,
\end{align*}
where $\eta^{+}=\eta$ and $\eta^{-}=-\overline\eta$.
Finally, integrate by parts to get
\begin{align*}
& \frac{1}{(2\pi)^2}\left(
\sum_{\diamond_1,{\diamond_2}\in\{+,-\}}
(-1)^{\bbone_{\diamond_1=\diamond_2}}
\int\limits_{\tfrac{\log 2}{\log t}}^{\infty}\int\limits_{\tfrac{\log 2}{\log t}}^{\infty}
\log\Theta(\eta^{\diamond_1}(\tau_1,y_1)-\eta^{\diamond_2}(\tau_2,y_2)\, |\,\omega)
\prod\limits_{i=1}^2
t^{2k_iy_i}
\,dy_i\right) \\
& \quad \quad =
-\frac{1}{4\pi^2}
\int\limits_{\tfrac{\log 2}{\log t}}^{\infty}\int\limits_{\tfrac{\log 2}{\log t}}^{\infty}
\log \left(
\frac{\Theta(\eta(\tau_1,y_1)-\eta(\tau_2,y_2)\,|\,\omega)\Theta(-\overline\eta(\tau_1,y_1)+\overline\eta(\tau_2,y_2)\,|\,\omega)}
{\Theta(\eta(\tau_1,y_1)+\overline\eta(\tau_2,y_2)\,|\,\omega)\Theta(-\overline\eta(\tau_1,y_1)-\eta(\tau_2,y_2)\,|\,\omega)}
\right)
t^{2k_1y_1}t^{2k_2y_2}\, dy_1dy_2.
\end{align*}
Using the fact that $\theta_1(\overline{z};t) = \overline{\theta_1(z;t)}$, we have
\begin{align*}
\frac{\Theta(\eta(\tau_1,y_1)-\eta(\tau_2,y_2) \, | \, \omega)\Theta(-\overline\eta(\tau_1,y_1)+\overline\eta(\tau_2,y_2)\, | \, \omega)}
{\Theta(\eta(\tau_1,y_1)+\overline\eta(\tau_2,y_2)\, | \, \omega)\Theta(-\overline\eta(\tau_1,y_1)-\eta(\tau_2,y_2)\, | \, \omega)} &= \frac{\theta_1(\zeta_1/\zeta_2;t) \theta_1(\overline{\zeta}_1/\overline{\zeta}_2;t)}{\theta_1(\zeta_1/\overline{\zeta}_2;t) \theta_1(\overline{\zeta}_1/\zeta_2;t)} \\
&= \left| \frac{\theta_1(\zeta_1/\zeta_2;t)}{\theta_1(\zeta_1/\overline{\zeta}_2;t)} \right|^2 \\
&= \left| \frac{\Theta(\eta(\tau_1,y_1)-\eta(\tau_2,y_2)\, | \, \omega)}{\Theta(\eta(\tau_1,y_1)+\overline\eta(\tau_2,y_2)\, | \, \omega)} \right|^2,
\end{align*}
where $\zeta_i = e^{2\pi\bi \eta(\tau_i,y_i)}$ for $i = 1,2$. This completes the proof. 
\end{proof}

\subsection{Gaussian free field convergence for the shift-mixed model}

In this section, we deduce Gaussian free field convergence for the shift-mixed model from the corresponding statement for the unshifted model. This argument in no way relies on the fact that the shift distribution is a discrete Gaussian; in fact, we show a general result for discrete distributions on $\Z$ subject to some moment conditions, and then specialize.

\begin{proposition}\label{thm:shift-mixed_fluctuations}
Fix $0 < t < 1$. Let $S$ be a $\Z$-valued random variable with moment generating function defined in a neighborhood of $0$. Then for any $\tau_1,\ldots,\tau_n \in (0,1]$ and integers $k_1,\ldots,k_n > 0$, the random vector 
\begin{equation}\label{eq:shift_prelimit_vec}
    \left( \frac{1}{2N} \sum_{y \in \frac{1}{2N} \Z'} \left(\hh\left(\lfloor 2N\tau_i\rfloor,2Ny-S\right) - \E[\hh\left(\lfloor 2N\tau_i\rfloor,2Ny\right)]\right) t^{k_iy}\right)_{1 \leq i \leq n}
\end{equation}
converges in distribution as $N \to \infty$ to the random vector
\begin{equation}\label{eq:shift_postlimit_vec}
    \left( \int_{\frac{\log 2}{\log t}}^\infty \left( \frac{1}{\sqrt{\pi}}\Phi(\eta(\tau_1,y)) - S \cH'(y) \right) t^{k_1y} dy , \ldots, \int_{\frac{\log 2}{\log t}}^\infty \left( \frac{1}{\sqrt{\pi}} \Phi(\eta(\tau_n,y)) - S \cH'(y) \right) t^{k_ny} dy \right),
\end{equation} 
where $\Phi$ is the Gaussian free field on the cylinder with Dirichlet boundary conditions as in \Cref{def:GFF} and $\cH$ is as defined in \Cref{def:limit_shape}.
\end{proposition}
\begin{proof}
The proof proceeds by splitting each prelimit random variable in such a way that the joint cumulants can be easily computed, then appealing to results on the unshifted model to show convergence of joint cumulants to the desired limiting random variables.
First note that
\begin{align*}
t^{k_iy}\hh\left(\lfloor 2N\tau_i\rfloor,2Ny-S\right) &= t^{\frac{k_i}{2N}S} \left(t^{k_i\left(y-\frac{S}{2N}\right)} \hh\left(\lfloor 2N\tau_i\rfloor,2Ny - S\right)\right) \\
&= \left(t^{\frac{k_i}{2N}S}-1\right)\left(t^{k_i y'} \hh\left(\lfloor 2N\tau_i\rfloor,2Ny'\right)\right) + \left(t^{k_i y'} \hh\left(\lfloor 2N\tau_i\rfloor,2Ny'\right)\right)
\end{align*}
where $y'=y-\tfrac{S}{2N}$. Use this, together with reindexing sums by $y'$, to write 
\begin{align}
     \frac{1}{2N} \sum_{y \in \frac{1}{2N} \Z'} \left(\hh\left(\lfloor 2N\tau_i\rfloor,2Ny-S\right) - \E[\hh\left(\lfloor 2N\tau_i\rfloor,2Ny\right)]\right) t^{k_iy} = A_N^{(i)} + B_N^{(i)} + C_N^{(i)},
\end{align}
where 
\begin{align}
    A_N^{(i)} &= (t^{\frac{k_i}{2N}S} - 1) \frac{1}{2N}\sum_{y \in \frac{1}{2N}\Z'} t^{k_iy} \E[\hh(\lfloor 2N\tau_i\rfloor,2Ny)]) \\
    B_N^{(i)} &= \frac{1}{2N}\sum_{y \in \frac{1}{2N}\Z'} t^{k_iy} \left(\hh(\floor{2N\tau_i},2Ny)-\E[\hh(\lfloor 2N\tau_i\rfloor,2Ny)])\right)\\
    C_N^{(i)} &= (t^{\frac{k_i}{2N}S}-1)\frac{1}{2N}\sum_{y \in \frac{1}{2N}\Z'} t^{k_iy}\left(\hh(\floor{2N\tau_i},2Ny)-\E[\hh(\lfloor 2N\tau_i\rfloor,2Ny)])\right).
\end{align}
Note that because $S$ is independent of $\hh$, $A_N^{(i)}$ and $B_N^{(j)}$ are independent for any $i,j$. It is a basic fact that any joint cumulant with two independent arguments is $0$, hence 
\[\kappa_m\left(A_N^{(i_1)}+B_N^{(i_1)}+C_N^{(i_1)},\ldots,A_N^{(i_m)}+B_N^{(i_m)}+C_N^{(i_m)}\right) 
= \kappa_m\left(A_N^{(i_1)},\ldots,A_N^{(i_m)}\right) + \kappa_m\left(B_N^{(i_1)},\ldots,B_N^{(i_m)}\right) + (\ldots) \]
where the $(\ldots)$ term consists of a finite (independent of $N$) number of joint cumulants in random variables $A_N^{(i)},B_N^{(i)},C_N^{(i)}$ which contain at least one $C_N^{(i)}$. We wish to show these latter terms are small, and claim that for any $m$ and any $i_1,\ldots, i_m \in \{1,\ldots,n\}$ (not required to be distinct),
\begin{equation}\label{eq:split_cumulants}
    \kappa_m\left(A_N^{(i_1)}+B_N^{(i_1)}+C_N^{(i_1)},\ldots,A_N^{(i_m)}+B_N^{(i_m)}+C_N^{(i_m)}\right) - \kappa_m\left(A_N^{(i_1)},\ldots,A_N^{(i_m)}\right) - \kappa_m\left(B_N^{(i_1)},\ldots,B_N^{(i_m)}\right) \to 0
\end{equation}
as $N \to \infty$. To show this, and to compute the cumulants in the $A$'s and $B$'s later, we rely on the following.

\textbf{Claim 1:} The joint moments (and hence joint cumulants) of $2N(t^{\frac{k_i}{2N}S}-1)$, $i=1,\ldots,n$ converge to those of $(k_i \log t) S$. 

\textbf{Claim 2:} The joint moments and cumulants of $\frac{1}{2N}\sum_{y \in \frac{1}{2N}\Z'} t^{k_iy}\left(\hh(\floor{2N\tau_i},2Ny)-\E[\hh(\lfloor 2N\tau_i\rfloor,2Ny)])\right)$ converge to those of $\int_{\frac{\log 2}{\log t}}^\infty \Phi(\eta(\tau_{i},y))t^{k_iy} dy$.

Claim 1 follows from Taylor expansion and the assumption that the moment generating function of $S$ is defined near $0$. Claim 2 follows since both first moments are clearly $0$, second cumulants $\kappa_2$ converge by \Cref{thm:covariance_asymptotics},  and higher cumulants $\kappa_\ell, \ell \geq 3$ converge by \Cref{thm:cumulant_asymptotics}.

Hence $A_N^{(i)}$ is a product of the $O(1/N)$ random variable $t^{\frac{k_i}{2N}S} - 1$ with the constant 
\[
\frac{1}{2N}\sum_{y \in \frac{1}{2N}\Z'} t^{k_iy} \E[\hh(\lfloor 2N\tau_i\rfloor,2Ny)]) = \E\left[ \frac{1}{2N}\sum_{y \in \frac{1}{2N}\Z'} t^{k_iy} \hh(\lfloor 2N\tau_i\rfloor,2Ny)) \right],
\]
which is $O(N)$ by \Cref{thm:moment_asymptotics}, so $A_N^{(i)}$ is $O(1)$. The random variable $B_N^{(i)}$ is $O(1)$ by Claim 2, while $C_N^{(i)}$ is a product of an $O(1/N)$ random variable and an $O(1)$ random variable by Claims 1 and 2. Therefore any joint moment featuring $A_N, B_N$, and at least one $C_N$ terms is $O(1/N)$. This proves~\eqref{eq:split_cumulants}. The limits of the cumulants $ \kappa_m(B_N^{(i_1)},\ldots,B_N^{(i_m)})$ are immediate from Claim 2, so we compute $\kappa_m(A_N^{(i_1)},\ldots,A_N^{(i_m)})$.

Since
\[
\kappa_m\left(A_N^{(i_1)},\ldots,A_N^{(i_m)}\right) = 
\left(\prod_{j=1}^m \E\left[\frac{1}{2N}\sum_{y \in \frac{1}{2N}\Z'} t^{k_{i_j}y} \frac{\hh(\tau_{i_j}^\sharp,2Ny)}{2N} \right] \right) \kappa_m\left(2N\Bigg(t^{\frac{k_{i_1}}{2N}S}-1\Bigg),\ldots,2N\Bigg(t^{\frac{k_{i_m}}{2N}S}-1\Bigg)\right),
\]
it follows by \Cref{thm:limit_shape} applied to the prefactor and Claim 1 applied to the cumulant that  
\begin{align}\label{eq:A_cumulant_convergence}
\begin{split}
    \kappa_m\left(A_N^{(i_1)},\ldots,A_N^{(i_m)}\right) &\to \left(\prod_{j=1}^m  (k_i \log t) \int\limits_{\frac{\log 2}{\log t}}^\infty  \cH(y)t^{k_iy} dy\right) \kappa_m(S) \\
    &=\kappa_m\left(-S\int\limits_{\frac{\log 2}{\log t}}^\infty \cH'(y) t^{k_{i_1} y} dy,\ldots,-S\int\limits_{\frac{\log 2}{\log t}}^\infty \cH'(y)t^{k_{i_m} y}  dy \right)
\end{split}
\end{align}
where we use integration by parts for the equality, and the boundary terms vanish by the explicit description of $\cH$ in \Cref{thm:limit_shape}.

Combining~\eqref{eq:split_cumulants} with these two computations yields 
\begin{align*}
    &\kappa_m\left(A_N^{(i_1)}+B_N^{(i_1)}+C_N^{(i_1)},\ldots,A_N^{(i_m)}+B_N^{(i_m)}+C_N^{(i_m)}\right) \to \\
    &\kappa_m\left(\frac{1}{\sqrt{\pi}} \int\limits_{\frac{\log 2}{\log t}}^\infty \Phi(\eta(\tau_{i_1},y))t^{k_{i_1}}dy,\ldots, \frac{1}{\sqrt{\pi}} \int\limits_{\frac{\log 2}{\log t}}^\infty \Phi(\eta(\tau_{i_m},y)) t^{k_{i_m}y}dy \right)\\
    &+ \kappa_m\left(-S\int\limits_{\frac{\log 2}{\log t}}^\infty  \cH'(y)t^{k_{i_1} y} dy,\ldots,-S\int\limits_{\frac{\log 2}{\log t}}^\infty  \cH'(y)t^{k_{i_m} y} dy \right) \\
    &= \kappa_m\left( \int\limits_{\frac{\log 2}{\log t}}^\infty \left( \frac{1}{\sqrt{\pi}}\Phi(\eta(\tau_{i_1},y))-S \cH'(y)\right) t^{k_{i_1}y}dy, \ldots, \int\limits_{\frac{\log 2}{\log t}}^\infty \left( \frac{1}{\sqrt{\pi}} \Phi(\eta(\tau_{i_m},y)) - S \cH'(y) \right)t^{k_{i_m} y} dy \right)
\end{align*}
so the joint cumulants of the random vector~\eqref{eq:shift_prelimit_vec} converge to those of~\eqref{eq:shift_postlimit_vec}. By the moment-cumulant relations all moments converge as well, and it is also clear that all joint moments of $A_N^{(i)}+B_N^{(i)}+C_N^{(i)}$ exist for finite $N$. Since the moment generating function of $S$ exists in a neighborhood of $0$, $S$ is determined by its moments, and since Gaussians are also determined by their moments, the random vector
\[
\left(\int\limits_{\frac{\log 2}{\log t}}^\infty \left( \frac{1}{\sqrt{\pi}} \Phi(\eta(\tau_{i_1},y)) -S \cH'(y) \right) t^{k_{i_1} y} dy , \ldots, \int\limits_{\frac{\log 2}{\log t}}^\infty \left( \frac{1}{\sqrt{\pi}} \Phi(\eta(\tau_{i_m},y)) -S \cH'(y) \right) t^{k_{i_m} y} dy \right)
\]
is uniquely determined by its (joint) moments. Hence convergence of moments implies weak convergence, completing the proof.
\end{proof}

\begin{proof}[Proof of {\Cref{thm:shifted_gff_intro}}]
By definition of the shift-mixed periodic Schur process we have equality in distribution
\[
\hs(\floor{2N\tau},2Ny) = \hh(\floor{2N\tau}, 2Ny - S)
\]
where $\hs$ is the random height function of the shift-mixed $q^{\mvol}$ measure with parameter $u$, $\hh$ is the random height function of the $q^{\mvol}$ measure, and $S$ is the discrete Gaussian given by 
\[
\Pr(S=k) = \frac{u^k t^{k^2/2}}{\theta_3(u;t)}
\]
with $t=q^N$. The moment generating function of $S$ is 
\[
f_S(a) = \E[e^{aS}] = \frac{\theta_3(ue^a;t)}{\theta_3(u;t)},
\]
which is clearly defined for $a$ in a neighborhood of $0$, so \Cref{thm:shifted_gff_intro} follows from \Cref{thm:shift-mixed_fluctuations}.
\end{proof}

\section{Matching with conjectures on height fluctuations}\label{sec:conf_str}

In this section we discuss the relation of our results to conjectures in the literature concerning fluctuations of height functions in tiling models. In \Cref{subsec:KO} we check that the conformal structure of the Gaussian free field in \Cref{thm:unshifted_gff_convergence_intro} is the one predicted by the Kenyon-Okounkov conjecture. In \Cref{sec:hol_hex} we show that the discrete component of the fluctuations in \Cref{thm:shifted_gff_intro} agrees with a conjecture of~\cite{gorin_2021} regarding height fluctuations in multiply connected planar domains.

\subsection{Kenyon--Okounkov conjecture}\label{subsec:KO}

In~\cite{KO07}, Kenyon and Okounkov conjecture that the height fluctuations of biperiodic bipartite dimer models on a simply connected domain with volume constraints are given by the Gaussian free field on the liquid region in some conformal structure which they also describe. We now give a more precise statement of their conjecture in the setting of the~$q^{\mvol}$ measure, based on~\cite[Sections 9.3 and 11.3]{gorin_2021}.

We note that the height function we defined for the cylinder is uniquely determined by (A) local relations: all increments of the height function are determined by the presence/absence of lozenges at those locations and (B) boundary conditions: the height function $h(\tau,y)$ is $0$ for all sufficiently negative $y$. Given a general tileable simply connected domain on the triangular lattice, one may define a height function similarly by fixing its value at a point and mandating that its local increments are the same as ours\footnote{See e.g.~\cite[Section 1]{gorin_2021} for a more extended discussion, although note that the height function there differs from ours by an affine transformation.}. 
Suppose $D_L$ is a sequence (in $L$) of simply connected tileable domains on the triangular lattice such that its boundary approximates that of some domain $D$ in $\R^2$ with continuous boundary, $\tfrac{1}{L} \partial D_L \to \partial D$. We can consider lozenge tilings in $D_L$ and the measure $q^{\mvol}$ on $D_L$, where $q = e^{c/L}$ depends on $L$ and a fixed parameter\footnote{If one considers tilings of an infinite domain, one must take $c \in \R_{<0}$.} $c \in \R$. We will use $\tau$ and $y$ for the horizontal and vertical coordinates on $\R^2$, for consistency with our previous coordinates on the cylinder. However, the conformal structure of~\cite{KO07} is in terms of coordinates where the axes are parallel to the sides of the $\rloz$ lozenge, whereas our initial choice of coordinates reflected the rotational invariance of our model. Hence we introduce new coordinates $\wh{\tau},\wh{y}$, for which the point $(\wh{\tau},\wh{y})$ corresponds to $(\tau,y + \tfrac{1}{2}\tau)$. In the case of the cylinder the $\wh{\tau},\wh{y}$ coordinate directions are shown in \Cref{fig:KO_discussion}.

It is known~\cite{KOS06} that the height function $h_L$ of such a random tiling has a limit shape, which can be described in local proportions $p_{\hloz}, p_{\lloz}, p_{\rloz}$ of lozenges. More precisely, by~\cite[Theorem 1]{KO07} there exists a function $(\wh{\tau},\wh{y}) \mapsto z(\wh{\tau},\wh{y})$, called the \emph{complex slope}, on the liquid region (which is the subset of $D$ for which $p_{\hloz}, p_{\lloz}, p_{\rloz}$ are all positive as in \Cref{def:liquid}) such that
\begin{align}\label{eq:slopes_and_z}
p_{\hloz} = \frac{\arg z}{\pi}, \quad \quad p_{\rloz} = -\frac{\arg(1 - z)}{\pi}, \quad \quad p_{\lloz} = 1 - p_{\rloz} - p_{\hloz}.
\end{align}
Furthermore,~\cite[Corollary 1]{KO07} implies that there exists an analytic function $Q$ such that
\begin{equation}\label{eq:Q}
    Q(e^{-c\wh{\tau}} z, e^{-c\wh{y}}(1 - z)) = 0.
\end{equation}
By~\cite[Section 1.5]{KO07}, $z,1 - z$ satisfy a differential equation related to the complex Burgers equation, and the analytic relation above follows from solving this equation by complex characteristics. We now state the conjecture of Kenyon and Okounkov~\cite[p15]{KO07}, which we formulate based on the presentation in~\cite[Conjecture 11.1]{gorin_2021}.

\begin{conjecture}\label{conj:KO}
The fluctuations of the height function on the liquid region converge as $L \to \infty$ to the Gaussian free field on the upper half plane with $0$-Dirichlet boundary conditions,
\[ \sqrt{\pi} \left( h_L(\tfrac{\wh{\tau}}{L},\tfrac{\wh{y}}{L}) - \E[h_L(\tfrac{\wh{\tau}}{L},\tfrac{\wh{y}}{L})] \right) \to \mathrm{GFF} \circ \zeta, \]
where the conformal structure of the liquid region is induced by $\zeta = e^{-c\wh{\tau}} z$ and $z$ is the complex slope.
\end{conjecture}

\begin{remark}\label{rmk:us_vs_vadim}
\Cref{conj:KO} has been translated to match our notation, so we give a dictionary between it and the original version~\cite[Conjecture 11.1]{gorin_2021}. First, the coordinate directions which we denote by $\wh{\tau},\wh{y}$ are denoted $x,y$ respectively in~\cite{gorin_2021}. The height function denoted $H$ in~\cite[Conjecture 11.1]{gorin_2021}, is related to our height function $h$ (up to choice of constant shifts, which do not affect fluctuations) by 
\[h(\wh{\tau},\wh{y}) = \wh{y} - H(\wh{\tau},\wh{y}). \]
The complex slope we denote by $z$ in \Cref{conj:KO} is actually $1-\bar{z}$ in the notation of~\cite{gorin_2021}. Hence the conformal structure given by our $\zeta$ above corresponds to the conformal structure given by $e^{-c x}(1-\bar{z})$ in the notation of~\cite{gorin_2021}. Note that the Gaussian free field in the latter conformal structure is the same as the one in the conjugate conformal structure $e^{-c x}(1-z)$ (still in the notation of~\cite{gorin_2021}), because the Laplacian commutes with complex conjugation and we are viewing both as distributions on the space of real-valued test functions. It follows from~\cite[Theorem 10.1]{gorin_2021} and the implicit function theorem that this conformal structure is the same as the one given by $e^{-c y}z$, as is also noted in~\cite[Section 2.3]{KO07}. This latter conformal structure is the one used in~\cite[Conjecture 11.1]{gorin_2021}.
\end{remark}

\begin{figure}[ht]
    \centering
    \includegraphics[scale=0.4]{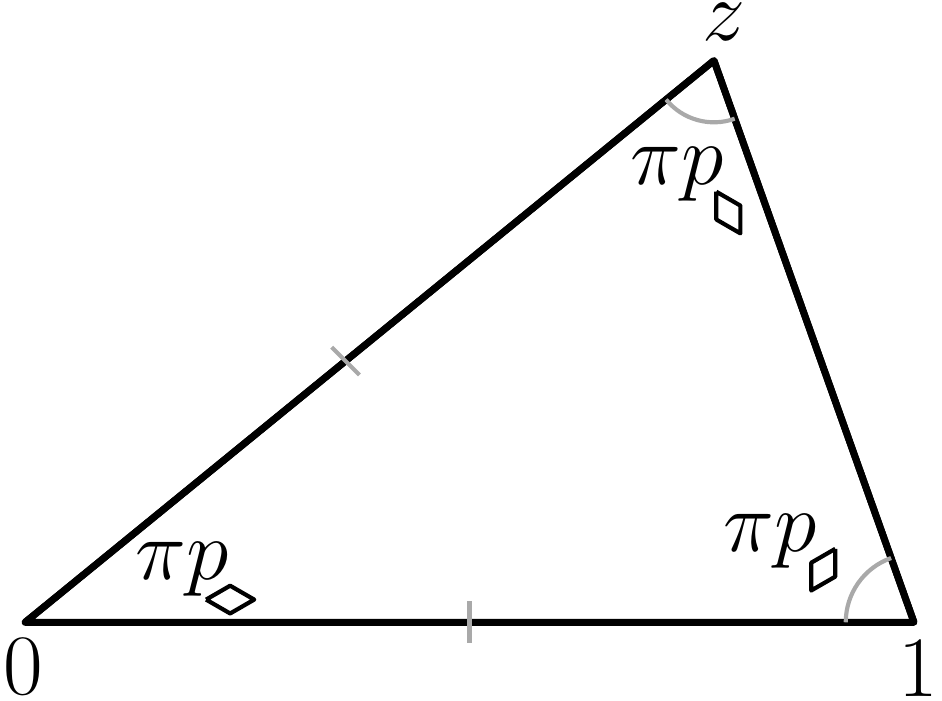}
    $\quad\quad\quad\quad\quad\quad$
    \includegraphics[scale=0.4]{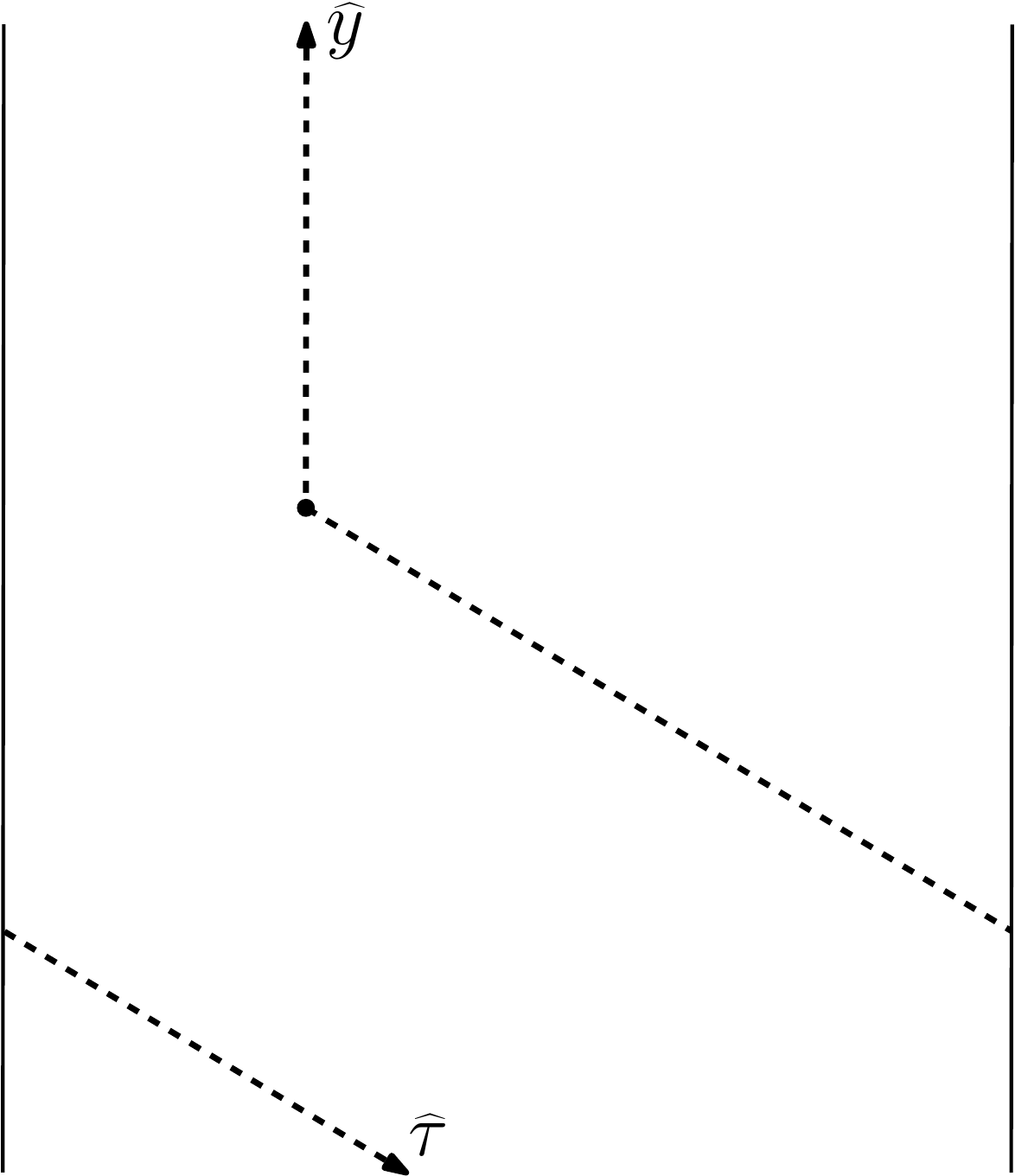}
    \caption{The isosceles triangle formed by $0,1,z$ where $z$ is the complex slope for the limiting unshifted $q^{\mvol}$ on the infinite cylinder (left); 
the $(\wh{\tau},\wh{y})$-coordinate axes (right).
} \label{fig:KO_discussion}
\end{figure}

The conjecture translates straightforwardly to our setting of the unshifted $q^{\mvol}$ measure on the infinite cylinder, and our main results verify it in this setting. We translate our results here to make the comparison with \Cref{conj:KO} immediate. In our setup the system size $L = 2N$ and the height function is indexed by $N$ instead of $L$, so $\hh$ corresponds to the system of size $2N$. Note also that our $q = t^{\frac{1}{2N}}$, therefore $c = \log t$. To verify \Cref{conj:KO} for the cylinder, it is enough to check that:

\begin{enumerate}[(i)]
    \item There is an explicitly given complex slope $z$, determined by~\eqref{eq:slopes_and_z}, which satisfies an equation of the form~\eqref{eq:Q} for some analytic function $Q$.
    
    \item The fluctuations of the height function are given by the $\zeta$-pullback of the Gaussian free field $\Phi$ on the cylinder with $0$-Dirichlet boundary conditions as in \Cref{thm:unshifted_gff_convergence_intro}, 
    \[ \sqrt{\pi} \left( \hh\left(\tfrac{\wh{\tau}}{L}, \tfrac{\wh{y}}{L}\right) - \E\left[\hh\left(\tfrac{\wh{\tau}}{L}, \tfrac{\wh{y}}{L}\right)\right] \right) \to \Phi \circ \zeta \]
    where $\zeta = e^{-c\wh{\tau}}z$.
\end{enumerate}

We first obtain the complex slope. \Cref{thm:complex_coordinate} implies
\[ \cH'(y) = \frac{\pi - \arg z}{\pi} = 1 - p_{\hloz} \]
where $z:= z(\tau,y)$ satisfies
\begin{align} \label{eq:z_equation}
(1 - z^{-1})(1 - z) = t^{2y}.
\end{align}
Since $p_{\rloz} = p_{\lloz}$ for the $q^{\mvol}$ measure on the cylinder, and it follows from~\eqref{eq:z_equation} that $|z| = 1$, the triangle formed by $0,1,z$ is isosceles with the distinguished angle at the vertex $0$, see \Cref{fig:KO_discussion}. Thus 
\[ p_{\hloz} = \frac{\arg z}{\pi}, \quad \quad p_{\lloz} = p_{\rloz} = -\frac{\arg(1 - z)}{\pi}. \]

Given these local proportions and the complex slope $z$, we may check that the prescribed conformal structure of \Cref{conj:KO} matches the conformal structure in our fluctuations result \Cref{thm:unshifted_gff_convergence_intro}. We must change coordinates to $(\wh{\tau},\wh{y})$ by
\begin{align} \label{eq:complex_coordinate_change}
(\tau,y) \mapsto (\wh{\tau},\wh{y}) = (\tau,y + \tfrac{1}{2}\tau)
\end{align}
as above. Recall from \Cref{fig:shift_mixed} that the curves corresponding to a fixed $y^\sharp$-coordinate in the $(\tau^\sharp,y^\sharp)$ coordinate system are jagged paths. To obtain prelimit coordinates which will limit to $(\wh{\tau},\wh{y})$, we must transform the coordinates based on parity:
\[ (\tau^\sharp,y^\sharp) \mapsto (\wh{\tau}^\sharp, \wh{y}^\sharp) = (\tau^\sharp,y^\sharp + \floor{ \tfrac{1}{2}\tau^\sharp}). \]
Taking the $L \to \infty$ limit and rescaling the domain $(\tau,y) = \tfrac{1}{L}(\tau^\sharp,y^\sharp)$ gives us the coordinate change~\eqref{eq:complex_coordinate_change}.

In terms of these new coordinates on the liquid region, we can describe our map $\eta$ in terms of $(\wh{\tau},\wh{y})$ as follows: let $\zeta$ be the solution to
\begin{align} \label{eq:zeta_equation}
(1 - t^{-\wh{\tau}} \zeta^{-1}) (1 - t^{\wh{\tau}} \zeta) = t^{2\wh{y} - \wh{\tau}},
\end{align}
which is just~\eqref{eq:zeta_defining} in the new coordinates, and note that $\eta = \tfrac{1}{2\pi\bi} \log t^{2\wh{\tau}} \zeta$ is just our original $\eta$ (\Cref{def:eta}) in these new coordinates. 
Equation~\eqref{eq:z_equation} implies that
\[ \zeta = t^{\wh{\tau}}z. \]
Moreover, note that~\eqref{eq:zeta_equation} can be rewritten as
\[ Q\left(t^{-\wh{\tau}}z,t^{-\wh{y}} (1 - z)\right) := \left( t^{-\wh{y}} (1 - z) \right)^2 + t^{-\wh{\tau}}z = 0 \]
which verifies (i).

Therefore, \Cref{thm:unshifted_gff_convergence_intro} becomes the statement that
\[ \hh(\lfloor 2N \tau \rfloor,y) - \E \hh(\lfloor 2N \tau \rfloor,y) \]
converges to the pullback of the Gaussian free field on the cylinder to our liquid region by the map $\eta$. The covariance kernel is given by
\[ \cK((\wh{\tau}_1,\wh{y}_1),(\wh{\tau}_2,\wh{y}_2)) = -\frac{1}{2\pi} \log \left| \frac{\Theta(\eta_1 - \eta_2|\omega)}{\Theta(\eta_1 + \bar{\eta}_2|\omega)} \right| = -\frac{1}{2\pi} \log \left| \frac{\theta_1(\zeta_1/\zeta_2;t)}{\theta_1(\zeta_1/\overline{\zeta}_2;t)} \right| \]
where $\zeta_i = \zeta(\wh{\tau}_i,\wh{y}_i)$ and $\eta_i = \eta(\wh{\tau}_i,\wh{y}_i)$ for $i = 1,2$. We note that the complex coordinate $\eta$ in which we have stated our results is not the same as the coordinate $\zeta$ in the Kenyon-Okounkov conjecture, though they are clearly conformally equivalent. This shows (ii), verifying the Kenyon-Okounkov conjecture in our setting.

\subsection{Hole height fluctuations for multiply connected domains}\label{sec:hol_hex}

As mentioned in the Introduction, there is also a general conjecture concerning the height fluctuations of macroscopic holes in multiply connected planar domains. In this subsection we state the conjecture and discuss how our results verify its natural extension to the cylinder.

There is a fruitful analogy between cylindric partitions, which are lozenge tilings of the cylinder, and lozenge tilings of multiply connected planar domains. The paradigmatic example of the latter is the \emph{holey hexagon}, shown in \Cref{fig:hol_hex}.
To define the height function on planar domains, the height at the outer boundary is fixed as usual, but the height along the inner boundary may vary with the tiling as shown. In this example, the height function is constant along the boundary of the hole, but in general for holes not of shape $\hloz$ the height function may vary along the perimeter of the hole. However, the increments of the height function along the boundary of the hole are independent of the tiling, so one may speak of the `height of a hole' by fixing a point on it and considering the height at that point. Since for height fluctuations the choice of this point does not matter, we will abuse terminology in referring to the height of a hole for general multiply connected domains below.

\begin{figure}[ht]
    \centering
    \includegraphics[scale=0.2]{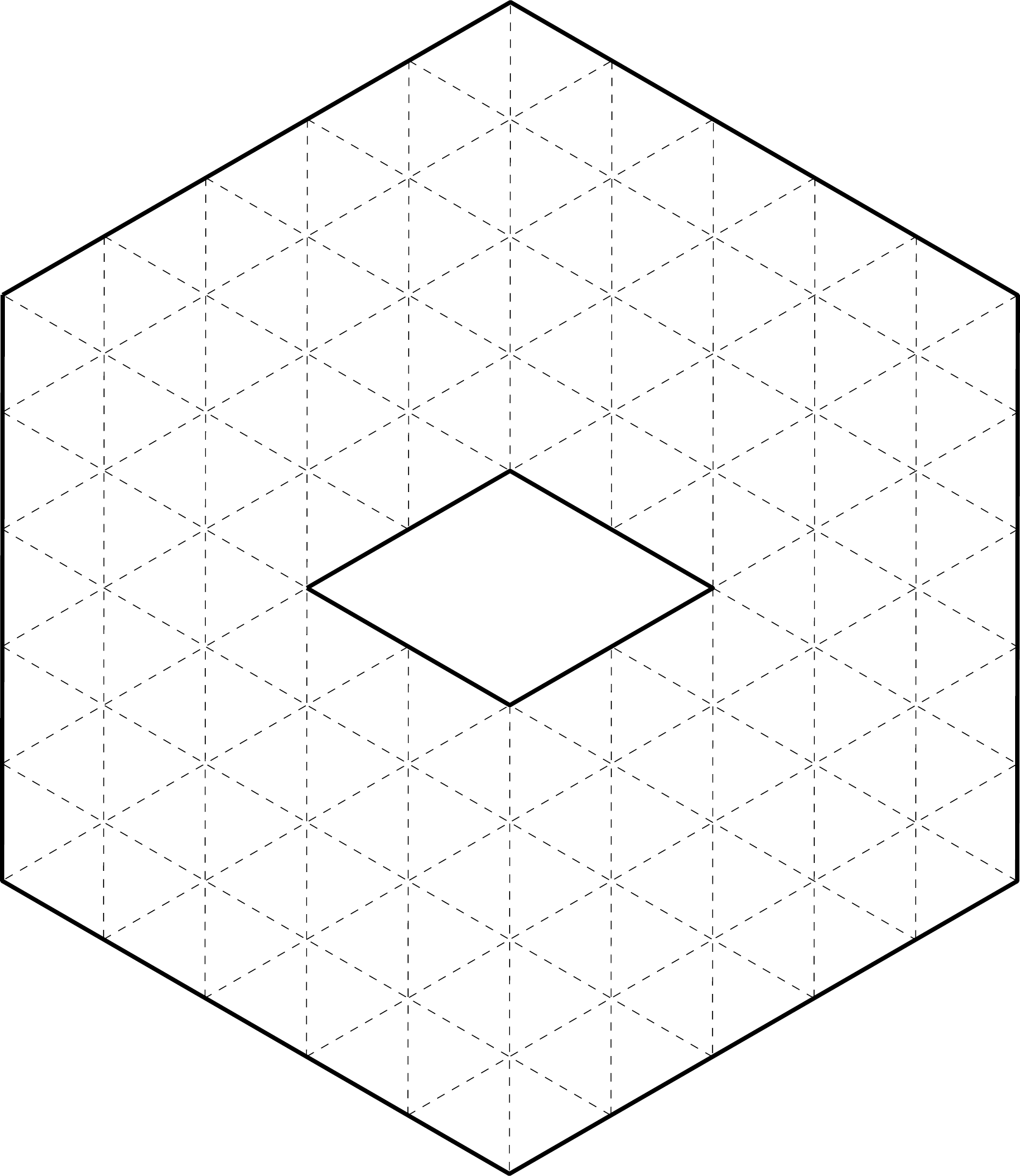}
    $\quad\quad\quad\quad\quad\quad$
    \includegraphics[scale=0.2]{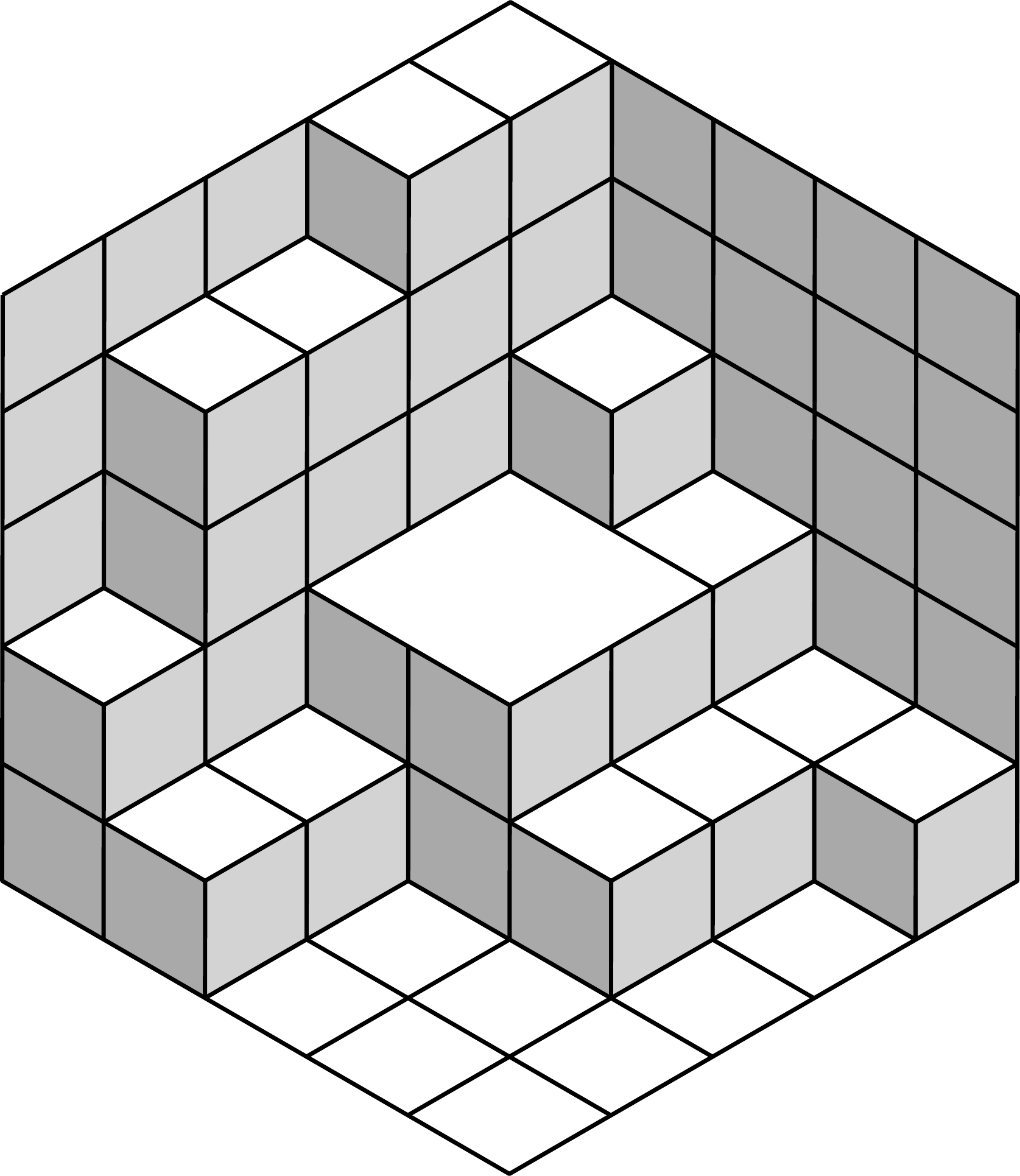}
    $\quad\quad\quad\quad\quad\quad$
     \includegraphics[scale=0.2]{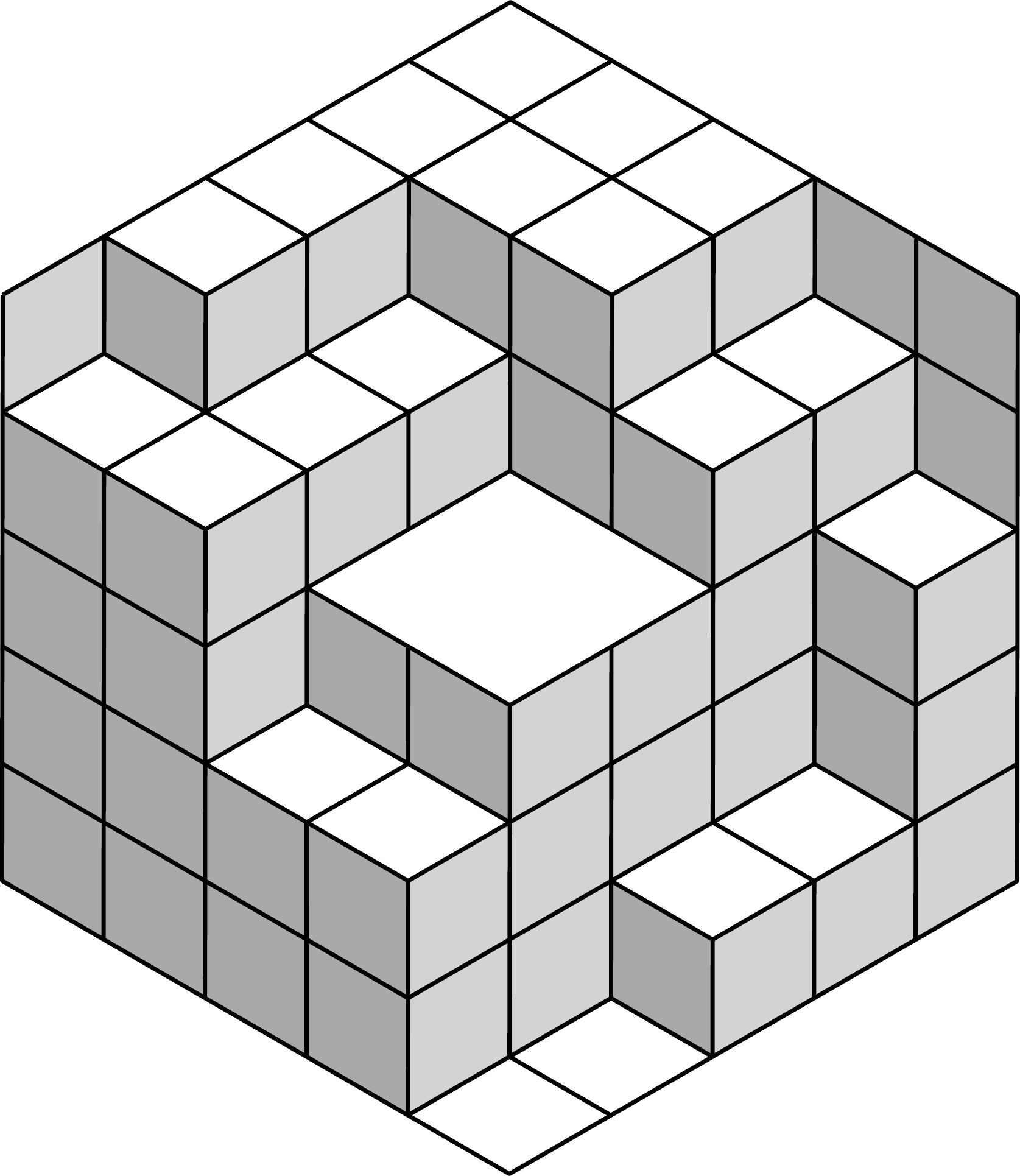}
    \caption{The holey hexagon domain (left). If the height at the bottom of the domain is taken to be $0$, then the height of the hole for the two tilings (center) and (right) are~$2$ and~$3$ respectively.}
    \label{fig:hol_hex}
\end{figure}

The fact that the height of the hole may vary leads to two reasonable notions of random tilings on such a domain, cf.~\cite[Section 24]{gorin_2021}: one may either (A) consider a random tiling among all tilings with probabilities given by some dimer weights, for example the uniform or $q^{\mvol}$ weights, or (B) fix some $H_0 \in \Z$ and condition the previous probability measure to live on those tilings for which the hole has height $H_0$.

In (A), the measure is a dimer model, and the positions of lozenges have a determinantal structure~\cite{Ken01}, while fixing the height at the hole as in (B) destroys this structure. Fixing the height of the hole in this setting is exactly analogous to fixing the shift $S$ of a cylindric partition, and as we saw the unshifted periodic Schur process similarly loses the determinantal structure of the shift-mixed one. This suggests to compare
\begin{align*}
\text{unshifted cylindric partitions} \quad \quad &\leftrightarrow \quad \quad \text{tilings of hexagon with hole height fixed}\\
\text{shift-mixed cylindric partitions} \quad \quad &\leftrightarrow \quad \quad \text{tilings of hexagon with hole height unrestricted}.
\end{align*}
This analogy also holds at the level of the limit objects. The height fluctuations of the holey hexagon with hole height fixed were shown to converge to the Gaussian free field in~\cite{BuG19}, analogous to our \Cref{thm:unshifted_gff_convergence_intro} for the unshifted $q^{\mvol}$ measure.

For tilings with the hole height unrestricted, which in the above analogy correspond to the shift-mixed model, there is a precise conjecture concerning the height of the hole of a random tiling. The setup is similar to the previous subsection: suppose $D_L$ is a sequence (in $L$) of tileable domains on the triangular lattice such that $\tfrac{1}{L} \partial D_L \to \partial D$ where $D$ is some multiply connected domain in $\R^2$ with continuous boundary. We can consider lozenge tilings in $D_L$ and define the uniform measure and the measure $q^{\mvol}$ on $D_L$, where $q = e^{c/L}$ depends on $L$ and a fixed parameter $c \in \R$. The following conjecture was originally stated for uniform tilings, but generalizes readily to $q^{\mvol}$.

\begin{conjecture}[{\cite[Conjecture 24.1]{gorin_2021}}]\label{conj:hole}
Let $D$ be a bounded domain in $\R^2$ with a single hole, $D_L$ be as above, and $H_{\operatorname{hole}}^{(L)}$
be the height of the hole of a $q^{\mvol}$ random tiling of $D_L$. Then there is a sequence of integer shifts $\mu_L$ such that $H_{\operatorname{hole}}^{(L)} - \mu_L$ converges in distribution to a discrete Gaussian $\cN_{\operatorname{discrete}}(m,C)$ (see \Cref{def:disc_gauss}), where $m \in [0,1)$ is some constant and $C$ is given by the Dirichlet energy
\[
C = \frac{\pi}{2} \int_{\zeta(\sL)}  \|\nabla g\|^2 \, dx \, dy.
\]
Here $\sL$ is the liquid region, $\zeta: \sL \to \C$ is the complex coordinate of the GFF as in \Cref{conj:KO}, and $g$ is the unique harmonic function on $\zeta(\sL)$ which is $0$ on the outer boundary and $1$ on the inner boundary. 
\end{conjecture}

A heuristic argument for this conjecture is given in~\cite[Section 24]{gorin_2021}, a proof for certain domains is expected to appear in~\cite{borot2021unpublished}, and a generalization of the conjecture to domains with multiple holes is given in~\cite[Conjecture 24.2]{gorin_2021}. 

In our case, the liquid region is the cylinder $(0,\tfrac{1}{2}) \times \R/\tfrac{|\log t|}{2\pi}\Z$. Identifying it with its fundamental domain, the function $g(x,y) = 2x$ is the unique harmonic function which is $0$ on the left boundary and $1$ on the other. The Dirichlet energy of this function is
\begin{equation}\label{eq:dirichlet_comp}
\frac{\pi}{2} \int_{(0,\tfrac{1}{2}) \times \R/\tfrac{|\log t|}{2\pi}\Z} \|\nabla g\|^2  \, dx \, dy =  \frac{\pi}{2} \int_{(0,\tfrac{1}{2}) \times \R/\tfrac{|\log t|}{2\pi}\Z} 4  \, dx \, dy  = 
\frac{|\log t|}{2},
\end{equation}
which is the variance parameter in the discrete Gaussian shift~\eqref{eq:specific_disc_gaus} appearing in \Cref{thm:shifted_gff_intro}. We recall from the previous subsection that the analogue of the Kenyon-Okounkov conformal structure in our setting is actually given by $\zeta$ rather than $\eta$ (as defined in \Cref{sec:main_proofs}), but since $\eta$ is the composition of $\zeta$ with a conformal map, the Dirichlet energy~\eqref{eq:dirichlet_comp} is the same as the Dirichlet energy in terms of $\zeta$ which is predicted in \Cref{conj:hole}.

Recall that for the shift-mixed periodic Schur process, the shift $S$ of the corresponding cylindric partition may be identified with an asymptotic shift of the height function as $h(\tau^\sharp,y^\sharp) = y^\sharp + S + \tfrac{1}{2}$ for all sufficiently large positive~$y$. Hence $S$ may be viewed the `height of the hole at infinity' for a tiling of the cylinder. The above calculation shows that under this interpretation, the shift-mixed $q^{\mvol}$ measure on cylindric partitions satisfies the appropriate generalization of \Cref{conj:hole}. It is interesting that in dimer models on planar domains with holes, the discrete Gaussian arises in the limit, while on the cylinder the discrete Gaussian shift is present before the limit, cf. \Cref{Sec:Dimers}.

\bibliographystyle{alpha_abbrvsort}
\bibliography{mybib}

\end{document}